\documentclass[11pt,reqno]{amsart}
\usepackage[utf8]{inputenc}
\usepackage[T1]{fontenc}
\usepackage{amsmath,amssymb,amsthm,mathtools}
\usepackage[margin=2.3cm]{geometry}
\usepackage{longtable,booktabs,array}
\usepackage{enumitem}
\usepackage[colorlinks=true,linkcolor=blue,citecolor=blue,urlcolor=blue]{hyperref}

\numberwithin{equation}{section}
\newtheorem{theorem}{Theorem}[section]
\newtheorem{proposition}[theorem]{Proposition}
\newtheorem{lemma}[theorem]{Lemma}
\newtheorem{corollary}[theorem]{Corollary}
\newtheorem*{theoremA}{Theorem A}
\theoremstyle{definition}
\newtheorem{definition}[theorem]{Definition}

\theoremstyle{remark}
\newtheorem{remark}[theorem]{Remark}

\newcommand{\C}{\mathbb{C}}
\newcommand{\Ann}{\operatorname{Ann}}
\newcommand{\Aut}{\operatorname{Aut}}
\newcommand{\Orb}{\operatorname{Orb}}
\newcommand{\Hom}{\operatorname{Hom}}
\newcommand{\rank}{\operatorname{rank}}
\newcommand{\Z}{\mathrm{Z}^{2}}
\newcommand{\B}{\mathrm{B}^{2}}
\newcommand{\HH}{\mathrm{H}^{2}}
\newcommand{\T}{\mathbf{T}}
\newcommand{\E}{\mathbf{E}}
\newcommand{\RR}[2]{\mathcal{R}^{#1}_{#2}}
\newcommand{\M}[1]{\mathcal{M}_{#1}}
\newcommand{\A}[1]{\mathcal{A}_{#1}}
\newcommand{\D}[1]{\Delta_{#1}}
\newcommand{\N}[1]{\nabla_{#1}}
\newcommand{\sspan}[1]{\langle #1\rangle}

\allowdisplaybreaks
\title[Five-dimensional nilpotent right alternative algebras]{The algebraic classification of five-dimensional nilpotent right alternative algebras}

\author{Zohida Raxmatova}
\address{New Uzbekistan University, Tashkent, Uzbekistan}
\email{rahmatovazohida04@gmail.com}

\author{Aloberdi Sattarov}
\address{Turan International University, Namangan, Uzbekistan}
\email{saloberdi90@mail.ru}

\subjclass[2020]{17A30, 17D15, 17A60}
\keywords{right alternative algebra, nilpotent algebra, central extension, algebraic classification, second cohomology}

\begin{document}

\begin{abstract}
We develop the method of central extensions (the Skjelbred--Sund method) for nilpotent right alternative algebras over the field of complex numbers and apply it to obtain the algebraic classification of five-dimensional nilpotent right alternative algebras.
More precisely, we classify, up to isomorphism, all complex five-dimensional nilpotent right alternative algebras that have no annihilator component (that is, which are not a direct sum of a smaller algebra and a one-dimensional algebra with zero product) and which are not $2$-step nilpotent.
Every such algebra is a non-split central extension of a nontrivial nilpotent right alternative algebra of dimension three (by a two-dimensional space) or of dimension four (by a one-dimensional space).
For each of the relevant three- and four-dimensional algebras we compute the second cohomology space, the automorphism group and its action on the second cohomology, and we determine all orbits that give non-split extensions.
The resulting list consists of $124$ algebras and families of algebras; it contains $29$ algebras with two-dimensional annihilator and $95$ algebras with one-dimensional annihilator.
\end{abstract}

\maketitle

\tableofcontents

\section{Introduction}

The algebraic classification (up to isomorphism) of algebras of small dimension in a variety defined by a family of polynomial identities is a classical problem of the theory of non-associative algebras.
There are many results on the algebraic classification of small-dimensional nilpotent algebras in the varieties of Jordan, Lie, Leibniz, Zinbiel, Novikov, right commutative and many other algebras; see, for example, \cite{AKK23,AKKS21,AKKS21b,AFK22,AK21,CGS12,dG07,HA16,JKK21,KM21} and the survey \cite{K23}.
The main tool for the classification of nilpotent algebras is the method of central extensions.
It was introduced by Skjelbred and Sund \cite{SS78} for Lie algebras, and it was later adapted to many other varieties of algebras; our presentation follows the approach of Hegazi, Abdelwahab and Calder\'on Mart\'{\i}n \cite{HAC16}.

An algebra $\mathcal{A}$ is called \emph{right alternative} if it satisfies the identity
\[
(xy)y=x(yy).
\]
Right alternative algebras form one of the classical varieties of non-associative algebras: alternative algebras are exactly the algebras that are simultaneously left and right alternative.
The algebraic and geometric classification of complex nilpotent right alternative algebras of dimension at most four was obtained by Ismailov, Kaygorodov and Mustafa \cite{IKM22}.
The aim of the present paper is to give the algebraic classification of complex five-dimensional nilpotent right alternative algebras.

Two classes of algebras have to be separated from the outset.
\begin{itemize}[leftmargin=*]
\item If $\mathcal{A}=\mathcal{A}'\oplus \C e$ is the direct sum of an algebra $\mathcal{A}'$ and a one-dimensional algebra with zero product, then $\mathcal{A}$ is right alternative (nilpotent) if and only if $\mathcal{A}'$ is; such an algebra is said to have an \emph{annihilator component}.
Five-dimensional algebras of this form are in bijection with the four-dimensional nilpotent right alternative algebras, which are known \cite{IKM22}.
In the language of central extensions they are exactly the \emph{split} central extensions (Lemma~\ref{lem:component}).
\item An algebra $\mathcal{A}$ is \emph{$2$-step nilpotent} if $(xy)z=x(yz)=0$ for all $x,y,z\in\mathcal{A}$.
Every $2$-step nilpotent algebra is right alternative, so their classification does not involve the right alternative identity at all; it is a problem of linear algebra (classification of subspaces of bilinear forms under the action of the general linear group), see, e.g., \cite{IKP22}.
In \cite{IKM22} the $2$-step nilpotent algebras are called ``non-pure''.
\end{itemize}

Our main result is the following theorem.

\begin{theoremA}
Let $\mathcal{A}$ be a five-dimensional complex nilpotent right alternative algebra which has no annihilator component and which is not $2$-step nilpotent.
Then $\mathcal{A}$ is isomorphic to one of the algebras $\A{1},\dots,\A{124}$ listed in Theorem~\ref{thm:main}.
The algebras $\A{1},\dots,\A{29}$ have a two-dimensional annihilator, and the algebras $\A{30},\dots,\A{124}$ have a one-dimensional annihilator.
The listed algebras are pairwise non-isomorphic, except for the isomorphisms between members of the same family that are indicated explicitly in Theorem~\ref{thm:main}.
\end{theoremA}

The list of Theorem~A consists of $93$ individual algebras, $27$ one-parameter families and $4$ two-parameter families (Corollary~\ref{cor:count}).
In particular, there are infinitely many isomorphism classes of such algebras.

In Section~\ref{sec:prelim} we collect the necessary definitions and the classification of nilpotent right alternative algebras of dimension at most four.
In Section~\ref{sec:method} we develop the method of central extensions for right alternative algebras and prove all the auxiliary statements that are needed; in particular, we show that every algebra of Theorem~A is a non-split central extension of a three-dimensional algebra by a two-dimensional space or of a four-dimensional algebra by a one-dimensional space (Lemma~\ref{lem:quotient}).
In Section~\ref{sec:data} we compute, for every relevant algebra of dimension three and four, the second cohomology space, the automorphism group and the set of cocycles giving non-split extensions.
Sections~\ref{sec:ann2}, \ref{sec:pure} and~\ref{sec:twostep} contain the orbit computations.
The final list and the counting of the isomorphism classes are given in Section~\ref{sec:main}.

All algebras and vector spaces are finite-dimensional over the field $\C$ of complex numbers (all arguments work over any algebraically closed field of characteristic zero).
In multiplication tables only the nonzero products of basis elements are written.
For a basis $e_1,\dots,e_n$ of an algebra and a subspace or a family of vectors we use $\sspan{\cdot}$ for the linear span.
\section{Preliminaries}\label{sec:prelim}

Let $\mathcal{A}$ be an algebra, i.e.\ a vector space with a bilinear multiplication $(x,y)\mapsto xy$.
The \emph{associator} of $x,y,z\in\mathcal{A}$ is $(x,y,z)=(xy)z-x(yz)$.
The algebra $\mathcal{A}$ is \emph{right alternative} if $(x,y,y)=0$ for all $x,y\in\mathcal{A}$.

\begin{lemma}\label{lem:linearization}
An algebra $\mathcal{A}$ over a field of characteristic different from $2$ is right alternative if and only if
\begin{equation}\label{eq:linRA}
(x,y,z)+(x,z,y)=0\qquad\text{for all } x,y,z\in\mathcal{A}.
\end{equation}
\end{lemma}

\begin{proof}
The associator is trilinear.
If $\mathcal{A}$ is right alternative, then $0=(x,y+z,y+z)-(x,y,y)-(x,z,z)=(x,y,z)+(x,z,y)$.
Conversely, putting $z=y$ in \eqref{eq:linRA} gives $2(x,y,y)=0$, hence $(x,y,y)=0$.
The last statement follows from trilinearity and the symmetry of \eqref{eq:linRA} in $y,z$.
\end{proof}

For an algebra $\mathcal{A}$ put $\mathcal{A}^{1}=\mathcal{A}$ and $\mathcal{A}^{k+1}=\sum_{i=1}^{k}\mathcal{A}^{i}\mathcal{A}^{k+1-i}$.
The algebra $\mathcal{A}$ is \emph{nilpotent} if $\mathcal{A}^{n}=0$ for some $n$.
The \emph{annihilator} of $\mathcal{A}$ is the ideal
\[
\Ann(\mathcal{A})=\{x\in\mathcal{A} : x\mathcal{A}+\mathcal{A}x=0\}.
\]
If $\mathcal{A}\neq0$ is nilpotent and $k$ is the largest integer with $\mathcal{A}^{k}\neq0$, then $\mathcal{A}^{k}\mathcal{A}+\mathcal{A}\mathcal{A}^{k}\subseteq\mathcal{A}^{k+1}=0$, so $0\neq\mathcal{A}^{k}\subseteq\Ann(\mathcal{A})$.
Thus \emph{every nonzero nilpotent algebra has a nonzero annihilator}.

A subspace $I\subseteq\Ann(\mathcal{A})$ is an \emph{annihilator component} of $\mathcal{A}$ if $\mathcal{A}=\mathcal{A}'\oplus I$ for an ideal $\mathcal{A}'$ of $\mathcal{A}$ (then $\mathcal{A}$ is the direct sum of the algebra $\mathcal{A}'$ and the algebra $I$ with zero multiplication).
An algebra is \emph{$2$-step nilpotent} if $\mathcal{A}^{2}\subseteq\Ann(\mathcal{A})$, i.e.\ if $(xy)z=x(yz)=0$ for all $x,y,z$.
Since all associators of a $2$-step nilpotent algebra vanish, we have:

\begin{lemma}\label{lem:2step}
Every $2$-step nilpotent algebra is a nilpotent right alternative algebra.
\end{lemma}

An algebra with zero multiplication will be called \emph{trivial}.
Clearly $\mathcal{A}$ is $2$-step nilpotent if and only if the quotient algebra $\mathcal{A}/\Ann(\mathcal{A})$ is trivial.

Using the classification obtained in \cite{IKM22}, we present the classification of nilpotent right alternative algebras of dimensions up to four.
Following \cite{IKM22}, a nilpotent right alternative algebra is called \emph{pure} if it is not $2$-step nilpotent.

\begin{theorem}[\cite{IKM22}]\label{thm:dim3}
There are no nonzero nilpotent algebras of dimension one, and every nontrivial two-dimensional nilpotent right alternative algebra is isomorphic to
\[
\RR{2}{1}:\ e_1e_1=e_2 .
\]
Every nontrivial three-dimensional nilpotent right alternative algebra is isomorphic to one of the following pairwise non-isomorphic algebras:
\[
\begin{array}{ll}
\RR31: & e_1e_1=e_2;\\
\RR32: & e_1e_1=e_3,\ e_2e_2=e_3;\\
\RR33: & e_1e_2=e_3,\ e_2e_1=-e_3;\\
\RR34(\alpha): & e_1e_1=\alpha e_3,\ e_2e_1=e_3,\ e_2e_2=e_3,\qquad \alpha\in\C;\\
\RR35: & e_1e_1=e_2,\ e_1e_2=e_3,\ e_2e_1=e_3.
\end{array}
\]
\end{theorem}

Note that $\RR31=\RR21\oplus\C e_3$ and that $\RR35$ is the only pure algebra of this list.

\begin{theorem}[\cite{IKM22}]\label{thm:dim4pure}
Every four-dimensional pure nilpotent right alternative algebra is isomorphic to one of the following pairwise non-isomorphic algebras:
\[
\begin{array}{ll}
\RR41: & e_1e_1=e_2,\ e_1e_2=e_3,\ e_2e_1=e_3;\\
\RR42: & e_1e_1=e_2,\ e_1e_2=e_4,\ e_2e_1=e_4,\ e_3e_3=e_4;\\
\RR43: & e_1e_1=e_2,\ e_1e_2=e_4,\ e_2e_1=e_4,\ e_3e_1=e_4,\ e_3e_3=e_4;\\
\RR44: & e_1e_1=e_2,\ e_1e_2=e_4,\ e_2e_1=e_4,\ e_3e_1=e_4;\\
\RR45: & e_1e_2=e_3,\ e_1e_3=e_4,\ e_2e_1=-e_3;\\
\RR46: & e_1e_2=e_3,\ e_1e_3=e_4,\ e_2e_1=-e_3,\ e_2e_2=e_4;\\
\RR47: & e_2e_1=e_3,\ e_2e_2=e_3,\ e_2e_3=e_4,\ e_3e_2=e_4;\\
\RR48: & e_1e_1=e_4,\ e_2e_1=e_3,\ e_2e_2=e_3,\ e_2e_3=e_4,\ e_3e_2=e_4;\\
\RR49: & e_1e_1=e_2,\ e_1e_2=e_3,\ e_2e_1=e_3,\ e_1e_3=e_4,\ e_2e_2=e_4,\ e_3e_1=e_4.
\end{array}
\]
\end{theorem}

Note that $\RR41=\RR35\oplus\C e_4$.

\begin{theorem}[\cite{IKM22}]\label{thm:dim4nonpure}
Every nontrivial four-dimensional $2$-step nilpotent right alternative algebra is isomorphic to one of the following pairwise non-isomorphic algebras:
\[
\begin{array}{ll}
\M1: & e_1e_1=e_2;\\
\M2: & e_1e_1=e_2,\ e_3e_3=e_4;\\
\M3: & e_1e_2=e_3,\ e_2e_1=-e_3;\\
\M4(\alpha): & e_1e_1=e_3,\ e_1e_2=e_3,\ e_2e_2=\alpha e_3;\\
\M5: & e_1e_1=e_3,\ e_1e_2=e_3,\ e_2e_1=e_3;\\
\M6: & e_1e_2=e_3,\ e_2e_1=e_4,\ e_2e_2=-e_3;\\
\M7(\gamma): & e_1e_1=e_3,\ e_1e_2=e_4,\ e_2e_1=-\gamma e_3,\ e_2e_2=-e_4;\\
\M8(\alpha): & e_1e_1=e_4,\ e_1e_2=\alpha e_4,\ e_2e_1=-\alpha e_4,\ e_2e_2=e_4,\ e_3e_3=e_4;\\
\M9: & e_1e_2=e_4,\ e_1e_3=e_4,\ e_2e_1=-e_4,\ e_2e_2=e_4,\ e_3e_1=e_4;\\
\M{10}: & e_1e_1=e_4,\ e_1e_2=e_4,\ e_2e_1=-e_4,\ e_3e_3=e_4;\\
\M{11}: & e_1e_2=e_3,\ e_2e_1=e_4;\\
\M{12}: & e_1e_1=e_4,\ e_1e_2=e_3,\ e_2e_1=-e_3,\ e_2e_2=2e_3+e_4;\\
\M{13}: & e_2e_1=e_4,\ e_2e_2=e_3;\\
\M{14}(\lambda): & e_1e_2=e_4,\ e_2e_1=\lambda e_4,\ e_2e_2=e_3;\\
\M{15}: & e_1e_2=e_4,\ e_2e_1=-e_4,\ e_3e_3=e_4.
\end{array}
\]
\end{theorem}

We have $\M1=\RR21\oplus\C e_3\oplus\C e_4$, $\M2=\RR21\oplus\RR21$ and $\M3=\RR33\oplus\C e_4$; the algebras $\M4(\alpha)$ and $\M5$ also have an annihilator component $\C e_4$. We shall use the parameter $\lambda$ for the algebras $\M{14}(\lambda)$, in order to reserve the letters $\alpha,\beta$ for the parameters of the five-dimensional algebras.
\section{The method of central extensions}\label{sec:method}

Throughout this section $\mathcal{A}$ is a right alternative algebra and $\mathbb{V}$ is a vector space.

\begin{definition}
A bilinear map $\theta\colon\mathcal{A}\times\mathcal{A}\to\mathbb{V}$ is called a \emph{cocycle} if
\begin{equation}\label{eq:cocycle}
\theta(xy,z)-\theta(x,yz)=\theta(x,zy)-\theta(xz,y)\qquad\text{for all }x,y,z\in\mathcal{A}.
\end{equation}
The cocycles form a vector space $\Z(\mathcal{A},\mathbb{V})$.
For a linear map $f\colon\mathcal{A}\to\mathbb{V}$ we define $\delta f\colon\mathcal{A}\times\mathcal{A}\to\mathbb{V}$ by $\delta f(x,y)=f(xy)$, and we put $\B(\mathcal{A},\mathbb{V})=\{\delta f: f\in\Hom(\mathcal{A},\mathbb{V})\}$.
The \emph{second cohomology space} is $\HH(\mathcal{A},\mathbb{V})=\Z(\mathcal{A},\mathbb{V})/\B(\mathcal{A},\mathbb{V})$.
We write $[\theta]$ for the class of a cocycle $\theta$.
\end{definition}

\begin{lemma}\label{lem:coboundary}
$\B(\mathcal{A},\mathbb{V})\subseteq\Z(\mathcal{A},\mathbb{V})$.
\end{lemma}

\begin{proof}
For $\theta=\delta f$ the difference of the two sides of \eqref{eq:cocycle} equals $f\big((x,y,z)+(x,z,y)\big)$, which vanishes by Lemma~\ref{lem:linearization}.
\end{proof}

\begin{definition}
For a bilinear map $\theta\colon\mathcal{A}\times\mathcal{A}\to\mathbb{V}$ the \emph{central extension} $\mathcal{A}_{\theta}$ of $\mathcal{A}$ by $\mathbb{V}$ is the vector space $\mathcal{A}\oplus\mathbb{V}$ with the product
\[
[x+v,\,y+w]_{\theta}=xy+\theta(x,y),\qquad x,y\in\mathcal{A},\ v,w\in\mathbb{V}.
\]
If $\dim\mathbb{V}=s$ we speak of an $s$-dimensional central extension.
\end{definition}

\begin{lemma}\label{lem:RAext}
The algebra $\mathcal{A}_{\theta}$ is right alternative if and only if $\theta\in\Z(\mathcal{A},\mathbb{V})$.
If $\mathcal{A}$ is nilpotent, then so is $\mathcal{A}_\theta$.
\end{lemma}

\begin{proof}
Let $X=x+v$, $Y=y+w$, $Z=z+u$ with $x,y,z\in\mathcal{A}$ and $v,w,u\in\mathbb{V}$.
Since $\mathbb{V}$ annihilates $\mathcal{A}_\theta$, the associator in $\mathcal{A}_\theta$ is
\[
(X,Y,Z)_{\theta}=[[x,y]_\theta,z]_\theta-[x,[y,z]_\theta]_\theta=(x,y,z)+\theta(xy,z)-\theta(x,yz).
\]
Hence
\[
(X,Y,Z)_\theta+(X,Z,Y)_\theta=\big((x,y,z)+(x,z,y)\big)+\big(\theta(xy,z)-\theta(x,yz)+\theta(xz,y)-\theta(x,zy)\big).
\]
The first bracket lies in $\mathcal{A}$ and vanishes because $\mathcal{A}$ is right alternative; the second lies in $\mathbb{V}$.
By Lemma~\ref{lem:linearization}, $\mathcal{A}_\theta$ is right alternative if and only if the second bracket vanishes for all $x,y,z$, i.e.\ if and only if $\theta$ is a cocycle.
Finally, since $\theta$ takes values in $\mathbb{V}\subseteq\Ann(\mathcal{A}_\theta)$, an induction on $k$ gives $\mathcal{A}_\theta^{k}\subseteq\mathcal{A}^{k}+\sum_{i+j=k}\theta(\mathcal{A}^{i},\mathcal{A}^{j})$ for $k\ge2$.
If $\mathcal{A}^{n}=0$, then $\mathcal{A}^{2n}_\theta\subseteq\sum_{i+j=2n}\theta(\mathcal{A}^i,\mathcal{A}^j)=0$, because $i+j=2n$ forces $i\geq n$ or $j\geq n$.
\end{proof}

The \emph{annihilator of a bilinear map} $\theta$ is $\Ann(\theta)=\{x\in\mathcal{A}:\theta(x,\mathcal{A})+\theta(\mathcal{A},x)=0\}$.

\begin{lemma}\label{lem:annext}
$\Ann(\mathcal{A}_\theta)=\big(\Ann(\theta)\cap\Ann(\mathcal{A})\big)\oplus\mathbb{V}$.
\end{lemma}

\begin{proof}
An element $x+v$ ($x\in\mathcal{A}$, $v\in\mathbb{V}$) lies in $\Ann(\mathcal{A}_\theta)$ if and only if $xy+\theta(x,y)=0$ and $yx+\theta(y,x)=0$ for all $y\in\mathcal{A}$.
Since $\mathcal{A}\cap\mathbb{V}=0$, this means $x\mathcal{A}=\mathcal{A}x=0$ and $\theta(x,\mathcal{A})=\theta(\mathcal{A},x)=0$.
\end{proof}

\begin{lemma}\label{lem:cohomologous}
If $\theta,\vartheta\in\Z(\mathcal{A},\mathbb{V})$ and $\vartheta-\theta=\delta f$, then the map $\Phi(x+v)=x+v+f(x)$ is an isomorphism $\mathcal{A}_\theta\to\mathcal{A}_\vartheta$.
In particular, if $\theta\in\B(\mathcal{A},\mathbb{V})$, then $\mathcal{A}_\theta\cong\mathcal{A}\oplus\mathbb{V}$ is the direct sum of $\mathcal{A}$ and the trivial algebra $\mathbb{V}$.
\end{lemma}

\begin{proof}
$\Phi$ is bijective, and
$\Phi([x+v,y+w]_\theta)=xy+\theta(x,y)+f(xy)=xy+\vartheta(x,y)=[\Phi(x+v),\Phi(y+w)]_\vartheta$.
\end{proof}

The following result shows that every algebra with a non-zero annihilator is a central extension of a smaller-dimensional
algebra.

\begin{lemma}\label{lem:quotient}
Let $\mathcal{A}$ be an $n$-dimensional right alternative algebra with $m=\dim\Ann(\mathcal{A})\neq0$.
Then there exist an $(n-m)$-dimensional right alternative algebra $\mathcal{A}'$, an $m$-dimensional space $\mathbb{V}$ and a cocycle $\theta\in\Z(\mathcal{A}',\mathbb{V})$ with $\Ann(\mathcal{A}')\cap\Ann(\theta)=0$ such that $\mathcal{A}\cong\mathcal{A}'_\theta$.
Moreover $\mathcal{A}'\cong\mathcal{A}/\Ann(\mathcal{A})$; in particular $\mathcal{A}'$ is unique up to isomorphism.
\end{lemma}

\begin{proof}
Let $\mathcal{A}'$ be a complement of $\mathbb{V}:=\Ann(\mathcal{A})$ in $\mathcal{A}$ and let $P\colon\mathcal{A}\to\mathcal{A}'$ be the projection along $\mathbb{V}$.
Define $[x,y]'=P(xy)$ for $x,y\in\mathcal{A}'$.
Since $\mathbb{V}$ annihilates $\mathcal{A}$, for all $x,y\in\mathcal{A}$ we have $xy=P(x)P(y)$, hence $P(xy)=[P(x),P(y)]'$.
Thus $P$ is a surjective homomorphism with kernel $\Ann(\mathcal{A})$, so $(\mathcal{A}',[\cdot,\cdot]')\cong\mathcal{A}/\Ann(\mathcal{A})$ is right alternative.
Put $\theta(x,y)=xy-[x,y]'\in\mathbb{V}$ for $x,y\in\mathcal{A}'$.
Then the identity map of the vector space $\mathcal{A}=\mathcal{A}'\oplus\mathbb{V}$ is an isomorphism $\mathcal{A}'_\theta\to\mathcal{A}$; by Lemma~\ref{lem:RAext}, $\theta$ is a cocycle.
By Lemma~\ref{lem:annext}, $\mathbb{V}=\Ann(\mathcal{A})=(\Ann(\theta)\cap\Ann(\mathcal{A}'))\oplus\mathbb{V}$, hence $\Ann(\theta)\cap\Ann(\mathcal{A}')=0$.
\end{proof}

Fix a basis $e_1',\dots,e_s'$ of $\mathbb{V}$ (in the applications, $e_{n+1},\dots,e_{n+s}$).
Every bilinear map $\theta\colon\mathcal{A}\times\mathcal{A}\to\mathbb{V}$ can be written uniquely as $\theta=\sum_{i=1}^{s}\theta_i e_i'$ with bilinear forms $\theta_i\colon\mathcal{A}\times\mathcal{A}\to\C$.
Clearly $\theta\in\Z(\mathcal{A},\mathbb{V})$ if and only if all $\theta_i\in\Z(\mathcal{A},\C)$, $\theta\in\B(\mathcal{A},\mathbb{V})$ if and only if all $\theta_i\in\B(\mathcal{A},\C)$, and $\Ann(\theta)=\bigcap_i\Ann(\theta_i)$.
We write $\Z(\mathcal{A})$, $\B(\mathcal{A})$, $\HH(\mathcal{A})$ for $\Z(\mathcal{A},\C)$, $\B(\mathcal{A},\C)$, $\HH(\mathcal{A},\C)$.

\begin{definition}
A central extension of $\mathcal{A}$ is called \emph{split} if it has an annihilator component, and \emph{non-split} otherwise.
\end{definition}

\begin{lemma}\label{lem:component}
Let $\theta=\sum_{i=1}^{s}\theta_i e_i'\in\Z(\mathcal{A},\mathbb{V})$ with $\Ann(\theta)\cap\Ann(\mathcal{A})=0$.
Then $\mathcal{A}_\theta$ has an annihilator component if and only if the classes $[\theta_1],\dots,[\theta_s]$ are linearly dependent in $\HH(\mathcal{A})$.
\end{lemma}

\begin{proof}
Suppose that $\sum_i c_i[\theta_i]=0$ with $(c_1,\dots,c_s)\neq0$, i.e.\ $\sum_ic_i\theta_i=\delta g$ for a linear form $g$ on $\mathcal{A}$.
Choose a basis $e_1'',\dots,e_s''$ of $\mathbb{V}$ such that the linear form $\ell(\sum_i a_ie_i')=\sum_ic_ia_i$ satisfies $\ell(e_j'')=0$ for $j<s$ and $\ell(e_s'')=1$.
Writing $\theta=\sum_j\theta_j''e_j''$ and applying $\ell$, we get $\theta_s''=\ell\circ\theta=\sum_ic_i\theta_i=\delta g$.
By Lemma~\ref{lem:cohomologous} we may replace $\theta$ by the cohomologous cocycle $\theta-\delta g\,e_s''$, whose last coordinate is zero.
Then $e_s''\in\Ann(\mathcal{A}_\theta)$ and $\mathcal{A}_\theta=\big(\mathcal{A}\oplus\sspan{e_1'',\dots,e_{s-1}''}\big)\oplus\C e_s''$, where the first summand is an ideal; so $\C e''_s$ is an annihilator component.

Conversely, let $\mathcal{A}_\theta=\mathcal{B}\oplus I$ with $\mathcal{B}$ an ideal and $0\neq I\subseteq\Ann(\mathcal{A}_\theta)$.
By Lemma~\ref{lem:annext} and the hypothesis, $\Ann(\mathcal{A}_\theta)=\mathbb{V}$, so $I\subseteq\mathbb{V}$.
Let $\pi\colon\mathcal{A}_\theta\to\C$ be a linear form with $\pi(\mathcal{B})=0$ and $\pi|_{I}\neq0$; then $\pi|_\mathbb{V}\neq 0$.
Since $[\mathcal{A}_\theta,\mathcal{A}_\theta]_\theta=[\mathcal{B},\mathcal{B}]_\theta\subseteq\mathcal{B}$, for all $x,y\in\mathcal{A}$ we get
$0=\pi(xy+\theta(x,y))=\pi(xy)+\sum_i\pi(e_i')\theta_i(x,y)$.
Hence $\sum_i\pi(e_i')\theta_i=-\delta(\pi|_{\mathcal{A}})\in\B(\mathcal{A})$ with $(\pi(e_1'),\dots,\pi(e_s'))\neq0$.
\end{proof}

The group $\Aut(\mathcal{A})$ acts on $\Z(\mathcal{A})$ by $(\phi\theta)(x,y)=\theta(\phi(x),\phi(y))$.
Since $\phi(\delta f)=\delta(f\circ\phi)$, the subspace $\B(\mathcal{A})$ is invariant, and we obtain an action on $\HH(\mathcal{A})$ and on the Grassmannian $G_s(\HH(\mathcal{A}))$ of $s$-dimensional subspaces: $\phi\sspan{[\theta_1],\dots,[\theta_s]}=\sspan{[\phi\theta_1],\dots,[\phi\theta_s]}$.
If $\sspan{[\theta_1],\dots,[\theta_s]}=\sspan{[\vartheta_1],\dots,[\vartheta_s]}$, then $\bigcap_i\Ann(\theta_i)\cap\Ann(\mathcal{A})=\bigcap_i\Ann(\vartheta_i)\cap\Ann(\mathcal{A})$, because $\B(\mathcal{A})$ consists of forms vanishing on $\Ann(\mathcal{A})\times\mathcal{A}$ and $\mathcal{A}\times\Ann(\mathcal{A})$.
Hence the set
\[
\T_s(\mathcal{A})=\Big\{W=\sspan{[\theta_1],\dots,[\theta_s]}\in G_s(\HH(\mathcal{A})) :\ \bigcap_{i=1}^{s}\Ann(\theta_i)\cap\Ann(\mathcal{A})=0\Big\}
\]
is well defined and $\Aut(\mathcal{A})$-invariant.
Let $\E(\mathcal{A},\mathbb{V})$ be the set of all non-split central extensions $\mathcal{A}_\theta$ with $\theta=\sum_i\theta_ie_i'$ and $\Ann(\theta)\cap\Ann(\mathcal{A})=0$.
By Lemma~\ref{lem:component}, $\mathcal{A}_\theta\in\E(\mathcal{A},\mathbb{V})$ if and only if $\sspan{[\theta_1],\dots,[\theta_s]}\in\T_s(\mathcal{A})$.

\begin{lemma}\label{lem:orbits}
Let $\mathcal{A}_\theta,\mathcal{A}_\vartheta\in\E(\mathcal{A},\mathbb{V})$ with $\theta=\sum_i\theta_ie'_i$ and $\vartheta=\sum_i\vartheta_ie'_i$.
Then $\mathcal{A}_\theta\cong\mathcal{A}_\vartheta$ if and only if
$\Orb\sspan{[\theta_1],\dots,[\theta_s]}=\Orb\sspan{[\vartheta_1],\dots,[\vartheta_s]}$.
\end{lemma}

\begin{proof}
Suppose $\phi\sspan{[\vartheta_i]}_i=\sspan{[\theta_i]}_i$ for some $\phi\in\Aut(\mathcal{A})$.
Then there is an invertible matrix $(m_{ij})$ and linear forms $f_i$ with $\vartheta_i\circ(\phi\times\phi)=\sum_jm_{ij}\theta_j+\delta f_i$.
Let $M\colon\mathbb{V}\to\mathbb{V}$ be the invertible linear map with $\sum_i\big(\sum_j m_{ij}\theta_j\big)e'_i=M\big(\sum_j\theta_je'_j\big)$ and let $f=\sum_i f_ie_i'$.
Then $\vartheta(\phi x,\phi y)=M\theta(x,y)+f(xy)$, and the map $\Psi(x+v)=\phi(x)+f(x)+Mv$ satisfies
\[
\Psi([x+v,y+w]_\theta)=\phi(xy)+f(xy)+M\theta(x,y)=\phi(x)\phi(y)+\vartheta(\phi x,\phi y)=[\Psi(x+v),\Psi(y+w)]_\vartheta,
\]
so $\Psi$ is an isomorphism.
Conversely, let $\Psi\colon\mathcal{A}_\theta\to\mathcal{A}_\vartheta$ be an isomorphism.
By Lemma~\ref{lem:annext}, $\Ann(\mathcal{A}_\theta)=\Ann(\mathcal{A}_\vartheta)=\mathbb{V}$, so $\Psi(\mathbb{V})=\mathbb{V}$ and $\Psi$ induces an automorphism $\phi$ of $\mathcal{A}_\theta/\mathbb{V}\cong\mathcal{A}$.
Thus $\Psi(x)=\phi(x)+g(x)$ for $x\in\mathcal{A}$, with $g\colon\mathcal{A}\to\mathbb{V}$ linear, and $\Psi|_{\mathbb{V}}=M$ is invertible.
Comparing the $\mathbb{V}$-components of $\Psi([x,y]_\theta)=[\Psi x,\Psi y]_\vartheta$ gives $M\theta(x,y)+g(xy)=\vartheta(\phi x,\phi y)$.
Taking coordinates, every $\phi\vartheta_i$ is a linear combination of $\theta_1,\dots,\theta_s$ modulo $\B(\mathcal{A})$, and since $M$ is invertible, $\sspan{[\phi\vartheta_i]}_i=\sspan{[\theta_i]}_i$.
\end{proof}

By Lemma~\ref{lem:quotient} and Lemma~\ref{lem:orbits}, the classification of the algebras of Theorem~A reduces to the following procedure, applied to every nontrivial nilpotent right alternative algebra $\mathcal{A}$ of dimension $3$ (with $s=2$) and of dimension $4$ (with $s=1$):
\begin{enumerate}[label=(\arabic*)]
\item determine $\HH(\mathcal{A})$, $\Ann(\mathcal{A})$ and $\Aut(\mathcal{A})$;
\item determine the set of $\Aut(\mathcal{A})$-orbits on $\T_s(\mathcal{A})$;
\item for every orbit, construct the extension corresponding to a representative.
\end{enumerate}
Five-dimensional algebras obtained from non-isomorphic algebras $\mathcal{A}$, or from different values of $s$, are non-isomorphic, since $\mathcal{A}\cong\mathcal{B}/\Ann(\mathcal{B})$ and $s=\dim\Ann(\mathcal{B})$ are isomorphism invariants of $\mathcal{B}$.

\textbf{Notation.}\label{subsec:notation}
Let $\mathcal{A}$ have basis $e_1,\dots,e_n$.
We denote by $\D{ij}$ the bilinear form with $\D{ij}(e_l,e_m)=\delta_{il}\delta_{jm}$; the forms $\D{ij}$ ($1\le i,j\le n$) form a basis of the space of bilinear forms on $\mathcal{A}$.
For every algebra $\mathcal{A}$ we fix cocycles whose classes $\N1,\N2,\dots$ form a basis of $\HH(\mathcal{A})$, and we write $\theta=\sum_i\alpha_i\N{i}$ for a class of $\HH(\mathcal{A})$.
For $\phi\in\Aut(\mathcal{A})$ we write $\phi\theta=\sum_i\alpha_i^{*}\N{i}$; the \emph{action formulas} express $\alpha^{*}_i$ in terms of $\alpha_1,\alpha_2,\dots$ and of the entries of $\phi$.
Automorphisms are described by the images of the basis vectors.
For a one-dimensional extension we write the class $\theta$ itself instead of $\sspan{\theta}$ when no confusion is possible; the new basis vector is $e_5$.
For a two-dimensional extension $\sspan{\theta_1,\theta_2}$ of a three-dimensional algebra, the new basis vectors are $e_4$ (corresponding to $\theta_1$) and $e_5$ (corresponding to $\theta_2$).
Scalar multiples of a class define the same point of $G_1(\HH(\mathcal{A}))$; we use this tacitly when we ``normalize'' a coefficient to $1$.
\section{Cohomology and automorphisms of the algebras of dimension three and four}\label{sec:data}

Now we obtain the classification of 5-dimensional nilpotent right alternative algebras by constructing central extensions of three and four dimensional nilpotent algebras. In this section, we give the description of the second cohomology space and automorphism of  three-dimensional nilpotent right alternative algebras \cite{IKM22}.

%
%

\begin{proposition}\label{prop:data3}
The second cohomology spaces of the three-dimensional nilpotent right alternative algebras are:
\[
\begin{array}{ll}
\HH(\RR31)=\sspan{[\D{12}+\D{21}],[\D{13}],[\D{31}],[\D{33}]}, & \Ann(\RR31)=\sspan{e_2,e_3};\\
\HH(\RR32)=\sspan{[\D{12}],[\D{21}],[\D{22}]}, & \Ann(\RR32)=\sspan{e_3};\\
\HH(\RR33)=\sspan{[\D{11}],[\D{12}],[\D{22}],[\D{13}],[\D{23}]}, & \Ann(\RR33)=\sspan{e_3};\\
\HH(\RR34(\alpha))=\sspan{[\D{11}],[\D{12}],[\D{22}]}\ (\alpha\neq0), & \Ann(\RR34(\alpha))=\sspan{e_3};\\
\HH(\RR34(0))=\sspan{[\D{11}],[\D{12}],[\D{22}],[\D{23}+\D{32}]}, & \\
\HH(\RR35)=\sspan{[\D{13}+\D{22}+\D{31}]}. &
\end{array}
\]
\end{proposition}

\begin{corollary}\label{cor:empty3}
$\T_2(\mathcal{A})=\emptyset$ for $\mathcal{A}\in\{\RR32,\RR34(\alpha)_{\alpha\neq0},\RR35\}$.
\end{corollary}

\begin{proof}
$\dim\HH(\RR35)=1<2$.
For $\RR32$ and $\RR34(\alpha)$, $\alpha\neq0$, no basis cocycle involves $e_3$, hence $e_3\in\Ann(\theta)\cap\Ann(\mathcal{A})$ for every cocycle $\theta$.
\end{proof}

\begin{proposition}\label{prop:data4pure}
For the pure four-dimensional  nilpotent right alternative algebras we have $\Ann(\RR41)=\sspan{e_3,e_4}$, $\Ann(\RR4i)=\sspan{e_4}$ for $2\le i\le 9$, and
\[
\begin{array}{l}
\HH(\RR41)=\sspan{[\D{14}],[\D{13}+\D{22}+\D{31}],[\D{41}],[\D{44}]},\\
\HH(\RR42)=\HH(\RR43)=\sspan{[\D{13}],[\D{12}+\D{21}],[\D{31}]},\\
\HH(\RR44)=\sspan{[\D{13}],[\D{12}+\D{21}],[\D{33}],[\D{14}+\D{22}+\D{23}+\D{32}+\D{41}]},\\
\HH(\RR45)=\sspan{[\D{11}],[\D{12}],[\D{22}],[\D{14}+\D{41}],[\D{23}],[\D{24}-\D{33}+\D{42}]},\\
\HH(\RR46)=\sspan{[\D{11}],[\D{12}],[\D{13}],[\D{23}],[\D{14}+\D{32}+\D{41}],[\D{24}-\D{33}+\D{42}]},\\
\HH(\RR47)=\HH(\RR48)=\sspan{[\D{11}],[\D{12}],[\D{21}]},\\
\HH(\RR49)=\sspan{[\D{14}+\D{23}+\D{32}+\D{41}]}.
\end{array}
\]
\end{proposition}

\begin{remark}
From the above previous proposition, we obtain that any extension of the algebras $\RR42, \RR43, \RR47$ and $\RR48$  are split. Therefore, is sufficient to consider for the algebras $\RR41$, $\RR44$,  $\RR45$, $\RR46$ and $\RR49.$\end{remark}

\begin{proposition}\label{prop:data4twostep}
For the $2$-step nilpotent four-dimensional algebras we have $\Ann(\M1)=\sspan{e_2,e_3,e_4}$; $\Ann(\M2)=\sspan{e_2,e_4}$; $\Ann(\M i)=\sspan{e_3,e_4}$ for $i\in\{3,4,5,6,7,11,12,13,14\}$; $\Ann(\M i)=\sspan{e_4}$ for $i\in\{8,9,10,15\}$; and
\[
\begin{array}{l}
\HH(\M1)=\sspan{[\D{13}],[\D{14}],[\D{12}+\D{21}],[\D{31}],[\D{33}],[\D{34}],[\D{41}],[\D{43}],[\D{44}]},\\
\HH(\M2)=\sspan{[\D{13}],[\D{12}+\D{21}],[\D{31}],[\D{34}+\D{43}]},\\
\HH(\M3)=\sspan{[\D{11}],[\D{12}],[\D{13}],[\D{14}],[\D{22}],[\D{23}],[\D{24}],[\D{41}],[\D{42}],[\D{43}],[\D{44}]},\\
\HH(\M4(\alpha))=\sspan{[\D{11}],[\D{12}],[\D{14}],[\D{21}],[\D{24}],[\D{41}],[\D{42}],[\D{44}]}\quad(\alpha\neq0),\\
\HH(\M4(0))=\sspan{[\D{11}],[\D{14}],[\D{21}],[\D{22}],[\D{24}],[\D{13}+\D{31}],[\D{41}],[\D{42}],[\D{44}]},\\
\HH(\M5)=\sspan{[\D{11}],[\D{12}],[\D{14}],[\D{22}],[\D{24}],[\D{41}],[\D{42}],[\D{44}]},\\
\HH(\M6)=\sspan{[\D{11}],[\D{12}],[\D{13}-\D{23}+\D{31}-\D{32}],[\D{24}+\D{42}],[\D{14}+\D{31}+\D{42}]},\\
\HH(\M7(\gamma))=\sspan{[\D{21}],[\D{12}],[\D{13}-\gamma\D{23}+\D{31}-\gamma\D{32}],[\D{14}-\D{24}+\D{41}-\D{42}]}\quad(\gamma\neq1),\\
\HH(\M7(1))=\sspan{[\D{11}],[\D{12}],[\D{13}-\D{23}+\D{31}-\D{41}],[\D{32}-\D{41}],[\D{14}-\D{24}+\D{41}-\D{42}]},\\
\HH(\M8(\alpha))=\HH(\M{10})=\HH(\M{15})=\sspan{[\D{11}],[\D{12}],[\D{13}],[\D{21}],[\D{22}],[\D{23}],[\D{31}],[\D{32}]},\\
\HH(\M9)=\sspan{[\D{11}],[\D{12}],[\D{13}],[\D{21}],[\D{22}],[\D{23}],[\D{32}],[\D{33}]},\\
\HH(\M{11})=\sspan{[\D{11}],[\D{22}],[\D{13}+\D{31}],[\D{14}+\D{31}],[\D{23}+\D{42}],[\D{24}+\D{42}]},\\
\HH(\M{12})=\sspan{[\D{11}],[\D{12}],[\D{14}+\D{23}+\D{32}+\D{41}],[\D{13}+2\D{23}-\D{24}+\D{31}+2\D{32}-\D{42}]},\\
\HH(\M{13})=\sspan{[\D{11}],[\D{12}],[\D{23}+\D{32}],[\D{31}-\D{42}],[\D{24}+\D{42}]},\\
\HH(\M{14}(\lambda))=\sspan{[\D{11}],[\D{12}],[(\lambda+1)\D{31}+\D{24}],[\D{23}+\D{32}],[\D{13}-\lambda\D{31}+\D{42}]}\quad(\lambda\neq0,-1),\\
\HH(\M{14}(0))=\sspan{[\D{11}],[\D{21}],[\D{24}+\D{31}],[\D{23}+\D{32}],[\D{14}+\D{41}],[\D{13}+\D{42}]},\\
\HH(\M{14}(-1))=\sspan{[\D{11}],[\D{12}],[\D{14}],[\D{24}],[\D{23}+\D{32}],[\D{13}+\D{31}+\D{42}]}.
\end{array}
\]
\end{proposition}

\begin{remark}
The dimension of $\HH(\mathcal{A})$ jumps exactly at the parameter values singled out above: $\dim\HH(\RR34(\alpha))=3$ for $\alpha\neq0$ and $4$ for $\alpha=0$; $\dim\HH(\M4(\alpha))=8$ for $\alpha\neq0$ and $9$ for $\alpha=0$; $\dim\HH(\M7(\gamma))=4$ for $\gamma\neq1$ and $5$ for $\gamma=1$; $\dim\HH(\M{14}(\lambda))=5$ for $\lambda\notin\{0,-1\}$ and $6$ for $\lambda\in\{0,-1\}$; and $\dim\HH(\M8(\alpha))=8$ for all $\alpha$.
Indeed, the rank of the linear system defining $\Z(\mathcal{A})$ drops only at the zeros of a suitable maximal minor, and these zeros are exactly the values listed.
\end{remark}

\begin{corollary}\label{cor:empty4}
$\T_1(\mathcal{A})=\emptyset$ for $\mathcal{A}\in\{\RR42,\RR43,\RR47,\RR48,\M4(\alpha)_{\alpha\neq0},\M5,\M8(\alpha),\M9,\M{10},\M{15}\}$.
\end{corollary}

\begin{proof}
For $\RR42,\RR43,\RR47,\RR48,\M8(\alpha),\M9,\M{10},\M{15}$ we have $e_4\in\Ann(\mathcal{A})$ and no basis cocycle involves $e_4$; for $\M4(\alpha)$ ($\alpha\neq0$) and $\M5$ we have $e_3\in\Ann(\mathcal{A})$ and no basis cocycle involves $e_3$.
In all cases this vector lies in $\Ann(\theta)\cap\Ann(\mathcal{A})$ for every cocycle $\theta$.
\end{proof}

The following proposition lists the automorphism groups of the algebras for which $\T_s(\mathcal{A})\neq\emptyset$.
All parameters are complex numbers subject only to the condition that the determinant of the matrix is nonzero.

\begin{proposition}\label{prop:aut}
The automorphism groups are as follows.
Every automorphism $\phi$ is written as its matrix with respect to the basis $e_1,\dots,e_n$: the $j$-th column of the matrix is the coordinate vector of $\phi(e_j)$, so that $\phi(e_j)=\sum_i\phi_{ij}e_i$.

\begin{itemize}
\item For the algebras $\RR31$, $\RR33$ and $\RR34(0)$, the corresponding matrices have the following forms, respectively:
$$\begin{pmatrix}x&0&0\\ y&x^2&u\\ z&0&t\end{pmatrix}, \ \begin{pmatrix}x&y&0\\ z&t&0\\ u&w&xt-yz\end{pmatrix}, \ \begin{pmatrix}x&x-y&0\\ 0&y&0\\ u&v&xy\end{pmatrix};$$

\item For the algebras $\RR41$, $\RR44$, $\RR45$, $\RR46$ and $\RR49$, the corresponding matrices have the following forms, respectively:
\[\begin{pmatrix}x&0&0&0\\ y&x^2&0&0\\ z&2xy&x^3&u\\ t&0&0&v\end{pmatrix}, \ \begin{pmatrix}x&0&0&0\\ y&x^2&0&0\\ z&0&x^2&0\\ t&x(2y+z)&u&x^3\end{pmatrix}, \ \begin{pmatrix}x&0&0&0\\ z&y&0&0\\ 0&0&xy&0\\ t&u&0&x^2y\end{pmatrix},
\]
\[\begin{pmatrix}x&0&0&0\\ y&x^2&0&0\\ -\frac{y^2}{x}&-2xy&x^3&0\\ t&u&-x^2y&x^4\end{pmatrix}, \ \begin{pmatrix}x&0&0&0\\ y&x^2&0&0\\ z&2xy&x^3&0\\ t&y^2+2xz&3x^2y&x^4\end{pmatrix}.
\]

\item For the algebras $\M1$, $\M2$, $\M3$, $\M4(0), \M6, \M7(\gamma), \gamma\neq1, \M7(1), \M{11}, \M{12}, \M{13}$ and $\M{14}(\lambda)$, the corresponding matrices have the following forms, respectively:

\[\begin{pmatrix}x&0&0&0\\ t&x^2&w&g\\ u&0&y&j\\ v&0&h&z\end{pmatrix}, \ \begin{pmatrix}x&0&0&0\\ t&x^2&z&0\\ 0&0&y&0\\ u&0&v&y^2\end{pmatrix}, \ \begin{pmatrix}x&y&0&0\\ t&z&0&0\\ u&w&xz-yt&h\\ v&g&0&j\end{pmatrix}, \ \begin{pmatrix}x&0&0&0\\ y-x&y&0&0\\ u&v&xy&j\\ t&w&0&z\end{pmatrix},
\]
\[\begin{pmatrix}x&0&0&0\\ 0&x&0&0\\ y&z&x^2&0\\ t&u&0&x^2\end{pmatrix}, \ \begin{pmatrix}x&0&0&0\\ 0&x&0&0\\ y&z&x^2&0\\ t&u&0&x^2\end{pmatrix}, \ \begin{pmatrix}x&z&0&0\\ x-y+z&y&0&0\\ u&v&x(y-z)&z(y-z)\\ t&w&(x-y+z)(y-z)&y(y-z)\end{pmatrix},
\]
\[\begin{pmatrix}x&0&0&0\\ 0&y&0&0\\ z&u&xy&0\\ t&v&0&xy\end{pmatrix},\ \begin{pmatrix}x&0&0&0\\ 0&x&0&0\\ y&z&x^2&0\\ u&v&0&x^2\end{pmatrix},
\ \begin{pmatrix}x&y&0&0\\ 0&z&0&0\\ t&u&z^2&0\\ v&w&yz&xz\end{pmatrix},
\ \begin{pmatrix}x&y&0&0\\ 0&z&0&0\\ t&v&z^2&0\\ u&w&(\lambda+1)yz&xz\end{pmatrix}.
\]
\end{itemize}
Moreover, the algebras $\M{2}, \M7(\gamma)), \gamma\notin\{0,1\}, \M{11}$ and $\M{12}$ also have an automorphism defined by the product $\phi\sigma$, where the $\sigma$ are respectively as follows:
\[\begin{pmatrix}0&0&1&0\\ 0&0&0&1\\ 1&0&0&0\\ 0&1&0&0\end{pmatrix}, \ \begin{pmatrix}0&\gamma p&0&0\\ p&0&0&0\\ 0&0&0&-\gamma^2p^2\\ 0&0&-p^2&0\end{pmatrix},\ \begin{pmatrix}0&1&0&0\\ 1&0&0&0\\ 0&0&0&1\\ 0&0&1&0\end{pmatrix}, \begin{pmatrix}0&p&0&0\\ p&0&0&0\\ 0&0&-p^2&2p^2\\ 0&0&0&p^2\end{pmatrix},\ p\neq0.\]

\end{proposition}

\begin{remark}
In each case the number of free parameters of the connected family equals the dimension of the derivation algebra of $\mathcal{A}$, and the solution set of the polynomial system $\phi(e_ie_j)=\phi(e_i)\phi(e_j)$, $\det\phi\neq0$. The additional automorphisms $\sigma$, $\sigma_p$ in (x), (xiv), (xvi) and (xvii) interchange the two generators; they are needed for the isomorphism problem of the extensions of $\M7(\gamma)$, $\M{11}$ and $\M{12}$ (Lemmas~\ref{lem:M7}, \ref{lem:M11} and~\ref{lem:M12}), whereas for $\M2$ the connected component already acts with finitely many orbits (Lemma~\ref{lem:M2}).
\end{remark}
\section{Algebras with two-dimensional annihilator}\label{sec:ann2}

By Proposition~\ref{prop:data3} and Corollary~\ref{cor:empty3}, the five-dimensional algebras of Theorem~A with two-dimensional annihilator are the extensions $\mathcal{A}_W$, $W\in\T_2(\mathcal{A})$, for $\mathcal{A}\in\{\RR31,\RR33,\RR34(0)\}$.
In this section we determine the $\Aut(\mathcal{A})$-orbits on $\T_2(\mathcal{A})$ for these three algebras.
We use the bases $\N1,\N2,\dots$ of $\HH(\mathcal{A})$ in the order of Proposition~\ref{prop:data3}, and the automorphisms of Proposition~\ref{prop:aut}.
A point of $G_2(\HH(\mathcal{A}))$ is written $W=\sspan{\theta_1,\theta_2}$; the new basis vectors of the extension are $e_4$ and $e_5$.

\textbf{Central extensions of $\RR31$.} Let us use the following notations: 
$$\N1=[\D{12}+\D{21}], \ \N2=[\D{13}], \  \N3=[\D{31}], \ \N4=[\D{33}]$$ and $\Ann(\RR31)=\sspan{e_2,e_3}$.

\begin{lemma}\label{lem:R31action}
Let $\theta=\sum_{i=1}^4\alpha_i\N{i}$ and let $\phi\in\Aut(\RR31)$ be as in Proposition~\ref{prop:aut}.
Then $\phi\theta=\sum_i\alpha_i^*\N{i}$ with
\[
\alpha_1^*=\alpha_1x^3,\quad
\alpha_2^*=\alpha_1ux+\alpha_2tx+\alpha_4tz,\quad
\alpha_3^*=\alpha_1ux+\alpha_3tx+\alpha_4tz,\quad
\alpha_4^*=\alpha_4t^2 .
\]
A subspace $W=\sspan{\theta_1,\theta_2}\in G_2(\HH(\RR31))$ lies in $\T_2(\RR31)$ if and only if the coefficient $\alpha_1$ does not vanish identically on $W$.
\end{lemma}

\begin{proof}
The action formulas are obtained by evaluating $\theta(\phi(e_i),\phi(e_j))$ and reducing modulo $\B(\RR31)=\sspan{\D{11}}$.
For the second claim, let $a e_2+be_3\in\Ann(\RR31)$.
For a cocycle $\theta=\alpha_1(\D{12}+\D{21})+\alpha_2\D{13}+\alpha_3\D{31}+\alpha_4\D{33}$ we have
$\theta(e_1,ae_2+be_3)=a\alpha_1+b\alpha_2$, $\theta(ae_2+be_3,e_1)=a\alpha_1+b\alpha_3$, $\theta(e_3,ae_2+be_3)=\theta(ae_2+be_3,e_3)=b\alpha_4$, and all other values vanish.
If $\alpha_1=0$ on $W$, then $e_2\in\Ann(\theta_1)\cap\Ann(\theta_2)\cap\Ann(\RR31)$.
Conversely, suppose $\alpha_1\neq0$ on $W$; choose the basis so that $\theta_1$ has $\alpha_1=1$ and $\theta_2$ has $\alpha_1=0$, $\theta_2\neq0$.
If $ae_2+be_3$ is annihilated by $\theta_2$ then $b\alpha_2=b\alpha_3=b\alpha_4=0$ for the coefficients of $\theta_2$, whence $b=0$; then $\theta_1(e_1,ae_2)=a=0$.
\end{proof}

\begin{proposition}\label{prop:R31}
Every element of $\T_2(\RR31)$ lies in the $\Aut(\RR31)$-orbit of exactly one of the following subspaces:
\[
\begin{array}{lll}
W_1=\sspan{\N1,\N4}, & W_2=\sspan{\N1+\N2,\N4}, & W_3=\sspan{\N1,\N2+\N4},\\
W_4=\sspan{\N1+\N2,\N2+\N4}, & W_5^{\alpha}=\sspan{\N1+\N4,\N2+\alpha\N3}, & W_6=\sspan{\N1+\N4,\N3},\\
W_7=\sspan{\N1+\N2+\N4,\N2+\N3}, & W_8^{\alpha}=\sspan{\N1,\N2+\alpha\N3}, & W_9=\sspan{\N1,\N3},\\
W_{10}=\sspan{\N1+\N2,\N2+\N3}, & &
\end{array}
\]
where $\alpha\in\C$; moreover $W_5^{\alpha}$ and $W_5^{\beta}$ \textup{(}resp.\ $W_8^\alpha$ and $W_8^\beta$\textup{)} lie in the same orbit only if $\alpha=\beta$.
\end{proposition}

\begin{proof}
Consider the linear map $\pi\colon\HH(\RR31)\to\C^2$, $\theta\mapsto(\alpha_1,\alpha_4)$.
By Lemma~\ref{lem:R31action}, $\pi(\phi\theta)=(x^3\alpha_1,t^2\alpha_4)$, so $\dim\pi(W)$, the property ``$\alpha_4\equiv0$ on $W$'' and the subspace $W\cap\ker\pi$ are invariants (the last one up to the action).
By Lemma~\ref{lem:R31action}, $\pi(W)$ contains a vector with $\alpha_1\neq0$.
Note also that the automorphisms with $x=t=1$, $y=0$ act by
\begin{equation}\label{eq:R31shift}
\theta\mapsto\theta+(\alpha_1u+\alpha_4z)(\N2+\N3),
\end{equation}
and the automorphisms with $y=z=u=0$ act diagonally.

\emph{Case 1: $\dim\pi(W)=2$.}
Then $W=\sspan{\N1+a\N2+b\N3,\ \N4+c\N2+d\N3}$.
By \eqref{eq:R31shift} with $u=-b$, $z=-d$ we may assume $b=d=0$.
The diagonal automorphisms send $\sspan{\N1+a\N2,\N4+c\N2}$ to $\sspan{\N1+atx^{-2}\N2,\ \N4+cxt^{-1}\N2}$, and the elements of $\Aut(\RR31)$ preserving the conditions $b=d=0$ have $u=z=0$.
Since $t/x^2$ and $x/t$ can be chosen independently, we obtain exactly the four orbits $W_1,\dots,W_4$ (the conditions $a\neq0$ and $c\neq0$ are invariant: $a\neq 0$ if and only if the $\N2$- and $\N3$-coefficients of the element of $W$ with $\pi=(1,0)$ differ, and similarly for $c$).

\emph{Case 2: $\dim\pi(W)=1$ and $\alpha_4\not\equiv0$ on $W$.}
Then $\pi(W)$ is spanned by $(1,a_4)$ with $a_4\neq0$ and, by a diagonal automorphism with $a_4t^2=x^3$, we may assume $a_4=1$.
Thus $W=\sspan{\N1+\N4+a\N2+b\N3,\ \theta_2}$, where $\theta_2=c\N2+d\N3\neq0$ spans $W\cap\ker\pi$; the point $[c:d]\in\mathbb{P}^1$ is an invariant, because every automorphism multiplies both $c$ and $d$ by $tx$.
If $c\neq d$, then $\N2+\N3$ and $\theta_2$ span $\sspan{\N2,\N3}$, hence, using \eqref{eq:R31shift} and subtracting a multiple of $\theta_2$, we can achieve $a=b=0$; we get $W_5^{\alpha}$ ($\alpha=d/c\neq1$) or $W_6$ ($c=0$).
If $c=d$, then $\theta_2=\N2+\N3$ and we can only achieve $b=0$; the diagonal automorphisms with $t^2=x^3$ multiply $a$ by $tx^{-2}=\pm x^{-1/2}$, so we obtain $W_5^1$ ($a=0$) or $W_7$ ($a\neq0$).
The subspaces $W_5^{1}$ and $W_7$ are in different orbits, since in $W_5^1$ the elements with $\pi=(1,1)$ lie in the span of $\N1+\N4$ and the invariant line $\sspan{\N2+\N3}$, while in $W_7$ they do not, and this property is preserved by the maps \eqref{eq:R31shift} and by the diagonal automorphisms.

\emph{Case 3: $\dim\pi(W)=1$ and $\alpha_4\equiv0$ on $W$.}
Then $W=\sspan{\N1+a\N2+b\N3,\ c\N2+d\N3}$, $(c,d)\neq(0,0)$, and the same argument as in Case~2 (with $z$ playing no role) gives $W_8^{\alpha}$, $W_9$ and $W_{10}$.

The three cases are distinguished by the invariants $\dim\pi(W)$ and $\alpha_4|_W$, and the invariant point $[c:d]$ shows that different values of $\alpha$ give different orbits.
\end{proof}

\begin{corollary}\label{cor:R31}
The non-split central extensions of $\RR31$ by a two-dimensional space are, up to isomorphism, the algebras $\A1,\dots,\A{10}$ below; the algebras $\A5(\alpha)$ \textup{(}resp.\ $\A8(\alpha)$\textup{)} are pairwise non-isomorphic for different values of $\alpha$.
\end{corollary}
\begin{longtable}{@{}l p{0.80\textwidth}@{}}
$\mathcal{A}_{1}$ & $e_{1}e_{1}=e_{2},\ e_{1}e_{2}=e_{4},\ e_{2}e_{1}=e_{4},\ e_{3}e_{3}=e_{5}$ \\ & \quad {\small $\theta=\langle \nabla_{1},\ \nabla_{4}\rangle$} \\[4pt]
$\mathcal{A}_{2}$ & $e_{1}e_{1}=e_{2},\ e_{1}e_{2}=e_{4},\ e_{1}e_{3}=e_{4},\ e_{2}e_{1}=e_{4}, \ e_{3}e_{3}=e_{5}$ \\ & \quad {\small $\theta=\langle \nabla_{1}+\nabla_{2},\ \nabla_{4}\rangle$} \\[4pt]
$\mathcal{A}_{3}$ & $e_{1}e_{1}=e_{2},\ e_{1}e_{2}=e_{4},\ e_{1}e_{3}=e_{5},\ e_{2}e_{1}=e_{4}, \ e_{3}e_{3}=e_{5}$ \\ & \quad {\small $\theta=\langle \nabla_{1},\ \nabla_{2}+\nabla_{4}\rangle$} \\[4pt]
$\mathcal{A}_{4}$ & $e_{1}e_{1}=e_{2},\ e_{1}e_{2}=e_{4},\ e_{1}e_{3}=e_{4}+e_{5},\ e_{2}e_{1}=e_{4}, \ e_{3}e_{3}=e_{5}$ \\ & \quad {\small $\theta=\langle \nabla_{1}+\nabla_{2},\ \nabla_{2}+\nabla_{4}\rangle$} \\[4pt]
$\mathcal{A}_{5}(\alpha)$ & $e_{1}e_{1}=e_{2},\ e_{1}e_{2}=e_{4},\ e_{1}e_{3}=e_{5},\ e_{2}e_{1}=e_{4}, \ e_{3}e_{1}=\alpha e_{5},\ e_{3}e_{3}=e_{4}$ \\ & \quad {\small $\theta=\langle \nabla_{1}+\nabla_{4},\ \nabla_{2}+\alpha \nabla_{3}\rangle$;\quad $\alpha\in\mathbb{C}$} \\[4pt]
$\mathcal{A}_{6}$ & $e_{1}e_{1}=e_{2},\ e_{1}e_{2}=e_{4},\ e_{2}e_{1}=e_{4},\ e_{3}e_{1}=e_{5}, \ e_{3}e_{3}=e_{4}$ \\ & \quad {\small $\theta=\langle \nabla_{1}+\nabla_{4},\ \nabla_{3}\rangle$} \\[4pt]
$\mathcal{A}_{7}$ & $e_{1}e_{1}=e_{2},\ e_{1}e_{2}=e_{4},\ e_{1}e_{3}=e_{4}+e_{5},\ e_{2}e_{1}=e_{4}, \ e_{3}e_{1}=e_{5},\ e_{3}e_{3}=e_{4}$ \\ & \quad {\small $\theta=\langle \nabla_{1}+\nabla_{2}+\nabla_{4},\ \nabla_{2}+\nabla_{3}\rangle$} \\[4pt]
$\mathcal{A}_{8}(\alpha)$ & $e_{1}e_{1}=e_{2},\ e_{1}e_{2}=e_{4},\ e_{1}e_{3}=e_{5},\ e_{2}e_{1}=e_{4}, \ e_{3}e_{1}=\alpha e_{5}$ \\ & \quad {\small $\theta=\langle \nabla_{1},\ \nabla_{2}+\alpha \nabla_{3}\rangle$;\quad $\alpha\in\mathbb{C}$} \\[4pt]
$\mathcal{A}_{9}$ & $e_{1}e_{1}=e_{2},\ e_{1}e_{2}=e_{4},\ e_{2}e_{1}=e_{4},\ e_{3}e_{1}=e_{5}$ \\ & \quad {\small $\theta=\langle \nabla_{1},\ \nabla_{3}\rangle$} \\[4pt]
$\mathcal{A}_{10}$ & $e_{1}e_{1}=e_{2},\ e_{1}e_{2}=e_{4},\ e_{1}e_{3}=e_{4}+e_{5},\ e_{2}e_{1}=e_{4}, \ e_{3}e_{1}=e_{5}$ \\ & \quad {\small $\theta=\langle \nabla_{1}+\nabla_{2},\ \nabla_{2}+\nabla_{3}\rangle$} \\[4pt]
\end{longtable}

\textbf{Central extensions of $\RR33$.}  Let us use the following notations:
 $$\N1=[\D{11}], \ \N2=[\D{12}], \ \N3=[\D{22}], \ \N4=[\D{13}], \ \N5=[\D{23}],$$ $\Ann(\RR33)=\sspan{e_3}$ and $\B(\RR33)=\sspan{\D{12}-\D{21}}$.
Let $V=\sspan{e_1,e_2}$ and let $V^*$ be its dual space with the coordinate functions $v_1,v_2$.
To $\theta=\sum\alpha_i\N{i}$ we associate the pair $(q_\theta,l_\theta)$ consisting of the quadratic form $q_\theta(v)=\alpha_1v_1^2+\alpha_2v_1v_2+\alpha_3v_2^2$ and the linear form $l_\theta(v)=\alpha_4v_1+\alpha_5v_2$ on $V$; this is a linear isomorphism of $\HH(\RR33)$ onto $S^2V^*\oplus V^*$.

\begin{lemma}\label{lem:R33action}
Let $\phi\in\Aut(\RR33)$ be as in Proposition~\ref{prop:aut}, let $P=\begin{pmatrix}x&y\\ z&t\end{pmatrix}$ and put $m(v)=uv_1+wv_2$.
Then
\[
q_{\phi\theta}(v)=q_\theta(Pv)+l_\theta(Pv)\,m(v),\qquad l_{\phi\theta}(v)=\det P\cdot l_\theta(Pv).
\]
A subspace $W\in G_2(\HH(\RR33))$ lies in $\T_2(\RR33)$ if and only if $l_\theta\neq0$ for some $\theta\in W$.
\end{lemma}

\begin{proof}
The formulas follow from the action formulas
$\alpha_1^*=\alpha_1x^2+\alpha_2xz+\alpha_3z^2+\alpha_4ux+\alpha_5uz$,
$\alpha_2^*=2\alpha_1xy+\alpha_2(tx+yz)+2\alpha_3tz+\alpha_4(uy+wx)+\alpha_5(tu+wz)$,
$\alpha_3^*=\alpha_1y^2+\alpha_2ty+\alpha_3t^2+\alpha_4wy+\alpha_5tw$,
$\alpha_4^*=(xt-yz)(\alpha_4x+\alpha_5z)$,
$\alpha_5^*=(xt-yz)(\alpha_4y+\alpha_5t)$, which are obtained directly (note that $\D{21}\equiv\D{12}$ modulo $\B(\RR33)$).
The only values of $\theta$ involving $e_3$ are $\theta(e_1,e_3)=\alpha_4$ and $\theta(e_2,e_3)=\alpha_5$, which proves the last claim.
\end{proof}

\begin{proposition}\label{prop:R33}
Every element of $\T_2(\RR33)$ lies in the orbit of exactly one of the subspaces
\[
\begin{array}{lllll}
\sspan{\N4,\N5}, & \sspan{\N4,\N1+\N5}, & \sspan{\N4,\N2+\N5}, & \sspan{\N3+\N4,\N1+\N5}, & \sspan{\N4,\N1},\\
\sspan{\N3+\N4,\N1}, & \sspan{\N4,\N2}, & \sspan{\N3+\N4,\N2}, & \sspan{\N4,\N3}, & \sspan{\N4,\N1+\N3}.
\end{array}
\]
\end{proposition}

\begin{proof}
Let $L\colon W\to V^*$, $\theta\mapsto l_\theta$.
By Lemma~\ref{lem:R33action}, $\rank L\in\{1,2\}$ is an invariant.

\emph{Case I: $\rank L=2$.}
Then $W=\{(Q(l),l): l\in V^*\}$ for a unique linear map $Q\colon V^*\to S^2V^*$.
Put
\[
f_W(v)=v_1\,Q(v_2)(v)-v_2\,Q(v_1)(v),
\]
a binary cubic form.
The map $Q\mapsto f_W$ from $\Hom(V^*,S^2V^*)$ (dimension $6$) to the space of binary cubic forms (dimension $4$) is surjective, and its kernel consists of the maps $Q(l)=l\cdot m$ with $m\in V^*$: indeed, $v_1Q(v_2)=v_2Q(v_1)$ forces $Q(v_1)=v_1m$ and $Q(v_2)=v_2m$ for some linear form $m$.
By Lemma~\ref{lem:R33action}, the automorphisms with $P=1$ change $Q$ exactly by such maps $l\mapsto l\cdot m$, so $W$ is determined, up to these automorphisms, by $f_W$.
Moreover, $f_{\phi W}(v)=(\det P)^{-2}f_W(Pv)$.
Indeed, by Lemma~\ref{lem:R33action}, $\phi W$ is the graph of the map $Q'(l')=(\det P)^{-1}\,Q\big(l'\circ P^{-1}\big)\circ P+(\det P)^{-1}l'\,m$; the last summand does not contribute to $f_{\phi W}$, and for the basis $a=v_1\circ P^{-1}$, $b=v_2\circ P^{-1}$ of $V^*$ we have $aQ(b)-bQ(a)=\det(P^{-1})\,(v_1Q(v_2)-v_2Q(v_1))$, because the expression $aQ(b)-bQ(a)$ is bilinear and alternating in $(a,b)$.
Evaluating at $Pv$ gives the claim.
Hence the orbits of Case I correspond bijectively to the orbits of binary cubic forms under $f\mapsto c\,f\circ P$ ($c\in\C^*$, $P\in GL_2(\C)$; note that $P=\lambda I$ multiplies $f$ by $\lambda^{-1}$).
There are exactly four such orbits: $f=0$, and $f\neq0$ with a triple root, with a double and a simple root, or with three simple roots in $\mathbb{P}^1$.
For $\sspan{\N4,\N5}$, $\sspan{\N4,\N1+\N5}$, $\sspan{\N4,\N2+\N5}$ and $\sspan{\N3+\N4,\N1+\N5}$ we have $Q(v_1)=0,0,0,v_2^2$ and $Q(v_2)=0,v_1^2,v_1v_2,v_1^2$, so that $f_W=0,\ v_1^3,\ v_1^2v_2,\ v_1^3-v_2^3$, respectively.

\emph{Case II: $\rank L=1$.}
Then $W=\sspan{\theta_1,\theta_2}$ with $l_{\theta_1}\neq0$ and $l_{\theta_2}=0$, so $q_{\theta_2}\neq0$ and the line $\sspan{\theta_2}$ is canonical.
Applying a suitable $P$, we may assume $l_{\theta_1}=v_1$, i.e.\ $\theta_1=\N4+(\text{terms in }\N1,\N2,\N3)$.
The automorphisms preserving the condition $l_{\theta_1}\in\sspan{v_1}$ are those with $y=0$; by Lemma~\ref{lem:R33action} they act on $q_{\theta_2}$ by $q\mapsto q\circ P$, $P=\begin{pmatrix}x&0\\z&t\end{pmatrix}$, and they allow us to change $q_{\theta_1}$ by any multiple of $v_1$ (using $u,w$) and by multiples of $q_{\theta_2}$ (change of basis of $W$).
Note that the property ``$v_1$ divides $q_{\theta_2}$'' is invariant.

(a) If $v_1\nmid q_{\theta_2}$, then the $v_2^2$-coefficient of $q_{\theta_2}$ is nonzero and $q_{\theta_1}$ can be made zero, so $\theta_1=\N4$.
Using $z$ we complete the square in $q_{\theta_2}$, which becomes $av_1^2+v_2^2$ up to a scalar; $a=0$ gives $\sspan{\N4,\N3}$ and $a\neq0$ gives, after scaling $v_1$, $\sspan{\N4,\N1+\N3}$.
These are different orbits since $q_{\theta_2}$ is degenerate in the first case only.

(b) If $v_1\mid q_{\theta_2}$, then $q_{\theta_2}=v_1(av_1+bv_2)$; if $b\neq0$ we can make $a=0$ by means of $z$, so $q_{\theta_2}\in\{v_1v_2,v_1^2\}$ up to scalars, and the rank of $q_{\theta_2}$ is invariant.
Now $q_{\theta_1}$ is determined modulo $v_1V^*$, i.e.\ by its $v_2^2$-coefficient, which is multiplied by $t^2\cdot(\text{nonzero})$; we get $\theta_1\in\{\N4,\N3+\N4\}$.
This gives the four orbits $\sspan{\N4,\N2}$, $\sspan{\N3+\N4,\N2}$, $\sspan{\N4,\N1}$, $\sspan{\N3+\N4,\N1}$.
\end{proof}

\begin{corollary}\label{cor:R33}
The non-split central extensions of $\RR33$ by a two-dimensional space are, up to isomorphism, the algebras $\A{11},\dots,\A{20}$ below.
\end{corollary}
\begin{longtable}{@{}l p{0.80\textwidth}@{}}
$\mathcal{A}_{11}$ & $e_{1}e_{2}=e_{3},\ e_{1}e_{3}=e_{4},\ e_{2}e_{1}=-e_{3},\ e_{2}e_{3}=e_{5}$ \\ & \quad {\small $\theta=\langle \nabla_{4},\ \nabla_{5}\rangle$} \\[4pt]
$\mathcal{A}_{12}$ & $e_{1}e_{1}=e_{5},\ e_{1}e_{2}=e_{3},\ e_{1}e_{3}=e_{4},\ e_{2}e_{1}=-e_{3}, \ e_{2}e_{3}=e_{5}$ \\ & \quad {\small $\theta=\langle \nabla_{4},\ \nabla_{1}+\nabla_{5}\rangle$} \\[4pt]
$\mathcal{A}_{13}$ & $e_{1}e_{2}=e_{3}+e_{5},\ e_{1}e_{3}=e_{4},\ e_{2}e_{1}=-e_{3},\ e_{2}e_{3}=e_{5}$ \\ & \quad {\small $\theta=\langle \nabla_{4},\ \nabla_{2}+\nabla_{5}\rangle$} \\[4pt]
$\mathcal{A}_{14}$ & $e_{1}e_{1}=e_{5},\ e_{1}e_{2}=e_{3},\ e_{1}e_{3}=e_{4},\ e_{2}e_{1}=-e_{3}, \ e_{2}e_{2}=e_{4},\ e_{2}e_{3}=e_{5}$ \\ & \quad {\small $\theta=\langle \nabla_{3}+\nabla_{4},\ \nabla_{1}+\nabla_{5}\rangle$} \\[4pt]
$\mathcal{A}_{15}$ & $e_{1}e_{1}=e_{5},\ e_{1}e_{2}=e_{3},\ e_{1}e_{3}=e_{4},\ e_{2}e_{1}=-e_{3}$ \\ & \quad {\small $\theta=\langle \nabla_{4},\ \nabla_{1}\rangle$} \\[4pt]
$\mathcal{A}_{16}$ & $e_{1}e_{1}=e_{5},\ e_{1}e_{2}=e_{3},\ e_{1}e_{3}=e_{4},\ e_{2}e_{1}=-e_{3}, \ e_{2}e_{2}=e_{4}$ \\ & \quad {\small $\theta=\langle \nabla_{3}+\nabla_{4},\ \nabla_{1}\rangle$} \\[4pt]
$\mathcal{A}_{17}$ & $e_{1}e_{2}=e_{3}+e_{5},\ e_{1}e_{3}=e_{4},\ e_{2}e_{1}=-e_{3}$ \\ & \quad {\small $\theta=\langle \nabla_{4},\ \nabla_{2}\rangle$} \\[4pt]
$\mathcal{A}_{18}$ & $e_{1}e_{2}=e_{3}+e_{5},\ e_{1}e_{3}=e_{4},\ e_{2}e_{1}=-e_{3},\ e_{2}e_{2}=e_{4}$ \\ & \quad {\small $\theta=\langle \nabla_{3}+\nabla_{4},\ \nabla_{2}\rangle$} \\[4pt]
$\mathcal{A}_{19}$ & $e_{1}e_{2}=e_{3},\ e_{1}e_{3}=e_{4},\ e_{2}e_{1}=-e_{3},\ e_{2}e_{2}=e_{5}$ \\ & \quad {\small $\theta=\langle \nabla_{4},\ \nabla_{3}\rangle$} \\[4pt]
$\mathcal{A}_{20}$ & $e_{1}e_{1}=e_{5},\ e_{1}e_{2}=e_{3},\ e_{1}e_{3}=e_{4},\ e_{2}e_{1}=-e_{3}, \ e_{2}e_{2}=e_{5}$ \\ & \quad {\small $\theta=\langle \nabla_{4},\ \nabla_{1}+\nabla_{3}\rangle$} \\[4pt]
\end{longtable}

\begin{remark}
The cubic form $f_W$ of Case~I is a complete invariant of the orbits of $\T_2(\RR33)$ on which $L$ is injective; in particular the corresponding four orbits are distinguished by the root multiplicities of $f_W$.
\end{remark}

\textbf{Central extensions of $\RR34(0)$.}   Let us use the following notations:
$$\N1=[\D{11}], \ \N2=[\D{12}], \ \N3=[\D{22}], \N4=[\D{23}+\D{32}]$$ and $\Ann(\RR34(0))=\sspan{e_3}$.

\begin{lemma}\label{lem:R340action}
For $\phi\in\Aut(\RR34(0))$ as in Proposition~\ref{prop:aut} we have
\[
\begin{array}{ll}
\alpha_1^*=\alpha_1x^2, & \alpha_2^*=\alpha_1x(x-y)+\alpha_2xy+\alpha_4uy,\\
\alpha_3^*=\alpha_1y(y-x)+\alpha_2y(x-y)+\alpha_3y^2-\alpha_4y(u-2v), & \alpha_4^*=\alpha_4xy^2 .
\end{array}
\]
A subspace $W\in G_2(\HH(\RR34(0)))$ lies in $\T_2$ if and only if $\alpha_4\not\equiv0$ on $W$.
\end{lemma}

\begin{proof}
The formulas are obtained by a direct computation modulo $\B=\sspan{\D{21}+\D{22}}$.
The vector $e_3$ appears only in $\N4$, via $\theta(e_2,e_3)=\theta(e_3,e_2)=\alpha_4$.
\end{proof}

\begin{proposition}\label{prop:R340}
Every element of $\T_2(\RR34(0))$ lies in the orbit of exactly one of the subspaces
\[
\begin{array}{lll}
\sspan{\N4,\N1+\alpha\N3}\ (\alpha\in\C), & \sspan{\N4,\N1+\N2}, & \sspan{\N4,\N1+\N2+\N3},\\
\sspan{\N4,\N2}, & \sspan{\N1+\N4,\N2}, & \sspan{\N4,\N2+\N3},\\
\sspan{\N1+\N4,\N2+\N3}, & \sspan{\N4,\N3}, & \sspan{\N1+\N4,\N3},
\end{array}
\]
and different values of $\alpha$ give different orbits.
\end{proposition}

\begin{proof}
Let $W\in\T_2$.
Then $W\cap\{\alpha_4=0\}=\sspan{\theta_2}$ is a canonical line, $\theta_2=b\N1+c\N2+d\N3\neq0$, and on it the automorphisms act through $x$ and $y$ only.
Write $W=\sspan{\theta_1,\theta_2}$ with $\theta_1=\N4+a_1\N1+a_2\N2+a_3\N3$.
By Lemma~\ref{lem:R340action}, the parameters $u,v$ allow us to make $a_2=a_3=0$ (the map $(u,v)\mapsto(uy,-y(u-2v))$ is bijective), and they do not affect $\theta_2$; the $\N1$-coefficient of $\theta_1$ is changed only by adding multiples of $\theta_2$.
Therefore: if $b\neq0$, we may assume $\theta_1=\N4$; if $b=0$, then $\theta_1=\N4+a\N1$ and $a$ is multiplied by $x^2/(xy^2)=x/y^2$.

Put $s=y/x$.
\emph{Case $b\neq0$.} Take $b=1$. The action on $\theta_2$ gives, after division by $x^2$,
\[
\N1+c\N2+d\N3\ \longmapsto\ \N1+\big(1+(c-1)s\big)\N2+\big((c-1)s+(1-c+d)s^2\big)\N3 .
\]
If $c\neq1$ we choose $s=1/(1-c)$ and get $\theta_2=\N1+\alpha\N3$ with $\alpha=-1+\frac{1-c+d}{(1-c)^2}$; the forms with $c=0$ are preserved only by $s=1$, which does not change $d$, so $\alpha$ is an invariant.
If $c=1$ (an invariant condition), then $d\mapsto ds^2$ and we get $\N1+\N2$ or $\N1+\N2+\N3$.

\emph{Case $b=0$.} If $c\neq0$ we take $c=1$ and $\theta_2\mapsto\N2+(1+(d-1)s)\N3$.
If $d\neq1$ we get $\theta_2=\N2$, and the stabilizer of this form has $s=1$; then $a\mapsto a/x$, giving $\sspan{\N4,\N2}$ and $\sspan{\N1+\N4,\N2}$.
If $d=1$, $\theta_2=\N2+\N3$ is fixed for all $s$ and $a\mapsto ax/y^2$, giving $\sspan{\N4,\N2+\N3}$ and $\sspan{\N1+\N4,\N2+\N3}$.
If $c=0$ then $\theta_2=\N3$ and, similarly, $\theta_1\in\{\N4,\N1+\N4\}$.
The invariants used ($b=0$ or not, the orbit of the line $\sspan{\theta_2}$, and whether $a=0$) separate all the listed orbits.
\end{proof}

\begin{corollary}\label{cor:R340}
The non-split central extensions of $\RR34(0)$ by a two-dimensional space are, up to isomorphism, the algebras $\A{21},\dots,\A{29}$ below.
\end{corollary}
\begin{longtable}{@{}l p{0.80\textwidth}@{}}
$\mathcal{A}_{21}(\alpha)$ & $e_{1}e_{1}=e_{5},\ e_{2}e_{1}=e_{3},\ e_{2}e_{2}=e_{3}+\alpha e_{5},\ e_{2}e_{3}=e_{4}, \ e_{3}e_{2}=e_{4}$ \\ & \quad {\small $\theta=\langle \nabla_{4},\ \nabla_{1}+\alpha \nabla_{3}\rangle$;\quad $\alpha\in\mathbb{C}$} \\[4pt]
$\mathcal{A}_{22}$ & $e_{1}e_{1}=e_{5},\ e_{1}e_{2}=e_{5},\ e_{2}e_{1}=e_{3},\ e_{2}e_{2}=e_{3}, \ e_{2}e_{3}=e_{4},\ e_{3}e_{2}=e_{4}$ \\ & \quad {\small $\theta=\langle \nabla_{4},\ \nabla_{1}+\nabla_{2}\rangle$} \\[4pt]
$\mathcal{A}_{23}$ & $e_{1}e_{1}=e_{5},\ e_{1}e_{2}=e_{5},\ e_{2}e_{1}=e_{3},\ e_{2}e_{2}=e_{3}+e_{5}, \ e_{2}e_{3}=e_{4},\ e_{3}e_{2}=e_{4}$ \\ & \quad {\small $\theta=\langle \nabla_{4},\ \nabla_{1}+\nabla_{2}+\nabla_{3}\rangle$} \\[4pt]
$\mathcal{A}_{24}$ & $e_{1}e_{2}=e_{5},\ e_{2}e_{1}=e_{3},\ e_{2}e_{2}=e_{3},\ e_{2}e_{3}=e_{4}, \ e_{3}e_{2}=e_{4}$ \\ & \quad {\small $\theta=\langle \nabla_{4},\ \nabla_{2}\rangle$} \\[4pt]
$\mathcal{A}_{25}$ & $e_{1}e_{1}=e_{4},\ e_{1}e_{2}=e_{5},\ e_{2}e_{1}=e_{3},\ e_{2}e_{2}=e_{3}, \ e_{2}e_{3}=e_{4},\ e_{3}e_{2}=e_{4}$ \\ & \quad {\small $\theta=\langle \nabla_{1}+\nabla_{4},\ \nabla_{2}\rangle$} \\[4pt]
$\mathcal{A}_{26}$ & $e_{1}e_{2}=e_{5},\ e_{2}e_{1}=e_{3},\ e_{2}e_{2}=e_{3}+e_{5},\ e_{2}e_{3}=e_{4}, \ e_{3}e_{2}=e_{4}$ \\ & \quad {\small $\theta=\langle \nabla_{4},\ \nabla_{2}+\nabla_{3}\rangle$} \\[4pt]
$\mathcal{A}_{27}$ & $e_{1}e_{1}=e_{4},\ e_{1}e_{2}=e_{5},\ e_{2}e_{1}=e_{3},\ e_{2}e_{2}=e_{3}+e_{5}, \ e_{2}e_{3}=e_{4},\ e_{3}e_{2}=e_{4}$ \\ & \quad {\small $\theta=\langle \nabla_{1}+\nabla_{4},\ \nabla_{2}+\nabla_{3}\rangle$} \\[4pt]
$\mathcal{A}_{28}$ & $e_{2}e_{1}=e_{3},\ e_{2}e_{2}=e_{3}+e_{5},\ e_{2}e_{3}=e_{4},\ e_{3}e_{2}=e_{4}$ \\ & \quad {\small $\theta=\langle \nabla_{4},\ \nabla_{3}\rangle$} \\[4pt]
$\mathcal{A}_{29}$ & $e_{1}e_{1}=e_{4},\ e_{2}e_{1}=e_{3},\ e_{2}e_{2}=e_{3}+e_{5},\ e_{2}e_{3}=e_{4}, \ e_{3}e_{2}=e_{4}$ \\ & \quad {\small $\theta=\langle \nabla_{1}+\nabla_{4},\ \nabla_{3}\rangle$} \\[4pt]
\end{longtable}
\section{One-dimensional extensions of the pure four-dimensional algebras}\label{sec:pure}

By Corollary~\ref{cor:empty4}, among the pure four-dimensional algebras only $\RR41$, $\RR44$, $\RR45$, $\RR46$ and $\RR49$ have non-split one-dimensional central extensions.
For $s=1$ we have $\T_1(\mathcal{A})=\{\sspan{\theta}: \Ann(\theta)\cap\Ann(\mathcal{A})=0\}$, and we call a class $\theta$ \emph{admissible} if $\sspan\theta\in\T_1(\mathcal{A})$.
We use the bases of Proposition~\ref{prop:data4pure} and the automorphisms of Proposition~\ref{prop:aut}; $\theta=\sum\alpha_i\N{i}$ and $\phi\theta=\sum\alpha_i^*\N{i}$.
Since we work in $G_1(\HH(\mathcal{A}))$, the class $\theta$ may be multiplied by a nonzero scalar.

\textbf{Central extensions of $\RR41$.}
Let us use the following notations: $$\N1=[\D{14}], \ \N2=[\D{13}+\D{22}+\D{31}], \ \N3=[\D{41}], \ \N4=[\D{44}]$$ and
\[
\alpha_1^*=\alpha_1vx+\alpha_2ux+\alpha_4tv,\quad \alpha_2^*=\alpha_2x^4,\quad \alpha_3^*=\alpha_2ux+\alpha_3vx+\alpha_4tv,\quad\alpha_4^*=\alpha_4v^2 .
\]

\begin{proposition}\label{prop:R41}
A class $\theta$ is admissible if and only if $\alpha_2\neq0$ and $(\alpha_1-\alpha_3,\alpha_4)\neq(0,0)$.
The admissible classes form the three orbits of $\N1+\N2+\N4$, $\N1+\N2$ and $\N2+\N4$.
\end{proposition}

\begin{proof}
Let $ae_3+be_4\in\Ann(\RR41)$.
We have $\theta(e_1,ae_3+be_4)=a\alpha_2+b\alpha_1$, $\theta(ae_3+be_4,e_1)=a\alpha_2+b\alpha_3$, $\theta(e_4,ae_3+be_4)=\theta(ae_3+be_4,e_4)=b\alpha_4$, and the other values vanish.
If $\alpha_2=0$ then $e_3\in\Ann(\theta)$; if $\alpha_4=0$ and $\alpha_1=\alpha_3$, then $\alpha_2e_4-\alpha_1e_3\in\Ann(\theta)$.
Conversely, under the stated conditions the equations force $b=0$ and then $a=0$.

Let $\theta$ be admissible.
The parameters $u,t$ add the same quantity to $\alpha_1^*$ and $\alpha_3^*$, and $\alpha^*_1-\alpha^*_3=(\alpha_1-\alpha_3)vx$.
Choosing $u$ we get $\alpha_3=0$, and the stabilizer of this normalization acts by $\alpha_1\mapsto\alpha_1vx$, $\alpha_2\mapsto\alpha_2x^4$, $\alpha_4\mapsto\alpha_4v^2$ (after renormalization).
If $\alpha_4\neq0$ we choose $v$ with $\alpha_4v^2=\alpha_2x^4$ (i.e.\ $\alpha_4=\alpha_2$), and then $\alpha_1/\alpha_2\mapsto\pm\alpha_1/(\alpha_2 x)$; we obtain $\N1+\N2+\N4$ or $\N2+\N4$.
If $\alpha_4=0$ then $\alpha_1\neq0$ and we obtain $\N1+\N2$.
The three classes are separated by the invariant conditions $\alpha_4=0$ and $\alpha_1=\alpha_3$.
\end{proof}

\begin{longtable}{@{}l p{0.90\textwidth}@{}}
$\mathcal{A}_{30}$ & $e_{1}e_{1}=e_{2},\ e_{1}e_{2}=e_{3},\ e_{1}e_{3}=e_{5},\ e_{1}e_{4}=e_{5}, \ e_{2}e_{1}=e_{3},\ e_{2}e_{2}=e_{5},\ e_{3}e_{1}=e_{5},\ e_{4}e_{4}=e_{5}$ \\ & \quad {\small $\theta=\langle \nabla_{1}+\nabla_{2}+\nabla_{4}\rangle$} \\[4pt]
$\mathcal{A}_{31}$ & $e_{1}e_{1}=e_{2},\ e_{1}e_{2}=e_{3},\ e_{1}e_{3}=e_{5},\ e_{1}e_{4}=e_{5}, \ e_{2}e_{1}=e_{3},\ e_{2}e_{2}=e_{5},\ e_{3}e_{1}=e_{5}$ \\ & \quad {\small $\theta=\langle \nabla_{1}+\nabla_{2}\rangle$} \\[4pt]
$\mathcal{A}_{32}$ & $e_{1}e_{1}=e_{2},\ e_{1}e_{2}=e_{3},\ e_{1}e_{3}=e_{5},\ e_{2}e_{1}=e_{3}, \ e_{2}e_{2}=e_{5},\ e_{3}e_{1}=e_{5},\ e_{4}e_{4}=e_{5}$ \\ & \quad {\small $\theta=\langle \nabla_{2}+\nabla_{4}\rangle$} \\[4pt]
\end{longtable}

\textbf{Central extensions of $\RR44$.} Let us use the following notations:
$$\N1=[\D{13}], \ \N2=[\D{12}+\D{21}], \ \N3=[\D{33}], \ \N4=[\D{14}+\D{22}+\D{23}+\D{32}+\D{41}]$$ and
\[
\begin{array}{ll}
\alpha_1^*=x(\alpha_1x^2+\alpha_3xz+\alpha_4u+\alpha_4xy), & \alpha_2^*=x(\alpha_2x^2-\alpha_3xz-\alpha_4u+2\alpha_4xy+2\alpha_4xz),\\
\alpha_3^*=\alpha_3x^4, & \alpha_4^*=\alpha_4x^4 .
\end{array}
\]

\begin{proposition}\label{prop:R44}
A class $\theta$ is admissible if and only if $\alpha_4\neq0$.
Every admissible class lies in the orbit of exactly one class $\alpha\N3+\N4$, $\alpha\in\C$.
\end{proposition}

\begin{proof}
The vector $e_4$ spans $\Ann(\RR44)$, and it occurs only in $\N4$ (through $\theta(e_1,e_4)=\theta(e_4,e_1)=\alpha_4$).
If $\alpha_4\neq0$, choose $u$ with $\alpha_1^*=0$; then $\alpha_2^*=x^2\big((\alpha_1+\alpha_2)x+3\alpha_4y+2\alpha_4z\big)$, which vanishes for a suitable $y$.
Since $\alpha_3^*/\alpha_4^*=\alpha_3/\alpha_4$, the number $\alpha=\alpha_3/\alpha_4$ is an invariant.
\end{proof}

\begin{longtable}{@{}l p{0.95\textwidth}@{}}
$\mathcal{A}_{33}(\alpha)$ & $e_{1}e_{1}=e_{2},\ e_{1}e_{2}=e_{4},\ e_{1}e_{4}=e_{5},\ e_{2}e_{1}=e_{4}, \ e_{2}e_{2}=e_{5}$ \\ & $e_{2}e_{3}=e_{5}, \ e_{3}e_{1}=e_{4},\ e_{3}e_{2}=e_{5}, \ e_{3}e_{3}=\alpha e_{5},\ e_{4}e_{1}=e_{5}$ \\ & \quad {\small $\theta=\langle \alpha \nabla_{3}+\nabla_{4}\rangle$;\quad $\alpha\in\mathbb{C}$} \\[4pt]
\end{longtable}

\textbf{Central extensions of $\RR45$.}  Let us use the following notations:
$$\N1=[\D{11}], \ \N2=[\D{12}], \ \N3=[\D{22}], \ \N4=[\D{14}+\D{41}], \ \N5=[\D{23}], \ \N6=[\D{24}-\D{33}+\D{42}]$$ and
\[
\begin{array}{ll}
\alpha_1^*=\alpha_1x^2+\alpha_2xz+\alpha_3z^2+2\alpha_4tx+2\alpha_6tz, &
\alpha_2^*=\alpha_2xy+2\alpha_3yz+2\alpha_4ux+2\alpha_6ty+2\alpha_6uz,\\
\alpha_3^*=y(\alpha_3y+2\alpha_6u), &
\alpha_4^*=x^2y(\alpha_4x+\alpha_6z),\\
\alpha_5^*=\alpha_5xy^2, &
\alpha_6^*=\alpha_6x^2y^2 .
\end{array}
\]

\begin{proposition}\label{prop:R45}
A class $\theta$ is admissible if and only if $(\alpha_4,\alpha_6)\neq(0,0)$.
Every admissible class lies in the orbit of exactly one of
\[
\N6,\ \N1+\N6,\ \N5+\N6,\ \N1+\N5+\N6,\ \N4,\ \N3+\N4,\ \N4+\N5,\ \N3+\N4+\N5.
\]
\end{proposition}

\begin{proof}
The vector $e_4$ occurs only in $\N4$ and $\N6$ (via $\theta(e_1,e_4)=\theta(e_4,e_1)=\alpha_4$ and $\theta(e_2,e_4)=\theta(e_4,e_2)=\alpha_6$), which gives the first claim.
The condition $\alpha_6=0$ is invariant.

\emph{Case $\alpha_6\neq0$.}
Choose $z$ with $\alpha_4x+\alpha_6z=0$, so that $\alpha_4^*=0$; then $\alpha_1^*$ no longer depends on $t$.
Next choose $u$ with $\alpha_3^*=0$, and finally $t$ with $\alpha_2^*=0$ (the coefficient of $t$ in $\alpha_2^*$ is $2\alpha_6y\neq0$).
Thus $\theta=\alpha_1\N1+\alpha_5\N5+\alpha_6\N6$, and the automorphisms preserving this form have $z=u=t=0$ and act by $\alpha_1\mapsto\alpha_1x^2$, $\alpha_5\mapsto\alpha_5xy^2$, $\alpha_6\mapsto\alpha_6x^2y^2$.
The ratios $\alpha_1/\alpha_6\mapsto(\alpha_1/\alpha_6)y^{-2}$ and $\alpha_5/\alpha_6\mapsto(\alpha_5/\alpha_6)x^{-1}$ can be normalized independently, and whether they vanish is invariant.
We obtain the first four classes.

\emph{Case $\alpha_6=0$, $\alpha_4\neq0$.}
Choose $t$ and $u$ with $\alpha_1^*=\alpha_2^*=0$.
The remaining coefficients transform by $\alpha_3\mapsto\alpha_3y^2$, $\alpha_4\mapsto\alpha_4x^3y$, $\alpha_5\mapsto\alpha_5xy^2$, so $\alpha_3/\alpha_4$ and $\alpha_5/\alpha_4$ are multiplied by $y/x^3$ and $y/x^2$, which can be prescribed independently.
We obtain the last four classes.
\end{proof}

\begin{longtable}{@{}l p{0.90\textwidth}@{}}
$\mathcal{A}_{34}$ & $e_{1}e_{2}=e_{3},\ e_{1}e_{3}=e_{4},\ e_{2}e_{1}=-e_{3},\ e_{2}e_{4}=e_{5}, \ e_{3}e_{3}=-e_{5},\ e_{4}e_{2}=e_{5}$ \\ & \quad {\small $\theta=\langle \nabla_{6}\rangle$} \\[4pt]
$\mathcal{A}_{35}$ & $e_{1}e_{1}=e_{5},\ e_{1}e_{2}=e_{3},\ e_{1}e_{3}=e_{4},\ e_{2}e_{1}=-e_{3}, \ e_{2}e_{4}=e_{5},\ e_{3}e_{3}=-e_{5},\ e_{4}e_{2}=e_{5}$ \\ & \quad {\small $\theta=\langle \nabla_{1}+\nabla_{6}\rangle$} \\[4pt]
$\mathcal{A}_{36}$ & $e_{1}e_{2}=e_{3},\ e_{1}e_{3}=e_{4},\ e_{2}e_{1}=-e_{3},\ e_{2}e_{3}=e_{5}, \ e_{2}e_{4}=e_{5},\ e_{3}e_{3}=-e_{5},\ e_{4}e_{2}=e_{5}$ \\ & \quad {\small $\theta=\langle \nabla_{5}+\nabla_{6}\rangle$} \\[4pt]
$\mathcal{A}_{37}$ & $e_{1}e_{1}=e_{5},\ e_{1}e_{2}=e_{3},\ e_{1}e_{3}=e_{4},\ e_{2}e_{1}=-e_{3}, \ e_{2}e_{3}=e_{5},\ e_{2}e_{4}=e_{5},\ e_{3}e_{3}=-e_{5},\ e_{4}e_{2}=e_{5}$ \\ & \quad {\small $\theta=\langle \nabla_{1}+\nabla_{5}+\nabla_{6}\rangle$} \\[4pt]
$\mathcal{A}_{38}$ & $e_{1}e_{2}=e_{3},\ e_{1}e_{3}=e_{4},\ e_{1}e_{4}=e_{5},\ e_{2}e_{1}=-e_{3}, \ e_{4}e_{1}=e_{5}$ \\ & \quad {\small $\theta=\langle \nabla_{4}\rangle$} \\[4pt]
$\mathcal{A}_{39}$ & $e_{1}e_{2}=e_{3},\ e_{1}e_{3}=e_{4},\ e_{1}e_{4}=e_{5},\ e_{2}e_{1}=-e_{3}, \ e_{2}e_{2}=e_{5},\ e_{4}e_{1}=e_{5}$ \\ & \quad {\small $\theta=\langle \nabla_{3}+\nabla_{4}\rangle$} \\[4pt]
$\mathcal{A}_{40}$ & $e_{1}e_{2}=e_{3},\ e_{1}e_{3}=e_{4},\ e_{1}e_{4}=e_{5},\ e_{2}e_{1}=-e_{3}, \ e_{2}e_{3}=e_{5},\ e_{4}e_{1}=e_{5}$ \\ & \quad {\small $\theta=\langle \nabla_{4}+\nabla_{5}\rangle$} \\[4pt]
$\mathcal{A}_{41}$ & $e_{1}e_{2}=e_{3},\ e_{1}e_{3}=e_{4},\ e_{1}e_{4}=e_{5},\ e_{2}e_{1}=-e_{3}, \ e_{2}e_{2}=e_{5},\ e_{2}e_{3}=e_{5},\ e_{4}e_{1}=e_{5}$ \\ & \quad {\small $\theta=\langle \nabla_{3}+\nabla_{4}+\nabla_{5}\rangle$} \\[4pt]
\end{longtable}

\textbf{Central extensions of $\RR46$.} Let us use the following notations: $$\N1=[\D{11}], \ \N2=[\D{12}], \ \N3=[\D{13}], \ \N4=[\D{23}], \ \N5=[\D{14}+\D{32}+\D{41}], \ \N6=[\D{24}-\D{33}+\D{42}]$$ and
\[
\begin{array}{l}
\alpha_1^*=x^{-2}\big(\alpha_1x^4+\alpha_2x^3y-\alpha_3x^2y^2-\alpha_4xy^3+2\alpha_5tx^3-\alpha_5xy^3+2\alpha_6tx^2y-\alpha_6y^4\big),\\
\alpha_2^*=\alpha_2x^3-2\alpha_3x^2y-3\alpha_4xy^2+2\alpha_5ux-3\alpha_5xy^2+2\alpha_6tx^2+2\alpha_6uy-4\alpha_6y^3,\\
\alpha_3^*=x^2\big(\alpha_3x^2+3\alpha_4xy+\alpha_5xy-2\alpha_6u+4\alpha_6y^2\big),\\
\alpha_4^*=x^4(\alpha_4x+\alpha_6y),\qquad \alpha_5^*=x^4(\alpha_5x+\alpha_6y),\qquad \alpha_6^*=\alpha_6x^6 .
\end{array}
\]

\begin{proposition}\label{prop:R46}
A class $\theta$ is admissible if and only if $(\alpha_5,\alpha_6)\neq(0,0)$.
Every admissible class lies in the orbit of exactly one of
\[
\N6,\ \N1+\N6,\ \alpha\N1+\N4+\N6,\ \alpha\N4+\N5,\ \N3-\tfrac13\N4+\N5 \ (\alpha\in\C),
\]
and different values of $\alpha$ give different orbits.
\end{proposition}

\begin{proof}
The vector $e_4$ occurs only in $\N5$ and $\N6$, which gives the first claim; the condition $\alpha_6=0$ is invariant.

\emph{Case $\alpha_6\neq0$.}
Choose $y$ with $\alpha_5x+\alpha_6y=0$; then $\alpha_5^*=0$ and $\alpha_4^*=x^5(\alpha_4-\alpha_5)$.
With this choice, the coefficients of $t$ in $\alpha_1^*$ and of $u$ in $\alpha_2^*$ vanish, while $\alpha_3^*$ contains $-2\alpha_6x^2u$ and $\alpha_2^*$ contains $2\alpha_6x^2t$.
Hence we may assume $\alpha_2=\alpha_3=\alpha_5=0$.
The automorphisms preserving these conditions have $y=t=u=0$ and act by $\alpha_1\mapsto\alpha_1x^2$, $\alpha_4\mapsto\alpha_4x^5$, $\alpha_6\mapsto\alpha_6x^6$, i.e.\ $\alpha_1/\alpha_6\mapsto x^{-4}\alpha_1/\alpha_6$ and $\alpha_4/\alpha_6\mapsto x^{-1}\alpha_4/\alpha_6$.
If $\alpha_4\neq0$ we normalize $\alpha_4=\alpha_6=1$, which determines $x$ uniquely, so that $\alpha=\alpha_1$ is an invariant: this gives $\alpha\N1+\N4+\N6$.
If $\alpha_4=0$ we obtain $\N6$ or $\N1+\N6$.

\emph{Case $\alpha_6=0$, $\alpha_5\neq0$.}
Now $\alpha_4^*/\alpha_5^*=\alpha_4/\alpha_5=:\alpha$ is an invariant.
The parameters $t$ and $u$ allow us to obtain $\alpha_1^*=\alpha_2^*=0$, and $\alpha_3^*=x^3\big(\alpha_3x+(3\alpha_4+\alpha_5)y\big)$.
If $\alpha\neq-\frac13$ we choose $y$ with $\alpha_3^*=0$ and obtain $\alpha\N4+\N5$.
If $\alpha=-\frac13$, then $\alpha_3/\alpha_5\mapsto x^{-1}\alpha_3/\alpha_5$, and we obtain $-\frac13\N4+\N5$ or $\N3-\frac13\N4+\N5$.
\end{proof}

\begin{longtable}{@{}l p{0.90\textwidth}@{}}
$\mathcal{A}_{42}$ & $e_{1}e_{2}=e_{3},\ e_{1}e_{3}=e_{4},\ e_{2}e_{1}=-e_{3},\ e_{2}e_{2}=e_{4}, \ e_{2}e_{4}=e_{5},\ e_{3}e_{3}=-e_{5},\ e_{4}e_{2}=e_{5}$ \\ & \quad {\small $\theta=\langle \nabla_{6}\rangle$} \\[4pt]
$\mathcal{A}_{43}$ & $e_{1}e_{1}=e_{5},\ e_{1}e_{2}=e_{3},\ e_{1}e_{3}=e_{4},\ e_{2}e_{1}=-e_{3}, \ e_{2}e_{2}=e_{4},\ e_{2}e_{4}=e_{5},\ e_{3}e_{3}=-e_{5},\ e_{4}e_{2}=e_{5}$ \\ & \quad {\small $\theta=\langle \nabla_{1}+\nabla_{6}\rangle$} \\[4pt]
$\mathcal{A}_{44}(\alpha)$ & $e_{1}e_{1}=\alpha e_{5},\ e_{1}e_{2}=e_{3},\ e_{1}e_{3}=e_{4},\ e_{2}e_{1}=-e_{3}, \ e_{2}e_{2}=e_{4}$ \\ & $e_{2}e_{3}=e_{5},\ e_{2}e_{4}=e_{5},\ e_{3}e_{3}=-e_{5}, \ e_{4}e_{2}=e_{5}$ \\ & \quad {\small $\theta=\langle \alpha \nabla_{1}+\nabla_{4}+\nabla_{6}\rangle$;\quad $\alpha\in\mathbb{C}$} \\[4pt]
$\mathcal{A}_{45}(\alpha)$ & $e_{1}e_{2}=e_{3},\ e_{1}e_{3}=e_{4},\ e_{1}e_{4}=e_{5},\ e_{2}e_{1}=-e_{3}, \ e_{2}e_{2}=e_{4},\ e_{2}e_{3}=\alpha e_{5},\ e_{3}e_{2}=e_{5},\ e_{4}e_{1}=e_{5}$ \\ & \quad {\small $\theta=\langle \alpha \nabla_{4}+\nabla_{5}\rangle$;\quad $\alpha\in\mathbb{C}$} \\[4pt]
$\mathcal{A}_{46}$ & $e_{1}e_{2}=e_{3},\ e_{1}e_{3}=e_{4}+3 e_{5},\ e_{1}e_{4}=3 e_{5},\ e_{2}e_{1}=-e_{3}$ \\ & $e_{2}e_{2}=e_{4},\ e_{2}e_{3}=-e_{5},\ e_{3}e_{2}=3 e_{5},\ e_{4}e_{1}=3 e_{5}$ \\ & \quad {\small $\theta=\langle 3 \nabla_{3}-\nabla_{4}+3 \nabla_{5}\rangle$} \\[4pt]
\end{longtable}

\textbf{Central extensions of $\RR49$.} Here $\HH(\RR49)=\sspan{\N1}$ with $\N1=[\D{14}+\D{23}+\D{32}+\D{41}]$, and $\N1$ is admissible because $\N1(e_1,e_4)=1$.
Hence $\RR49$ has, up to isomorphism, exactly one non-split one-dimensional central extension:

\begin{longtable}{@{}l p{0.80\textwidth}@{}}
$\mathcal{A}_{47}$ & $e_{1}e_{1}=e_{2},\ e_{1}e_{2}=e_{3},\ e_{1}e_{3}=e_{4},\ e_{1}e_{4}=e_{5}, \ e_{2}e_{1}=e_{3}$ \\ & $e_{2}e_{2}=e_{4},\ e_{2}e_{3}=e_{5},\ e_{3}e_{1}=e_{4}, \ e_{3}e_{2}=e_{5},\ e_{4}e_{1}=e_{5}$ \\ & \quad {\small $\theta=\langle \nabla_{1}\rangle$} \\[4pt]
\end{longtable}

\begin{remark}
The algebra $\A{47}$ is generated by $e_1$: we have $e_2=e_1e_1$, $e_3=e_1e_2$, $e_4=e_1e_3$ and $e_5=e_1e_4$.
Since $e_3\in\A{47}^{3}$ and $e_4=e_1e_3\in\A{47}^{4}$, we get $0\neq e_5=e_4e_1\in\A{47}^{5}$, so $\A{47}$ has nilindex $6$, the maximal possible value for a five-dimensional nilpotent algebra.
\end{remark}
\section{One-dimensional extensions of the $2$-step nilpotent four-dimensional algebras}\label{sec:twostep}

By Corollary~\ref{cor:empty4}, the $2$-step nilpotent four-dimensional algebras with $\T_1(\mathcal{A})\neq\emptyset$ are
$\M1$, $\M2$, $\M3$, $\M4(0)$, $\M6$, $\M7(\gamma)$, $\M{11}$, $\M{12}$, $\M{13}$ and $\M{14}(\lambda)$.
Since these algebras are $2$-step nilpotent, the cocycle condition \eqref{eq:cocycle} only involves the values of $\theta$ on pairs $(u,w)$ with $u\in\mathcal{A}^2$ or $w\in\mathcal{A}^2$; the resulting bases of $\HH(\mathcal{A})$ are those of Proposition~\ref{prop:data4twostep}, which we use throughout together with the automorphisms of Proposition~\ref{prop:aut}.
As in Section~\ref{sec:pure}, a class $\theta$ is \emph{admissible} if $\sspan\theta\in\T_1(\mathcal{A})$.
In all cases $\Ann(\mathcal{A})\supseteq\mathcal{A}^2$, and admissibility is decided by writing down the values $\theta(e_i,m)$, $\theta(m,e_i)$ for a general $m\in\Ann(\mathcal{A})$.
We repeatedly use the following observation.

\begin{lemma}\label{lem:rank}
Let $\Ann(\mathcal{A})=\sspan{e_3,e_4}$ and let $m=ae_3+be_4$.
Then $\theta$ is admissible if and only if the vectors $\big(\theta(e_i,e_3),\theta(e_i,e_4)\big)$ and $\big(\theta(e_3,e_i),\theta(e_4,e_i)\big)$, $i=1,2,3,4$, span $\C^2$.
\end{lemma}

\begin{proof}
$m\in\Ann(\theta)$ if and only if $a\theta(e_i,e_3)+b\theta(e_i,e_4)=0$ and $a\theta(e_3,e_i)+b\theta(e_4,e_i)=0$ for all $i$.
This system has only the trivial solution if and only if its coefficient vectors span $\C^2$.
\end{proof}

\textbf{Central extensions of $\M1$.} Here $\N1,\dots,\N9$ are the classes of $\D{13}$, $\D{14}$, $\D{12}+\D{21}$, $\D{31}$, $\D{33}$, $\D{34}$, $\D{41}$, $\D{43}$, $\D{44}$, and $\Ann(\M1)=\sspan{e_2,e_3,e_4}$.
Put $U=\sspan{e_3,e_4}$.
We encode $\theta=\sum\alpha_i\N{i}$ by the number $\alpha_3$ and by
\[
\lambda=(\alpha_1,\alpha_2),\qquad \rho=(\alpha_4,\alpha_7),\qquad B=\begin{pmatrix}\alpha_5&\alpha_6\\ \alpha_8&\alpha_9\end{pmatrix},
\]
i.e.\ $\lambda=\theta(e_1,\cdot)|_U$, $\rho=\theta(\cdot,e_1)|_U$ and $B$ is the Gram matrix of $\theta|_{U\times U}$.
For $\phi\in\Aut(\M1)$ as in Proposition~\ref{prop:aut} put
$S=\begin{pmatrix}y&j\\ h&z\end{pmatrix}$, $m=(u,v)^{T}$ and $n=(w,g)$.

\begin{lemma}\label{lem:M1action}
We have $\alpha_3^*=\alpha_3x^3$ and
\[
\lambda^*=x\lambda S+\alpha_3x\,n+m^{T}BS,\qquad \rho^*=x\rho S+\alpha_3x\,n+m^{T}B^{T}S,\qquad B^*=S^{T}BS .
\]
In particular, if $\alpha_3\neq0$ we may assume $\lambda=0$, and then $\delta=\rho-\lambda$ transforms by $\delta^*=x\delta S+m^{T}(B^{T}-B)S$.
Moreover $\theta$ is admissible if and only if $\alpha_3\neq0$ and, after normalizing $\lambda=0$, $\ker B\cap\ker B^{T}\cap\ker\delta=0$ \textup{(}kernels of $B$, $B^T$ acting on column vectors and of the row vector $\delta$\textup{)}.
\end{lemma}

\begin{proof}
The formulas follow by expanding $\theta(\phi(e_i),\phi(e_j))$, using that the only nonzero values of $\theta$ involving $e_1$ or $e_2$ are $\theta(e_1,e_2)=\theta(e_2,e_1)=\alpha_3$, $\lambda$ and $\rho$ (the class $\D{11}$ is a coboundary).
For admissibility, let $ae_2+m'$ with $m'\in U$.
Then $\theta(e_1,ae_2+m')=a\alpha_3+\lambda m'$, $\theta(ae_2+m',e_1)=a\alpha_3+\rho m'$, $\theta(e_2,ae_2+m')=\theta(ae_2+m',e_2)=0$, and $\theta(U,m')$, $\theta(m',U)$ are given by $Bm'$ and $B^{T}m'$.
If $\alpha_3=0$ then $e_2\in\Ann(\theta)$.
If $\alpha_3\neq0$ and $\lambda=0$, the equations give $a=0$ and $\delta m'=Bm'=B^{T}m'=0$.
\end{proof}

\begin{lemma}\label{lem:congruence}
Let $B$ be a nonsymmetric complex $2\times2$ matrix and write $B=G+sK$ with $G$ symmetric, $s\in\C^{*}$ and $K=\begin{pmatrix}0&1\\-1&0\end{pmatrix}$.
Under the action $B\mapsto cS^{T}BS$ \textup{(}$c\in\C^*$, $S\in GL_2(\C)$\textup{)}:
\begin{enumerate}[label=\textup{(\alph*)}]
\item if $G=0$, then $B$ lies in the orbit of $K$;
\item if $G\neq0$, then $\kappa(B)=\det G/s^{2}$ is an invariant, and $B$ lies in the orbit of $B_\kappa=\begin{pmatrix}1&1\\-1&\kappa\end{pmatrix}$ with $\kappa=\kappa(B)$.
\end{enumerate}
Moreover $\det B=\det G+s^2$, so $B_\kappa$ is singular exactly for $\kappa=-1$.
\end{lemma}

\begin{proof}
Since $S^{T}KS=\det(S)K$, the symmetric and antisymmetric parts transform separately: $G\mapsto cS^{T}GS$, $s\mapsto c\det(S)s$.
Hence $\det G/s^2$ is invariant, and (a) is clear.
In (b), choose $v$ with $G(v,v)=g_1\neq0$ (possible because $G\neq0$ and $\operatorname{char}\C\neq2$) and $w$ with $G(v,w)=0$, $w\notin\C v$.
In the basis $(v,w)$, $G=\operatorname{diag}(g_1,g_2)$.
Replacing $w$ by a multiple we can assume $\det(v,w)=g_1/s$; then $c=g_1^{-1}$ and $S=(v\ w)$ give $cS^TBS=\operatorname{diag}(1,g_2/g_1)+K=B_{\kappa'}$ with $\kappa'=\kappa(B)$ by invariance.
The formula for $\det B$ is immediate.
\end{proof}

\begin{proposition}\label{prop:M1}
Every admissible class for $\M1$ lies in the orbit of exactly one of
\[
\begin{array}{lll}
\N3+\N5+\N6-\N8+\alpha\N9\ (\alpha\in\C), & \N3+\N6-\N8, & \N3+\N6+\N8,\\
\N3+\N4+\N6+\N8, & \N3+\N4+\N6+\N7+\N8, & \N3+\N5+\N7,
\end{array}
\]
and different values of $\alpha$ give different orbits.
\end{proposition}

\begin{proof}
Let $\theta$ be admissible.
By Lemma~\ref{lem:M1action} we may assume $\alpha_3=1$ and $\lambda=0$.
An automorphism preserves the condition $\lambda=0$ if and only if $n=-x^{-1}m^{T}BS$; hence, after dividing by $\alpha_3^*=x^3$, the group acts on the pairs $(B,\delta)$ by
$(B,\delta)\mapsto\big(x^{-3}S^{T}BS,\ x^{-3}(x\delta S+m^T(B^T-B)S)\big)$, with $x,S,m$ arbitrary.
The congruence class of $B$ up to scalars is an invariant.

\emph{(1) $B$ is not symmetric.}
Then $B^{T}-B$ is invertible and we may take $\delta=0$; then $\theta$ is admissible if and only if $\ker B\cap\ker B^T=0$, which always holds for a nonsymmetric $2\times 2$ matrix (if $B$ has rank one, $B=ab^{T}$ with $a\not\parallel b$ and $\ker B\cap\ker B^T=b^\perp\cap a^\perp=0$).
By Lemma~\ref{lem:congruence}, we get $B=K$ (the class $\N3+\N6-\N8$) or $B=B_\alpha$ (the class $\N3+\N5+\N6-\N8+\alpha\N9$), where $\alpha=\kappa(B)$ is an invariant.

\emph{(2) $B$ is symmetric of rank $2$.}
Then $B$ is congruent to $\begin{pmatrix}0&1\\1&0\end{pmatrix}$, admissibility holds automatically, and $\delta$ transforms by $\delta\mapsto x^{-2}\delta S$ where $S^{T}BS=x^{3}B$ (we may take $B$ fixed).
The quadratic function $\delta\mapsto\delta B^{-1}\delta^{T}$ is multiplied by $x^{-1}$, and the group $\{S : S^{T}BS\in\C^{*}B\}$ acts transitively on the nonzero isotropic vectors and on the non-isotropic vectors up to scalars.
Hence $\delta=0$, $\delta=(1,0)$ or $\delta=(1,1)$, which give $\N3+\N6+\N8$, $\N3+\N4+\N6+\N8$ and $\N3+\N4+\N6+\N7+\N8$.

\emph{(3) $B$ is symmetric of rank $1$.}
Then $B$ is congruent to $\operatorname{diag}(1,0)$, and $\ker B=\ker B^{T}=\sspan{e_4}$, so admissibility means $\delta(e_4)\neq0$, i.e.\ $\alpha_7\neq0$.
The matrices $S$ with $S^{T}BS\in\C^*B$ are the lower triangular ones ($j=0$), and $\delta S=(\delta_1y+\delta_2h,\delta_2z)$; choosing $h$ we get $\delta=(0,1)$, i.e.\ the class $\N3+\N5+\N7$.

\emph{(4) $B=0$.} Then $\ker\delta\neq0$, so $\theta$ is not admissible.
\end{proof}

\begin{longtable}{@{}l p{0.90\textwidth}@{}}
$\mathcal{A}_{48}(\alpha)$ & $e_{1}e_{1}=e_{2},\ e_{1}e_{2}=e_{5},\ e_{2}e_{1}=e_{5},\ e_{3}e_{3}=e_{5}, \ e_{3}e_{4}=e_{5},\ e_{4}e_{3}=-e_{5},\ e_{4}e_{4}=\alpha e_{5}$ \\ & \quad {\small $\theta=\langle \nabla_{3}+\nabla_{5}+\nabla_{6}-\nabla_{8}+\alpha \nabla_{9}\rangle$;\quad $\alpha\in\mathbb{C}$} \\[4pt]
$\mathcal{A}_{49}$ & $e_{1}e_{1}=e_{2},\ e_{1}e_{2}=e_{5},\ e_{2}e_{1}=e_{5},\ e_{3}e_{4}=e_{5}, \ e_{4}e_{3}=-e_{5}$ \\ & \quad {\small $\theta=\langle \nabla_{3}+\nabla_{6}-\nabla_{8}\rangle$} \\[4pt]
$\mathcal{A}_{50}$ & $e_{1}e_{1}=e_{2},\ e_{1}e_{2}=e_{5},\ e_{2}e_{1}=e_{5},\ e_{3}e_{4}=e_{5}, \ e_{4}e_{3}=e_{5}$ \\ & \quad {\small $\theta=\langle \nabla_{3}+\nabla_{6}+\nabla_{8}\rangle$} \\[4pt]
$\mathcal{A}_{51}$ & $e_{1}e_{1}=e_{2},\ e_{1}e_{2}=e_{5},\ e_{2}e_{1}=e_{5},\ e_{3}e_{1}=e_{5}, \ e_{3}e_{4}=e_{5},\ e_{4}e_{3}=e_{5}$ \\ & \quad {\small $\theta=\langle \nabla_{3}+\nabla_{4}+\nabla_{6}+\nabla_{8}\rangle$} \\[4pt]
$\mathcal{A}_{52}$ & $e_{1}e_{1}=e_{2},\ e_{1}e_{2}=e_{5},\ e_{2}e_{1}=e_{5},\ e_{3}e_{1}=e_{5}, \ e_{3}e_{4}=e_{5},\ e_{4}e_{1}=e_{5},\ e_{4}e_{3}=e_{5}$ \\ & \quad {\small $\theta=\langle \nabla_{3}+\nabla_{4}+\nabla_{6}+\nabla_{7}+\nabla_{8}\rangle$} \\[4pt]
$\mathcal{A}_{53}$ & $e_{1}e_{1}=e_{2},\ e_{1}e_{2}=e_{5},\ e_{2}e_{1}=e_{5},\ e_{3}e_{3}=e_{5}, \ e_{4}e_{1}=e_{5}$ \\ & \quad {\small $\theta=\langle \nabla_{3}+\nabla_{5}+\nabla_{7}\rangle$} \\[4pt]
\end{longtable}

\textbf{Central extensions of $\M2$.} Let us use the following notations: 
$$\N1=[\D{13}], \ \N2=[\D{12}+\D{21}], \ \N3=[\D{31}], \ \N4=[\D{34}+\D{43}], \ \Ann(\M2)=\sspan{e_2,e_4}$$ and for the connected automorphisms of Proposition~\ref{prop:aut}
\[
\alpha_1^*=\alpha_1xy+\alpha_2xz+\alpha_4uy,\quad \alpha_2^*=\alpha_2x^3,\quad \alpha_3^*=\alpha_2xz+\alpha_3xy+\alpha_4uy,\quad \alpha_4^*=\alpha_4y^3,
\]
while the involution $\sigma$ maps $(\alpha_1,\alpha_2,\alpha_3,\alpha_4)$ to $(\alpha_3,\alpha_4,\alpha_1,\alpha_2)$.

\begin{lemma}\label{lem:M2}
A class $\theta$ is admissible if and only if $\alpha_2\alpha_4\neq0$.
Every admissible class lies in the orbit of exactly one of $\N2+\N4$ and $\N2+\N3+\N4$.
In particular, the classes $\N2+\alpha\N3+\N4$ with $\alpha\neq0$ all lie in one orbit.
\end{lemma}

\begin{proof}
We have $\theta(e_1,ae_2+be_4)=a\alpha_2$, $\theta(e_3,ae_2+be_4)=b\alpha_4$, $\theta(ae_2+be_4,e_1)=a\alpha_2$, $\theta(ae_2+be_4,e_3)=b\alpha_4$, and the other values vanish; this gives the first claim.
The parameters $z,u$ add the same quantity to $\alpha_1^*$ and $\alpha_3^*$, and $\alpha_3^*-\alpha_1^*=(\alpha_3-\alpha_1)xy$; so we may assume $\alpha_1=0$ and the condition $\alpha_1=\alpha_3$ is invariant.
Taking $y^3=x^3\alpha_2/\alpha_4$ we get $\alpha_2=\alpha_4$, and then $\alpha_3/\alpha_2\mapsto(\alpha_3/\alpha_2)\,xy/x^3=\omega x^{-1}(\alpha_3/\alpha_2)$ with $\omega^3=1$, which can be made equal to $1$ if $\alpha_3\neq0$.
\end{proof}

\begin{longtable}{@{}l p{0.80\textwidth}@{}}
$\mathcal{A}_{54}$ & $e_{1}e_{1}=e_{2},\ e_{1}e_{2}=e_{5},\ e_{2}e_{1}=e_{5},\ e_{3}e_{3}=e_{4}, \ e_{3}e_{4}=e_{5},\ e_{4}e_{3}=e_{5}$ \\ & \quad {\small $\theta=\langle \nabla_{2}+\nabla_{4}\rangle$} \\[4pt]
$\mathcal{A}_{55}$ & $e_{1}e_{1}=e_{2},\ e_{1}e_{2}=e_{5},\ e_{2}e_{1}=e_{5},\ e_{3}e_{1}=e_{5}, \ e_{3}e_{3}=e_{4},\ e_{3}e_{4}=e_{5},\ e_{4}e_{3}=e_{5}$ \\ & \quad {\small $\theta=\langle \nabla_{2}+\nabla_{3}+\nabla_{4}\rangle$} \\[4pt]
\end{longtable}

\textbf{Central extensions of $\M3$.} Here $\N1,\dots,\N{11}$ are the classes of $\D{11}$, $\D{12}$, $\D{13}$, $\D{14}$, $\D{22}$, $\D{23}$, $\D{24}$, $\D{41}$, $\D{42}$, $\D{43}$, $\D{44}$, $\Ann(\M3)=\sspan{e_3,e_4}$ and $\B(\M3)=\sspan{\D{12}-\D{21}}$.
Write $V=\sspan{e_1,e_2}$ and encode $\theta$ by
\[
q(v)=\alpha_1v_1^2+\alpha_2v_1v_2+\alpha_5v_2^2,\quad \lambda=(\alpha_3,\alpha_6),\quad \mu=(\alpha_4,\alpha_7),\quad \nu=(\alpha_8,\alpha_9),\quad \alpha_{10},\quad\alpha_{11};
\]
thus $\lambda=\theta(V,e_3)$, $\mu=\theta(V,e_4)$, $\nu=\theta(e_4,V)$, $\alpha_{10}=\theta(e_4,e_3)$, $\alpha_{11}=\theta(e_4,e_4)$, and $\theta(e_3,\cdot)=0$.
For $\phi$ as in Proposition~\ref{prop:aut} put $P=\begin{pmatrix}x&y\\ t&z\end{pmatrix}$ and $d=\det P$.

\begin{lemma}\label{lem:M3action}
We have
\[
\begin{array}{ll}
\alpha_{10}^*=\alpha_{10}jd, & \lambda^*=d\big(\lambda P+\alpha_{10}(v,g)\big),\\
\alpha_{11}^*=j(\alpha_{10}h+\alpha_{11}j), & \nu^*=j\big(\nu P+\alpha_{10}(u,w)+\alpha_{11}(v,g)\big),\\
\mu^*=h\big(\lambda P+\alpha_{10}(v,g)\big)+j\big(\mu P+\alpha_{11}(v,g)\big), &
\end{array}
\]
and $q^*(v)=q(Pv)+(\text{terms of degree one in } u,v,w,g)$; more precisely, if $\alpha_{10}=0$, the parameters $u,w$ change $(\alpha_1,\alpha_2,\alpha_5)$ by $\big(u\ell_1,\ u\ell_2+w\ell_1,\ w\ell_2\big)$, where $(\ell_1,\ell_2)=\lambda P$, and all other coefficients do not depend on $u,w$.
A class $\theta$ is admissible if and only if either $\alpha_{10}\neq0$, or $\alpha_{10}=0$, $\lambda\neq0$, and $\big(\nu\neq0$ or $\alpha_{11}\neq0$ or $\mu\notin\C\lambda\big)$.
\end{lemma}

\begin{proof}
The formulas are obtained by expanding $\theta(\phi(e_i),\phi(e_j))$, using $\D{21}\equiv\D{12}$ modulo $\B(\M3)$.
For $m=ae_3+be_4$: $\theta(V,m)=a\lambda+b\mu$, $\theta(m,V)=b\nu$, $\theta(m,e_3)=b\alpha_{10}$, $\theta(m,e_4)=b\alpha_{11}$, $\theta(e_4,m)=a\alpha_{10}+b\alpha_{11}$, $\theta(e_3,m)=0$.
If $\alpha_{10}\neq0$ these vanish only for $a=b=0$.
If $\alpha_{10}=0$, the vector $e_3$ is annihilated when $\lambda=0$, and a vector $ae_3+e_4$ is annihilated if and only if $\nu=0$, $\alpha_{11}=0$ and $\mu=-a\lambda$.
\end{proof}

\begin{proposition}\label{prop:M3}
Every admissible class for $\M3$ lies in the orbit of exactly one of the following classes \textup{(}$\alpha\in\C$, different values of $\alpha$ giving different orbits\textup{)}:
\begin{align*}
&\alpha_{10}\neq0:&& \N{10},\ \N1+\N{10},\ \N2+\N{10},\ \N4+\N{10},\ \N4+\N5+\N{10},\\
&&& \N1+\N4+\N5+\N{10},\ \N2+\N4+\N{10},\ \N1+\N4+\N{10};\\
&\alpha_{10}=0:&& \N3+\alpha\N5+\N7+\N{11},\ \N3+\N5+\N{11},\ \N3+\N{11},\ \N3+\alpha\N7+\N9,\ \N3+\N5-\N7+\N9,\\
&&& \N3+\N7,\ \N3+\N7+\N8,\ \N3+\N8,\ \N3+\N5+\N8 .
\end{align*}
\end{proposition}

\begin{proof}
The condition $\alpha_{10}=0$ is invariant.

\emph{Case $\alpha_{10}\neq0$.}
By Lemma~\ref{lem:M3action} we can successively achieve $\lambda=0$ (by $v,g$), $\alpha_{11}=0$ (by $h$) and $\nu=0$ (by $u,w$).
The automorphisms preserving these normalizations have $u=v=w=g=h=0$; they act by $q\mapsto q\circ P$, $\mu\mapsto j\mu P$, $\alpha_{10}\mapsto jd\alpha_{10}$.
Normalizing $\alpha_{10}=1$ (i.e.\ $j=c^{-1}d^{-1}$ after multiplication of $\theta$ by $c$), the pair $(q,\mu)$ transforms by $(q,\mu)\mapsto(c\,q\circ P,\ d^{-1}\mu P)$.
If $\mu=0$, then $q$ is classified by its rank: $\N{10}$, $\N1+\N{10}$, $\N2+\N{10}$.
If $\mu\neq0$, we may take $\mu=(1,0)$, and the stabilizer consists of the matrices $P=\begin{pmatrix}x&0\\t&1\end{pmatrix}$ (and arbitrary $c$).
Under $v_1\mapsto xv_1$, $v_2\mapsto tv_1+v_2$, the $v_2^2$-coefficient $\alpha_5$ of $q$ is preserved (up to $c$).
If $\alpha_5\neq0$ we complete the square and obtain $q=v_2^2$ or $q=v_1^2+v_2^2$, i.e.\ $\N4+\N5+\N{10}$ or $\N1+\N4+\N5+\N{10}$.
If $\alpha_5=0$, then $q=v_1(\alpha_1v_1+\alpha_2v_2)$; if $\alpha_2\neq0$ we kill $\alpha_1$ by means of $t$ and get $\N2+\N4+\N{10}$, otherwise $\N4+\N{10}$ or $\N1+\N4+\N{10}$.
These classes are separated by the rank of $q$, by whether $\mu=0$, and, for $\mu\neq0$, by whether the linear form $\mu$ divides $q$ and whether $q\in\C\mu^2$.

\emph{Case $\alpha_{10}=0$.}
Then $\lambda\neq0$ and $\lambda^*=d\lambda P$; we take $\lambda=(1,0)$, so that the remaining automorphisms have $y=0$ and $(\ell_1,\ell_2)=(dx,0)$.
By Lemma~\ref{lem:M3action} we can kill $\alpha_1$ and $\alpha_2$ with $u,w$, and $\alpha_4$ with $h$ (the coefficient of $h$ in $\alpha_4^*$ is $dx\neq0$).
The condition $\alpha_{11}=0$ is invariant.

(a) $\alpha_{11}\neq0$.
Then $\theta$ is admissible, and using $v,g$ we obtain $\nu=0$; since $g$ enters $\alpha_7^*$ and $\alpha_9^*$ in the same way, the quantity $\alpha_7-\alpha_9$ is transformed into $jz(\alpha_7-\alpha_9)$, and after the normalization $\alpha_9=0$ we have $\mu=(0,\alpha_7)$.
The automorphisms preserving all normalizations have $v=g=0$ and act by
\[
\alpha_3\mapsto x^2z\alpha_3,\quad \alpha_5\mapsto z^2\alpha_5,\quad \alpha_7\mapsto jz\alpha_7,\quad \alpha_{11}\mapsto j^2\alpha_{11}.
\]
Normalizing $\alpha_3=\alpha_{11}=1$ (scaling $\theta$ and choosing $x$), we have $\alpha_7\mapsto(z/j)\alpha_7$ and $\alpha_5\mapsto(z/j)^2\alpha_5$.
If $\alpha_7\neq0$ we take $\alpha_7=1$, and $\alpha=\alpha_5$ becomes an invariant: $\N3+\alpha\N5+\N7+\N{11}$.
If $\alpha_7=0$ we get $\N3+\N{11}$ or $\N3+\N5+\N{11}$.

(b) $\alpha_{11}=0$.
Now $\nu^*=j\nu P$, $\alpha_7\mapsto jz\alpha_7$, and admissibility means $\nu\neq0$ or $\alpha_7\neq0$.
The parameter $g$ enters $\alpha_5^*=z\big(\alpha_5z+g(\alpha_7+\alpha_9)\big)$ (the other coefficients affected by $g$ can be restored by $u,w$), so we can kill $\alpha_5$ whenever $\alpha_7+\alpha_9\neq0$.
If $\alpha_9\neq0$, we kill $\alpha_8$ by $t$; for $\alpha_7\neq-\alpha_9$ we kill $\alpha_5$ and obtain $\N3+\alpha\N7+\N9$ with the invariant $\alpha=\alpha_7/\alpha_9\neq-1$, while for $\alpha_7=-\alpha_9$ the coefficient $\alpha_5$ is multiplied by $z/x^2$ relative to $\alpha_3$, and we obtain $\N3-\N7+\N9$ (the member $\alpha=-1$ of the family) or $\N3+\N5-\N7+\N9$.
If $\alpha_9=0\neq\alpha_8$: for $\alpha_7\neq0$ we kill $\alpha_5$ and normalize $\alpha_7/\alpha_8\mapsto(z/x)\alpha_7/\alpha_8$ to $1$, obtaining $\N3+\N7+\N8$; for $\alpha_7=0$ we get $\N3+\N8$ or $\N3+\N5+\N8$.
If $\nu=0$, then $\alpha_7\neq0$, we kill $\alpha_5$ and get $\N3+\N7$.
All conditions used ($\alpha_{11}=0$, $\nu=0$, $\alpha_9=0$, $\alpha_7=0$, $\alpha_7+\alpha_9=0$, $\alpha_5=0$ in the non-killable cases) are invariant under the stabilizers, which proves that the listed orbits are distinct.
\end{proof}

\begin{longtable}{@{}l p{0.90\textwidth}@{}}
$\mathcal{A}_{56}$ & $e_{1}e_{2}=e_{3},\ e_{2}e_{1}=-e_{3},\ e_{4}e_{3}=e_{5}$ \\ & \quad {\small $\theta=\langle \nabla_{10}\rangle$} \\[4pt]
$\mathcal{A}_{57}$ & $e_{1}e_{1}=e_{5},\ e_{1}e_{2}=e_{3},\ e_{2}e_{1}=-e_{3},\ e_{4}e_{3}=e_{5}$ \\ & \quad {\small $\theta=\langle \nabla_{1}+\nabla_{10}\rangle$} \\[4pt]
$\mathcal{A}_{58}$ & $e_{1}e_{2}=e_{3}+e_{5},\ e_{2}e_{1}=-e_{3},\ e_{4}e_{3}=e_{5}$ \\ & \quad {\small $\theta=\langle \nabla_{2}+\nabla_{10}\rangle$} \\[4pt]
$\mathcal{A}_{59}$ & $e_{1}e_{2}=e_{3},\ e_{1}e_{4}=e_{5},\ e_{2}e_{1}=-e_{3},\ e_{4}e_{3}=e_{5}$ \\ & \quad {\small $\theta=\langle \nabla_{4}+\nabla_{10}\rangle$} \\[4pt]
$\mathcal{A}_{60}$ & $e_{1}e_{2}=e_{3},\ e_{1}e_{4}=e_{5},\ e_{2}e_{1}=-e_{3},\ e_{2}e_{2}=e_{5}, \ e_{4}e_{3}=e_{5}$ \\ & \quad {\small $\theta=\langle \nabla_{4}+\nabla_{5}+\nabla_{10}\rangle$} \\[4pt]
$\mathcal{A}_{61}$ & $e_{1}e_{1}=e_{5},\ e_{1}e_{2}=e_{3},\ e_{1}e_{4}=e_{5},\ e_{2}e_{1}=-e_{3}, \ e_{2}e_{2}=e_{5},\ e_{4}e_{3}=e_{5}$ \\ & \quad {\small $\theta=\langle \nabla_{1}+\nabla_{4}+\nabla_{5}+\nabla_{10}\rangle$} \\[4pt]
$\mathcal{A}_{62}$ & $e_{1}e_{2}=e_{3}+e_{5},\ e_{1}e_{4}=e_{5},\ e_{2}e_{1}=-e_{3},\ e_{4}e_{3}=e_{5}$ \\ & \quad {\small $\theta=\langle \nabla_{2}+\nabla_{4}+\nabla_{10}\rangle$} \\[4pt]
$\mathcal{A}_{63}$ & $e_{1}e_{1}=e_{5},\ e_{1}e_{2}=e_{3},\ e_{1}e_{4}=e_{5},\ e_{2}e_{1}=-e_{3}, \ e_{4}e_{3}=e_{5}$ \\ & \quad {\small $\theta=\langle \nabla_{1}+\nabla_{4}+\nabla_{10}\rangle$} \\[4pt]
$\mathcal{A}_{64}(\alpha)$ & $e_{1}e_{2}=e_{3},\ e_{1}e_{3}=e_{5},\ e_{2}e_{1}=-e_{3},\ e_{2}e_{2}=\alpha e_{5}, \ e_{2}e_{4}=e_{5},\ e_{4}e_{4}=e_{5}$ \\ & \quad {\small $\theta=\langle \nabla_{3}+\alpha \nabla_{5}+\nabla_{7}+\nabla_{11}\rangle$;\quad $\alpha\in\mathbb{C}$} \\[4pt]
$\mathcal{A}_{65}$ & $e_{1}e_{2}=e_{3},\ e_{1}e_{3}=e_{5},\ e_{2}e_{1}=-e_{3},\ e_{2}e_{2}=e_{5}$ \\ & $e_{4}e_{4}=e_{5}$ \\ & \quad {\small $\theta=\langle \nabla_{3}+\nabla_{5}+\nabla_{11}\rangle$} \\[4pt]
$\mathcal{A}_{66}$ & $e_{1}e_{2}=e_{3},\ e_{1}e_{3}=e_{5},\ e_{2}e_{1}=-e_{3},\ e_{4}e_{4}=e_{5}$ \\ & \quad {\small $\theta=\langle \nabla_{3}+\nabla_{11}\rangle$} \\[4pt]
$\mathcal{A}_{67}(\alpha)$ & $e_{1}e_{2}=e_{3},\ e_{1}e_{3}=e_{5},\ e_{2}e_{1}=-e_{3},\ e_{2}e_{4}=\alpha e_{5}, \ e_{4}e_{2}=e_{5}$ \\ & \quad {\small $\theta=\langle \nabla_{3}+\alpha \nabla_{7}+\nabla_{9}\rangle$;\quad $\alpha\in\mathbb{C}$} \\[4pt]
$\mathcal{A}_{68}$ & $e_{1}e_{2}=e_{3},\ e_{1}e_{3}=e_{5},\ e_{2}e_{1}=-e_{3},\ e_{2}e_{2}=e_{5}, \ e_{2}e_{4}=-e_{5},\ e_{4}e_{2}=e_{5}$ \\ & \quad {\small $\theta=\langle \nabla_{3}+\nabla_{5}-\nabla_{7}+\nabla_{9}\rangle$} \\[4pt]
$\mathcal{A}_{69}$ & $e_{1}e_{2}=e_{3},\ e_{1}e_{3}=e_{5},\ e_{2}e_{1}=-e_{3},\ e_{2}e_{4}=e_{5}$ \\ & \quad {\small $\theta=\langle \nabla_{3}+\nabla_{7}\rangle$} \\[4pt]
$\mathcal{A}_{70}$ & $e_{1}e_{2}=e_{3},\ e_{1}e_{3}=e_{5},\ e_{2}e_{1}=-e_{3},\ e_{2}e_{4}=e_{5}, \ e_{4}e_{1}=e_{5}$ \\ & \quad {\small $\theta=\langle \nabla_{3}+\nabla_{7}+\nabla_{8}\rangle$} \\[4pt]
$\mathcal{A}_{71}$ & $e_{1}e_{2}=e_{3},\ e_{1}e_{3}=e_{5},\ e_{2}e_{1}=-e_{3},\ e_{4}e_{1}=e_{5}$ \\ & \quad {\small $\theta=\langle \nabla_{3}+\nabla_{8}\rangle$} \\[4pt]
$\mathcal{A}_{72}$ & $e_{1}e_{2}=e_{3},\ e_{1}e_{3}=e_{5},\ e_{2}e_{1}=-e_{3},\ e_{2}e_{2}=e_{5}, \ e_{4}e_{1}=e_{5}$ \\ & \quad {\small $\theta=\langle \nabla_{3}+\nabla_{5}+\nabla_{8}\rangle$} \\[4pt]
\end{longtable}

\textbf{Central extensions of $\M4(0)$.} Here $\N1,\dots,\N9$ are the classes of $\D{11}$, $\D{14}$, $\D{21}$, $\D{22}$, $\D{24}$, $\D{13}+\D{31}$, $\D{41}$, $\D{42}$, $\D{44}$ and $\Ann(\M4(0))=\sspan{e_3,e_4}$.
For $\phi$ as in Proposition~\ref{prop:aut} we have $\alpha_6^*=\alpha_6x^2y$, $\alpha_9^*=\alpha_9z^2$, and
\[
\begin{array}{ll}
\alpha_2^*=\alpha_6jx+z\big(\alpha_2x+\alpha_5(y-x)+\alpha_9t\big), & \alpha_7^*=\alpha_6jx+z\big(\alpha_7x+\alpha_8(y-x)+\alpha_9t\big),\\
\alpha_4^*=\alpha_4y^2+(\alpha_5+\alpha_8)wy+\alpha_9w^2, & \alpha_5^*=z(\alpha_5y+\alpha_9w),\qquad \alpha_8^*=z(\alpha_8y+\alpha_9w),
\end{array}
\]
while $\alpha_1^*=2\alpha_6ux+\dots$ and $\alpha_3^*=\alpha_6vx+\dots$, where the dots denote terms not involving $u$ \textup{(}resp.\ $v$\textup{)}.

\begin{proposition}\label{prop:M40}
A class $\theta$ is admissible if and only if $\alpha_6\neq0$ and $(\alpha_2-\alpha_7,\alpha_5,\alpha_8,\alpha_9)\neq0$.
Every admissible class lies in the orbit of exactly one of the following classes \textup{(}$\alpha\in\C$\textup{)}:
\[
\begin{array}{lll}
\alpha\N4+\N5+\N6+\N9, & \alpha\N4+\N5+\N6-\N7+\N9, & \N6+\N9,\\
\N6+\N7+\N9, & \N4+\N6+\N9, & \N4+\N6+\N7+\N9,\\
\N6+\N7, & \N4+\N6+\N7, & \N5+\N6+\alpha\N8,\\
\N5+\N6+(\alpha-1)\N7+\alpha\N8\ (\alpha\neq1), & \N5+\N6+\N7+\N8, & \N4+\N5+\N6-\N8,\\
\N4+\N5+\N6-2\N7-\N8, & \N6+\N8, & \N6+\N7+\N8,
\end{array}
\]
and, in each family, different values of $\alpha$ give different orbits.
\end{proposition}

\begin{proof}
For $m=ae_3+be_4$ we have $\theta(e_1,m)=a\alpha_6+b\alpha_2$, $\theta(e_2,m)=b\alpha_5$, $\theta(m,e_1)=a\alpha_6+b\alpha_7$, $\theta(m,e_2)=b\alpha_8$, $\theta(e_4,m)=\theta(m,e_4)=b\alpha_9$; this gives the admissibility criterion.
Let $\theta$ be admissible; we normalize $\alpha_6=1$.
Using $u$, $v$ and $j$ we obtain $\alpha_1=\alpha_3=\alpha_2=0$.
Since $\alpha_7^*-\alpha_2^*=z\big((\alpha_7-\alpha_2)x+(\alpha_8-\alpha_5)(y-x)\big)$, after this normalization the remaining coefficients transform by
\begin{equation}\label{eq:M40red}
\begin{array}{lll}
\alpha_4\mapsto\alpha_4y^2+(\alpha_5+\alpha_8)wy+\alpha_9w^2, & \alpha_5\mapsto z(\alpha_5y+\alpha_9w), & \alpha_6\mapsto x^2y,\\
\alpha_7\mapsto z\big(\alpha_7x+(\alpha_8-\alpha_5)(y-x)\big), & \alpha_8\mapsto z(\alpha_8y+\alpha_9w), & \alpha_9\mapsto \alpha_9z^2,
\end{array}
\end{equation}
where $x,y,z,w$ are arbitrary with $xyz\neq0$.
The condition $\alpha_9=0$ is invariant.

\emph{Case $\alpha_9\neq0$.}
Choosing $w$ we get $\alpha_8=0$; the stabilizer of this normalization has $w=0$ and acts by $\alpha_4\mapsto\alpha_4y^2$, $\alpha_5\mapsto yz\alpha_5$, $\alpha_7\mapsto z(\alpha_7x-\alpha_5(y-x))$, $\alpha_9\mapsto z^2\alpha_9$, $\alpha_6\mapsto x^2y$.
If $\alpha_5\neq0$ and $\alpha_7\neq-\alpha_5$, the choice $y=x(1+\alpha_7/\alpha_5)$ gives $\alpha_7=0$; then $y=x$ is forced, and the normalization $\alpha_5=\alpha_6=\alpha_9=1$ determines $z=x^2$ and $x=1$, so that $\alpha=\alpha_4$ is an invariant: $\alpha\N4+\N5+\N6+\N9$.
If $\alpha_5\neq0$ and $\alpha_7=-\alpha_5$ (an invariant condition), the normalization $\alpha_5=\alpha_6=\alpha_9=1$ gives $z=x^2$, $y=x^2$ and $\alpha=\alpha_4$ is again invariant: $\alpha\N4+\N5+\N6-\N7+\N9$.
If $\alpha_5=0$, then $\alpha_4/\alpha_6$ and $\alpha_7/\alpha_6$ are multiplied by $y/x^2$ and $z/(xy)$, and $\alpha_9/\alpha_6$ by $z^2/(x^2y)$; after normalizing $\alpha_9=\alpha_6$ the first two factors become $z^2/x^4$ and $x/z$, which are independent. We get $\N6+\N9$, $\N6+\N7+\N9$, $\N4+\N6+\N9$, $\N4+\N6+\N7+\N9$.

\emph{Case $\alpha_9=0$.}
Now $\alpha_5$ and $\alpha_8$ are both multiplied by $yz$, so $[\alpha_5:\alpha_8]$ (if defined) is an invariant, and by \eqref{eq:M40red} the relation $\alpha_7=\alpha_8-\alpha_5$ is invariant.
If $\alpha_5=\alpha_8=0$, then $\alpha_7\neq0$ and $\alpha_4\mapsto\alpha_4y^2$; we get $\N6+\N7$ and $\N4+\N6+\N7$.
Suppose $(\alpha_5,\alpha_8)\neq0$.
If $\alpha_5+\alpha_8\neq0$ we kill $\alpha_4$ by $w$; if $\alpha_5+\alpha_8=0$ then $\alpha_4\mapsto y^2\alpha_4$.
If $\alpha_8\neq\alpha_5$ and $\alpha_7\neq\alpha_8-\alpha_5$, we choose $y=x(\alpha_8-\alpha_5-\alpha_7)/(\alpha_8-\alpha_5)$ and obtain $\alpha_7=0$; if $\alpha_8=\alpha_5$, then $\alpha_7\mapsto xz\alpha_7$.
Writing $\alpha_5=1$, $\alpha_8=\alpha$ when $\alpha_5\neq0$, we obtain:
$\N5+\N6+\alpha\N8$ (here $\alpha_7=0$ and $\alpha_4=0$, including $\alpha=-1$ with $\alpha_4=0$);
$\N5+\N6+(\alpha-1)\N7+\alpha\N8$ with $\alpha\neq1$ (here $\alpha_7=\alpha_8-\alpha_5\neq0$);
$\N5+\N6+\N7+\N8$ ($\alpha=1$, $\alpha_7\neq0$);
and, for $\alpha=-1$ and $\alpha_4\neq0$, the classes $\N4+\N5+\N6-\N8$ and $\N4+\N5+\N6-2\N7-\N8$ (in both cases $\alpha_4/\alpha_6$ can be normalized to $1$, because either $y=x$ and $\alpha_4/\alpha_6\mapsto x^{-1}\alpha_4/\alpha_6$, or $y$ is free).
If $\alpha_5=0\neq\alpha_8$ we get $\N6+\N8$ ($\alpha_7=0$) and $\N6+\N7+\N8$ ($\alpha_7=\alpha_8$).
All normalizations of ratios use the factors $z/x^2$ (for $\alpha_5/\alpha_6$ or $\alpha_8/\alpha_6$), which are free.
The invariants $[\alpha_5:\alpha_8]$, $\alpha_9=0$, $\alpha_7=\alpha_8-\alpha_5$, $\alpha_7=-\alpha_5$ and the vanishing of the non-killable coefficients separate all orbits.
\end{proof}

\begin{longtable}{@{}l p{0.90\textwidth}@{}}
$\mathcal{A}_{73}(\alpha)$ & $e_{1}e_{1}=e_{3},\ e_{1}e_{2}=e_{3},\ e_{1}e_{3}=e_{5},\ e_{2}e_{2}=\alpha e_{5}, \ e_{2}e_{4}=e_{5},\ e_{3}e_{1}=e_{5},\ e_{4}e_{4}=e_{5}$ \\ & \quad {\small $\theta=\langle \alpha \nabla_{4}+\nabla_{5}+\nabla_{6}+\nabla_{9}\rangle$;\quad $\alpha\in\mathbb{C}$} \\[4pt]
$\mathcal{A}_{74}(\alpha)$ & $e_{1}e_{1}=e_{3},\ e_{1}e_{2}=e_{3},\ e_{1}e_{3}=e_{5},\ e_{2}e_{2}=\alpha e_{5}, \ e_{2}e_{4}=e_{5},\ e_{3}e_{1}=e_{5},\ e_{4}e_{1}=-e_{5},\ e_{4}e_{4}=e_{5}$ \\ & \quad {\small $\theta=\langle \alpha \nabla_{4}+\nabla_{5}+\nabla_{6}-\nabla_{7}+\nabla_{9}\rangle$;\quad $\alpha\in\mathbb{C}$} \\[4pt]
$\mathcal{A}_{75}$ & $e_{1}e_{1}=e_{3},\ e_{1}e_{2}=e_{3},\ e_{1}e_{3}=e_{5},\ e_{3}e_{1}=e_{5}, \ e_{4}e_{4}=e_{5}$ \\ & \quad {\small $\theta=\langle \nabla_{6}+\nabla_{9}\rangle$} \\[4pt]
$\mathcal{A}_{76}$ & $e_{1}e_{1}=e_{3},\ e_{1}e_{2}=e_{3},\ e_{1}e_{3}=e_{5},\ e_{3}e_{1}=e_{5}, \ e_{4}e_{1}=e_{5},\ e_{4}e_{4}=e_{5}$ \\ & \quad {\small $\theta=\langle \nabla_{6}+\nabla_{7}+\nabla_{9}\rangle$} \\[4pt]
$\mathcal{A}_{77}$ & $e_{1}e_{1}=e_{3},\ e_{1}e_{2}=e_{3},\ e_{1}e_{3}=e_{5},\ e_{2}e_{2}=e_{5}, \ e_{3}e_{1}=e_{5},\ e_{4}e_{4}=e_{5}$ \\ & \quad {\small $\theta=\langle \nabla_{4}+\nabla_{6}+\nabla_{9}\rangle$} \\[4pt]
$\mathcal{A}_{78}$ & $e_{1}e_{1}=e_{3},\ e_{1}e_{2}=e_{3},\ e_{1}e_{3}=e_{5},\ e_{2}e_{2}=e_{5}, \ e_{3}e_{1}=e_{5},\ e_{4}e_{1}=e_{5},\ e_{4}e_{4}=e_{5}$ \\ & \quad {\small $\theta=\langle \nabla_{4}+\nabla_{6}+\nabla_{7}+\nabla_{9}\rangle$} \\[4pt]
$\mathcal{A}_{79}$ & $e_{1}e_{1}=e_{3},\ e_{1}e_{2}=e_{3},\ e_{1}e_{3}=e_{5},\ e_{3}e_{1}=e_{5}, \ e_{4}e_{1}=e_{5}$ \\ & \quad {\small $\theta=\langle \nabla_{6}+\nabla_{7}\rangle$} \\[4pt]
$\mathcal{A}_{80}$ & $e_{1}e_{1}=e_{3},\ e_{1}e_{2}=e_{3},\ e_{1}e_{3}=e_{5},\ e_{2}e_{2}=e_{5}, \ e_{3}e_{1}=e_{5},\ e_{4}e_{1}=e_{5}$ \\ & \quad {\small $\theta=\langle \nabla_{4}+\nabla_{6}+\nabla_{7}\rangle$} \\[4pt]
$\mathcal{A}_{81}(\alpha)$ & $e_{1}e_{1}=e_{3},\ e_{1}e_{2}=e_{3},\ e_{1}e_{3}=e_{5},\ e_{2}e_{4}=e_{5}, \ e_{3}e_{1}=e_{5},\ e_{4}e_{2}=\alpha e_{5}$ \\ & \quad {\small $\theta=\langle \nabla_{5}+\nabla_{6}+\alpha \nabla_{8}\rangle$;\quad $\alpha\in\mathbb{C}$} \\[4pt]
$\mathcal{A}_{82}(\alpha)$ & $e_{1}e_{1}=e_{3},\ e_{1}e_{2}=e_{3},\ e_{1}e_{3}=e_{5},\ e_{2}e_{4}=e_{5}, \ e_{3}e_{1}=e_{5},\ e_{4}e_{1}=(\alpha - 1) e_{5},\ e_{4}e_{2}=\alpha e_{5}$ \\ & \quad {\small $\theta=\langle \nabla_{5}+\nabla_{6}+(\alpha - 1) \nabla_{7}+\alpha \nabla_{8}\rangle$;\quad $\alpha\neq 1$} \\[4pt]
$\mathcal{A}_{83}$ & $e_{1}e_{1}=e_{3},\ e_{1}e_{2}=e_{3},\ e_{1}e_{3}=e_{5},\ e_{2}e_{4}=e_{5}, \ e_{3}e_{1}=e_{5},\ e_{4}e_{1}=e_{5},\ e_{4}e_{2}=e_{5}$ \\ & \quad {\small $\theta=\langle \nabla_{5}+\nabla_{6}+\nabla_{7}+\nabla_{8}\rangle$} \\[4pt]
$\mathcal{A}_{84}$ & $e_{1}e_{1}=e_{3},\ e_{1}e_{2}=e_{3},\ e_{1}e_{3}=e_{5},\ e_{2}e_{2}=e_{5}, \ e_{2}e_{4}=e_{5},\ e_{3}e_{1}=e_{5},\ e_{4}e_{2}=-e_{5}$ \\ & \quad {\small $\theta=\langle \nabla_{4}+\nabla_{5}+\nabla_{6}-\nabla_{8}\rangle$} \\[4pt]
$\mathcal{A}_{85}$ & $e_{1}e_{1}=e_{3},\ e_{1}e_{2}=e_{3},\ e_{1}e_{3}=e_{5},\ e_{2}e_{2}=e_{5}, \ e_{2}e_{4}=e_{5},\ e_{3}e_{1}=e_{5},\ e_{4}e_{1}=-2 e_{5},\ e_{4}e_{2}=-e_{5}$ \\ & \quad {\small $\theta=\langle \nabla_{4}+\nabla_{5}+\nabla_{6}-2 \nabla_{7}-\nabla_{8}\rangle$} \\[4pt]
$\mathcal{A}_{86}$ & $e_{1}e_{1}=e_{3},\ e_{1}e_{2}=e_{3},\ e_{1}e_{3}=e_{5},\ e_{3}e_{1}=e_{5}, \ e_{4}e_{2}=e_{5}$ \\ & \quad {\small $\theta=\langle \nabla_{6}+\nabla_{8}\rangle$} \\[4pt]
$\mathcal{A}_{87}$ & $e_{1}e_{1}=e_{3},\ e_{1}e_{2}=e_{3},\ e_{1}e_{3}=e_{5},\ e_{3}e_{1}=e_{5}, \ e_{4}e_{1}=e_{5},\ e_{4}e_{2}=e_{5}$ \\ & \quad {\small $\theta=\langle \nabla_{6}+\nabla_{7}+\nabla_{8}\rangle$} \\[4pt]
\end{longtable}

\textbf{Central extensions of $\M6$.} Let us use the following notations:
$$\N1=[\D{11}], \ \N2=[\D{12}], \ \N3=[\D{13}-\D{23}+\D{31}-\D{32}], \ \N4=[\D{24}+\D{42}], \ \N5=[\D{14}+\D{31}+\D{42}]$$ and
\[
\begin{array}{l}
\alpha_1^*=x\big(\alpha_1x+(2\alpha_3+\alpha_5)y+\alpha_5t\big),\qquad \alpha_2^*=x\big(\alpha_2x-\alpha_3(y+z)+(\alpha_4+\alpha_5)(t+2u)\big),\\
\alpha_3^*=\alpha_3x^3,\qquad\alpha_4^*=\alpha_4x^3,\qquad\alpha_5^*=\alpha_5x^3 .
\end{array}
\]

\begin{proposition}\label{prop:M6}
Every admissible class for $\M6$ lies in the orbit of exactly one of
\[
\alpha\N3+\beta\N4+\N5\ \big((\alpha,\beta)\neq(0,-1)\big),\ \N3+\alpha\N4\ (\alpha\neq0),\ -\N4+\N5,\ \N2-\N4+\N5,
\]
and different parameters give different orbits.
\end{proposition}

\begin{proof}
By Lemma~\ref{lem:rank}, $\theta$ is admissible if and only if the vectors $(\alpha_3,\alpha_5)$, $(-\alpha_3,\alpha_4)$, $(\alpha_3+\alpha_5,0)$, $(-\alpha_3,\alpha_4+\alpha_5)$ span $\C^2$.
If $\alpha_5\neq0$ this always holds (the minor of the first and third vectors is $-\alpha_5(\alpha_3+\alpha_5)$ and that of the first and fourth is $\alpha_3(\alpha_4+2\alpha_5)$; if both vanish then $\alpha_3=-\alpha_5$ and $\alpha_4=-2\alpha_5$, and the minor of the first two vectors is $\alpha_3(\alpha_4+\alpha_5)=\alpha_5^2\neq0$); if $\alpha_5=0$ it holds if and only if $\alpha_3\alpha_4\neq0$.
The point $(\alpha_3:\alpha_4:\alpha_5)$ is an invariant.
If $\alpha_3\neq0$, we kill $\alpha_2$ by $z$ and then $\alpha_1$ by $y$ or $t$; if $\alpha_3=0$ (so $\alpha_5\ne0$), we kill $\alpha_1$ by $t$ and, provided $\alpha_4+\alpha_5\neq0$, $\alpha_2$ by $u$.
This gives $\alpha\N3+\beta\N4+\N5$ with $(\alpha,\beta)\neq(0,-1)$ and $\N3+\alpha\N4$, $\alpha\neq0$.
If $(\alpha_3,\alpha_4,\alpha_5)\sim(0,-1,1)$, then $\alpha_2\mapsto x^2\alpha_2$ while $\alpha_5\mapsto x^3\alpha_5$, and we obtain $-\N4+\N5$ or $\N2-\N4+\N5$.
\end{proof}

\begin{longtable}{@{}l p{0.80\textwidth}@{}}
$\mathcal{A}_{88}(\alpha,\beta)$ & $e_{1}e_{2}=e_{3},\ e_{1}e_{3}=\alpha e_{5},\ e_{1}e_{4}=e_{5},\ e_{2}e_{1}=e_{4}, \ e_{2}e_{2}=-e_{3}$ \\ & $e_{2}e_{3}=- \alpha e_{5},\ e_{2}e_{4}=\beta e_{5},\ e_{3}e_{1}=(\alpha + 1) e_{5}, \ e_{3}e_{2}=- \alpha e_{5},\ e_{4}e_{2}=(\beta + 1) e_{5}$ \\ & \quad {\small $\theta=\langle \alpha \nabla_{3}+\beta \nabla_{4}+\nabla_{5}\rangle$;\quad $(\alpha,\beta)\neq(0,-1)$} \\[4pt]
$\mathcal{A}_{89}(\alpha)$ & $e_{1}e_{2}=e_{3},\ e_{1}e_{3}=e_{5},\ e_{2}e_{1}=e_{4},\ e_{2}e_{2}=-e_{3}, \ e_{2}e_{3}=-e_{5}$ \\ & $e_{2}e_{4}=\alpha e_{5},\ e_{3}e_{1}=e_{5},\ e_{3}e_{2}=-e_{5}, \ e_{4}e_{2}=\alpha e_{5}$ \\ & \quad {\small $\theta=\langle \nabla_{3}+\alpha \nabla_{4}\rangle$;\quad $\alpha\neq 0$} \\[4pt]
$\mathcal{A}_{90}$ & $e_{1}e_{2}=e_{3},\ e_{1}e_{4}=e_{5},\ e_{2}e_{1}=e_{4},\ e_{2}e_{2}=-e_{3}, \ e_{2}e_{4}=-e_{5},\ e_{3}e_{1}=e_{5}$ \\ & \quad {\small $\theta=\langle -\nabla_{4}+\nabla_{5}\rangle$} \\[4pt]
$\mathcal{A}_{91}$ & $e_{1}e_{2}=e_{3}+e_{5},\ e_{1}e_{4}=e_{5},\ e_{2}e_{1}=e_{4},\ e_{2}e_{2}=-e_{3}, \ e_{2}e_{4}=-e_{5},\ e_{3}e_{1}=e_{5}$ \\ & \quad {\small $\theta=\langle \nabla_{2}-\nabla_{4}+\nabla_{5}\rangle$} \\[4pt]
\end{longtable}

\textbf{Central extensions of $\M7(\gamma)$, $\gamma\neq1$.} Let us use the following notations: 
$$\N1=[\D{21}], \ \N2=[\D{12}], \ \N3=[\D{13}-\gamma\D{23}+\D{31}-\gamma\D{32}], \ \N4=[\D{14}-\D{24}+\D{41}-\D{42}].$$
For the automorphisms of the first kind in Proposition~\ref{prop:aut},
\[
\begin{array}{l}
\alpha_1^*=x\big(\alpha_1x+\alpha_3(\gamma y+z)+\alpha_4((2\gamma-1)t+u)\big),\ \alpha_2^*=x\big(\alpha_2x-\alpha_3(\gamma y+(2\gamma-1)z)-\alpha_4(t+u)\big),\\
\alpha_3^*=\alpha_3x^3,\ \alpha_4^*=\alpha_4x^3,
\end{array}
\]
and, if $\gamma\neq0$, $\sigma_p$ maps $(\alpha_1,\alpha_2,\alpha_3,\alpha_4)$ to $(\gamma p^2\alpha_2,\gamma p^2\alpha_1,p^3\alpha_4,\gamma^3p^3\alpha_3)$.

\begin{lemma}\label{lem:M7}
Let $\gamma\neq1$.
The matrix $\sigma_p$ of Proposition~\ref{prop:aut} is an automorphism of $\M7(\gamma)$ for every $\gamma\neq0$ and $p\neq0$.
A class $\theta$ is admissible if and only if $\alpha_3\alpha_4\neq0$, and every admissible class lies in the orbit of a class $\N3+\alpha\N4$, $\alpha\neq0$.
The classes $\N3+\alpha\N4$ and $\N3+\beta\N4$ lie in the same orbit if and only if $\beta=\alpha$, or $\gamma\neq0$ and $\beta=\gamma^3\alpha^{-1}$.
\end{lemma}

\begin{proof}
By its columns, the matrix $\sigma_p$ maps $e_1\mapsto pe_2$, $e_2\mapsto\gamma pe_1$, $e_3\mapsto-p^2e_4$, $e_4\mapsto-\gamma^2p^2e_3$. We check $\sigma_p(e_1)\sigma_p(e_1)=p^2e_2e_2=-p^2e_4=\sigma_p(e_3)$, $\sigma_p(e_1)\sigma_p(e_2)=\gamma p^2e_2e_1=-\gamma^2p^2e_3=\sigma_p(e_4)$, $\sigma_p(e_2)\sigma_p(e_1)=\gamma p^2e_1e_2=\gamma p^2e_4=\sigma_p(-\gamma e_3)$, $\sigma_p(e_2)\sigma_p(e_2)=\gamma^2p^2e_1e_1=\gamma^2p^2e_3=\sigma_p(-e_4)$, and all other products of the $\sigma_p(e_i)$ vanish.
By Lemma~\ref{lem:rank}, admissibility means that $(\alpha_3,\alpha_4)$ and $(-\gamma\alpha_3,-\alpha_4)$ span $\C^2$, i.e.\ $(\gamma-1)\alpha_3\alpha_4\neq0$.
The coefficients of $t,u$ in $(\alpha_1^*,\alpha_2^*)$ form the matrix $x\alpha_4\begin{pmatrix}2\gamma-1&1\\-1&-1\end{pmatrix}$ with determinant $2(1-\gamma)x^2\alpha_4^2\neq0$, so we can kill $\alpha_1$ and $\alpha_2$.
The connected automorphisms preserve the ratio $\alpha_4/\alpha_3$, and $\sigma_p$ maps $\N3+\alpha\N4$ to a multiple of $\N3+\gamma^3\alpha^{-1}\N4$ (modulo $\N1,\N2$, which are then killed again).
Since every automorphism is either connected or a composition of a connected one with some $\sigma_p$, the claim follows.
\end{proof}

\begin{longtable}{@{}l p{0.90\textwidth}@{}}
$\mathcal{A}_{92}(\alpha,\gamma)$ & $e_{1}e_{1}=e_{3},\ e_{1}e_{2}=e_{4},\ e_{1}e_{3}=e_{5},\ e_{1}e_{4}=\alpha e_{5},\ e_{2}e_{1}=- \gamma e_{3},\ e_{2}e_{2}=-e_{4}$ \\ & $e_{2}e_{3}=- \gamma e_{5},\ e_{2}e_{4}=- \alpha e_{5},\ e_{3}e_{1}=e_{5},\ e_{3}e_{2}=- \gamma e_{5},\ e_{4}e_{1}=\alpha e_{5},\ e_{4}e_{2}=- \alpha e_{5}$ \\ & \quad {\small $\theta=\langle \nabla_{3}+\alpha \nabla_{4}\rangle$;\quad $\alpha\neq0,\ \gamma\neq1;\ \mathcal{A}_{92}(\alpha,\gamma)\cong\mathcal{A}_{92}(\gamma^{3}\alpha^{-1},\gamma)$} \\[4pt]
\end{longtable}

\textbf{Central extensions of $\M7(1)$.} Let us use the following notations: 
$$\N1=[\D{11}],\ \N2=[\D{12}],\ \N3=[\D{13}-\D{23}+\D{31}-\D{41}],\ \N4=[\D{32}-\D{41}],\ \N5=[\D{14}-\D{24}+\D{41}-\D{42}].$$
For $\phi$ as in Proposition~\ref{prop:aut} put $d=y-z$; note that $\det\phi|_{\sspan{e_1,e_2}}=d(x+z)\neq0$.
Then
\[
\alpha_3^*=d^2\big((\alpha_3+\alpha_5)x-\alpha_5d\big),\
\alpha_4^*=d^2\big(\alpha_3z+\alpha_4(x+z)+\alpha_5(d-x)\big),\
\alpha_5^*=d^2(\alpha_3z+\alpha_5y),
\]
and the automorphisms with $x=y=1$, $z=0$ act by $(\alpha_1,\alpha_2)\mapsto(\alpha_1,\alpha_2)+(t+w)U_1+(u+v)U_2$ with
\[
U_1=(\alpha_5-s,\ -\alpha_5),\quad U_2=(\alpha_3,\ \alpha_4),\quad s=\alpha_3+\alpha_4 .
\]

\begin{proposition}\label{prop:M71}
A class $\theta$ is admissible if and only if $\alpha_3+\alpha_4\neq0$.
Every admissible class lies in the orbit of exactly one of
\[
\N4+\alpha\N5\ (\alpha\in\C),\quad \N1+\N4+\N5,\quad \N3-\N5,
\]
and different values of $\alpha$ give different orbits.
\end{proposition}

\begin{proof}
By Lemma~\ref{lem:rank}, admissibility means that $(\alpha_3,\alpha_5)$, $(-\alpha_3,-\alpha_5)$, $(\alpha_3,\alpha_5-\alpha_3-\alpha_4)$, $(\alpha_4,-\alpha_5)$ span $\C^2$.
The $2\times2$ minors are $-\alpha_3s$, $-\alpha_5s$ and $s(\alpha_4-\alpha_5)$; they all vanish if and only if $s=0$ (if $s\neq0$ they cannot all vanish, since $\alpha_3=\alpha_5=0$ and $\alpha_4=\alpha_5$ would give $s=0$).

Put $p=\alpha_3+\alpha_5$ and $q=\alpha_5$.
From the action formulas,
\[
s^*=d^2(x+z)s,\quad p^*=d^2(x+z)p,\quad q^*=d^2(pz+qd).
\]
Hence $p/s$ is an invariant.
If $p\neq0$, we choose $d\neq0$ arbitrary, $x=qd/p$ and $z\neq-x$; then $q^*=p^*$, i.e.\ $\alpha_3^*=0$, and $\theta\sim\N4+\alpha\N5$ modulo $\N1,\N2$, with $\alpha=p/s\neq0$.
If $p=0$, then $q^*=d^3q$; for $q\neq0$ the ratio $q/s$ is multiplied by $d/(x+z)$ and can be made equal to $-1$, which gives $\N3-\N5$ modulo $\N1,\N2$; for $q=0$ we get $\N4$ modulo $\N1,\N2$.

It remains to deal with $\alpha_1,\alpha_2$.
The determinant of $(U_1,U_2)$ equals $s(\alpha_5-\alpha_4)$.
For $\N3-\N5$ it is $-1$ and for $\N4+\alpha\N5$ it is $\alpha-1$, so in these cases $\alpha_1,\alpha_2$ can be killed unless $\alpha=1$.
For $\theta=\alpha_1\N1+\alpha_2\N2+\N4+\N5$ we have $U_1=(0,-1)$, $U_2=(0,1)$, so we may assume $\alpha_2=0$.
The automorphisms preserving the normal form $\alpha_3=0$, $\alpha_4=\alpha_5$ satisfy $d=x$, and for them $\alpha_1^*=\alpha_1x(x+z)$ (the contribution of the parameters $t,u,v,w$ to $\alpha_1^*$ is $(t+w)(d-x)+(u+v)(x-d)=0$) while $s^*=x^2(x+z)s$.
Thus $\alpha_1/s\mapsto x^{-1}\alpha_1/s$, and we obtain the two orbits $\N4+\N5$ and $\N1+\N4+\N5$.
\end{proof}

\begin{longtable}{@{}l p{0.90\textwidth}@{}}
$\mathcal{A}_{93}(\alpha)$ & $e_{1}e_{1}=e_{3},\ e_{1}e_{2}=e_{4},\ e_{1}e_{4}=\alpha e_{5},\ e_{2}e_{1}=-e_{3},\ e_{2}e_{2}=-e_{4}$ \\ & $e_{2}e_{4}=- \alpha e_{5},\ e_{3}e_{2}=e_{5},\ e_{4}e_{1}=(\alpha - 1) e_{5},\ e_{4}e_{2}=- \alpha e_{5}$ \\ & \quad {\small $\theta=\langle \nabla_{4}+\alpha \nabla_{5}\rangle$;\quad $\alpha\in\mathbb{C}$} \\[4pt]
$\mathcal{A}_{94}$ & $e_{1}e_{1}=e_{3}+e_{5},\ e_{1}e_{2}=e_{4},\ e_{1}e_{4}=e_{5},\ e_{2}e_{1}=-e_{3},\ e_{2}e_{2}=-e_{4},\ e_{2}e_{4}=-e_{5},\ e_{3}e_{2}=e_{5},\ e_{4}e_{2}=-e_{5}$ \\ & \quad {\small $\theta=\langle \nabla_{1}+\nabla_{4}+\nabla_{5}\rangle$} \\[4pt]
$\mathcal{A}_{95}$ & $e_{1}e_{1}=e_{3},\ e_{1}e_{2}=e_{4},\ e_{1}e_{3}=e_{5},\ e_{1}e_{4}=-e_{5},\ e_{2}e_{1}=-e_{3},\ e_{2}e_{2}=-e_{4}$ \\ & $e_{2}e_{3}=-e_{5},\ e_{2}e_{4}=e_{5},\ e_{3}e_{1}=e_{5},\ e_{4}e_{1}=-2 e_{5},\ e_{4}e_{2}=e_{5}$ \\ & \quad {\small $\theta=\langle \nabla_{3}-\nabla_{5}\rangle$} \\[4pt]
\end{longtable}

\begin{remark}
By Proposition~\ref{prop:M71}, every class with $\alpha_3+\alpha_4=0$ (for instance $\N3-\N4$ or $\N3-\N4+\N5$) has $\Ann(\theta)\cap\Ann(\M7(1))\neq0$, so the corresponding extensions are split and do not appear in the classification.
\end{remark}

\textbf{Central extensions of $\M{11}$.} Here $\N1=[\D{11}]$, $\N2=[\D{22}]$, $\N3=[\D{13}+\D{31}]$, $\N4=[\D{14}+\D{31}]$, $\N5=[\D{23}+\D{42}]$, $\N6=[\D{24}+\D{42}]$.
For the automorphisms of the first kind in Proposition~\ref{prop:aut},
\[
\begin{array}{lll}
\alpha_1^*=x\big(\alpha_1x+(2\alpha_3+\alpha_4)z+\alpha_4t\big), & \alpha_3^*=\alpha_3x^2y, & \alpha_4^*=\alpha_4x^2y,\\
\alpha_2^*=y\big(\alpha_2y+\alpha_5u+(\alpha_5+2\alpha_6)v\big), & \alpha_5^*=\alpha_5xy^2, & \alpha_6^*=\alpha_6xy^2,
\end{array}
\]
and $\sigma$ maps $(\alpha_1,\dots,\alpha_6)$ to $(\alpha_2,\alpha_1,\alpha_6,\alpha_5,\alpha_4,\alpha_3)$.

\begin{lemma}\label{lem:M11}
A class $\theta$ is admissible if and only if the vectors $(\alpha_3,\alpha_4)$, $(\alpha_5,\alpha_6)$, $(\alpha_3+\alpha_4,0)$, $(0,\alpha_5+\alpha_6)$ span $\C^2$.
Every admissible class lies in the orbit of one of
\[
\begin{array}{lll}
\alpha\N3+\N4+\N5+\beta\N6\ \big((\alpha,\beta)\neq(-1,-1)\big), & \alpha\N3+\N4+\N6, & \alpha\N3+\N4\ (\alpha\neq-1),\\
\N2+\alpha\N3+\N4\ (\alpha\neq-1), & \N3+\N6, &
\end{array}
\]
and two of these classes lie in the same orbit if and only if they coincide or they are $\alpha\N3+\N4+\N5+\beta\N6$ and $\beta\N3+\N4+\N5+\alpha\N6$.
\end{lemma}

\begin{proof}
The first claim is Lemma~\ref{lem:rank}.
Put $\pi_1=(\alpha_3,\alpha_4)$ and $\pi_2=(\alpha_5,\alpha_6)$.
The connected automorphisms multiply $\pi_1$ by $x^2y$ and $\pi_2$ by $xy^2$, and $\sigma$ interchanges $\pi_1$ and the reversed vector $(\alpha_6,\alpha_5)$ of $\pi_2$.
We can kill $\alpha_1$ if $\pi_1\neq0$ and $\alpha_2$ if $\pi_2\neq0$.
If $\pi_1=\pi_2=0$ the class is not admissible.

If $\pi_1\neq0\neq\pi_2$, then $\alpha_1=\alpha_2=0$, and the orbit is determined by the pair of points $([\alpha_3:\alpha_4],[\alpha_5:\alpha_6])\in\mathbb{P}^1\times\mathbb{P}^1$ modulo the involution $([a:b],[c:d])\mapsto([d:c],[b:a])$, because the ratio of the scalings, $x/y$, is arbitrary.
Every pair is equivalent to $([\alpha:1],[1:\beta])$, $([\alpha:1],[0:1])$ or $([1:0],[0:1])$ (the pair $([1:0],[1:\beta])$ is mapped by $\sigma$ to $([\beta:1],[0:1])$), and the involution identifies only $([\alpha:1],[1:\beta])$ with $([\beta:1],[1:\alpha])$.
The admissibility condition for $([\alpha:1],[1:\beta])$ fails exactly for $\alpha=\beta=-1$ (the $2\times2$ minors are $\alpha\beta-1$, $-(\alpha+1)$, $\alpha(1+\beta)$, $-\beta(\alpha+1)$, $1+\beta$ and $(\alpha+1)(1+\beta)$), while the classes $\alpha\N3+\N4+\N6$ and $\N3+\N6$ are always admissible.

If $\pi_2=0$ (the case $\pi_1=0$ being equivalent to it by $\sigma$), admissibility means $\alpha_4(\alpha_3+\alpha_4)\neq0$, so $\pi_1=[\alpha:1]$ with $\alpha\neq-1$; the coefficient $\alpha_2$ is multiplied by $y^2$ against $x^2y$ for $\alpha_4$, so it can be normalized to $0$ or $1$.
We get $\alpha\N3+\N4$ and $\N2+\alpha\N3+\N4$.
\end{proof}

\begin{longtable}{@{}l p{0.90\textwidth}@{}}
$\mathcal{A}_{96}(\alpha,\beta)$ & $e_{1}e_{2}=e_{3},\ e_{1}e_{3}=\alpha e_{5},\ e_{1}e_{4}=e_{5},\ e_{2}e_{1}=e_{4}$ \\ & $e_{2}e_{3}=e_{5},\ e_{2}e_{4}=\beta e_{5},\ e_{3}e_{1}=(\alpha + 1) e_{5},\ e_{4}e_{2}=(\beta + 1) e_{5}$ \\ & \quad {\small $\theta=\langle \alpha \nabla_{3}+\nabla_{4}+\nabla_{5}+\beta \nabla_{6}\rangle$;\quad $(\alpha,\beta)\neq(-1,-1);\ \mathcal{A}_{96}(\alpha,\beta)\cong\mathcal{A}_{96}(\beta,\alpha)$} \\[4pt]
$\mathcal{A}_{97}(\alpha)$ & $e_{1}e_{2}=e_{3},\ e_{1}e_{3}=\alpha e_{5},\ e_{1}e_{4}=e_{5},\ e_{2}e_{1}=e_{4},\ e_{2}e_{4}=e_{5},\ e_{3}e_{1}=(\alpha + 1) e_{5},\ e_{4}e_{2}=e_{5}$ \\ & \quad {\small $\theta=\langle \alpha \nabla_{3}+\nabla_{4}+\nabla_{6}\rangle$;\quad $\alpha\in\mathbb{C}$} \\[4pt]
$\mathcal{A}_{98}(\alpha)$ & $e_{1}e_{2}=e_{3},\ e_{1}e_{3}=\alpha e_{5},\ e_{1}e_{4}=e_{5},\ e_{2}e_{1}=e_{4},\ e_{3}e_{1}=(\alpha + 1) e_{5}$ \\ & \quad {\small $\theta=\langle \alpha \nabla_{3}+\nabla_{4}\rangle$;\quad $\alpha\neq-1$} \\[4pt]
$\mathcal{A}_{99}(\alpha)$ & $e_{1}e_{2}=e_{3},\ e_{1}e_{3}=\alpha e_{5},\ e_{1}e_{4}=e_{5},\ e_{2}e_{1}=e_{4},\ e_{2}e_{2}=e_{5},\ e_{3}e_{1}=(\alpha + 1) e_{5}$ \\ & \quad {\small $\theta=\langle \nabla_{2}+\alpha \nabla_{3}+\nabla_{4}\rangle$;\quad $\alpha\neq-1$} \\[4pt]
$\mathcal{A}_{100}$ & $e_{1}e_{2}=e_{3},\ e_{1}e_{3}=e_{5},\ e_{2}e_{1}=e_{4},\ e_{2}e_{4}=e_{5},\ e_{3}e_{1}=e_{5},\ e_{4}e_{2}=e_{5}$ \\ & \quad {\small $\theta=\langle \nabla_{3}+\nabla_{6}\rangle$} \\[4pt]
\end{longtable}

\textbf{Central extensions of $\M{12}$.} Here $\N1=[\D{11}]$, $\N2=[\D{12}]$, $\N3=[\D{14}+\D{23}+\D{32}+\D{41}]$, $\N4=[\D{13}+2\D{23}-\D{24}+\D{31}+2\D{32}-\D{42}]$.
For the automorphisms of the first kind in Proposition~\ref{prop:aut} we have $\alpha_3^*=\alpha_3x^3$, $\alpha_4^*=\alpha_4x^3$ and
\[
\begin{array}{l}
\alpha_1^*=x\big(\alpha_1x+2\alpha_3(u-v-y-z)+2\alpha_4(u+v-y-3z)\big),\\
\alpha_2^*=x\big(\alpha_2x+2\alpha_3(v+y)+2\alpha_4(-u+2y+z)\big),
\end{array}
\]
while $\sigma_p$ maps $(\alpha_1,\alpha_2,\alpha_3,\alpha_4)$ to $\big(-p^2(\alpha_1+2\alpha_2),\ p^2\alpha_2,\ p^3(2\alpha_3+3\alpha_4),\ -p^3(\alpha_3+2\alpha_4)\big)$.

\begin{lemma}\label{lem:M12}
The matrices $\sigma_p$ of Proposition~\ref{prop:aut} are automorphisms of $\M{12}$.
A class $\theta$ is admissible if and only if $\alpha_3+\alpha_4\neq0$.
Every admissible class lies in the orbit of a class $\alpha\N3+\N4$ with $\alpha\neq-1$, and $\alpha\N3+\N4$, $\beta\N3+\N4$ lie in the same orbit if and only if $\beta=\alpha$ or $\beta=-\frac{2\alpha+3}{\alpha+2}$ \textup{(}for $\alpha=-2$ the second possibility is void\textup{)}.
\end{lemma}

\begin{proof}
By its columns, the matrix $\sigma_p$ maps $e_1\mapsto pe_2$, $e_2\mapsto pe_1$, $e_3\mapsto-p^2e_3$, $e_4\mapsto p^2(2e_3+e_4)$. We have $\sigma_p(e_1)^2=p^2e_2e_2=p^2(2e_3+e_4)=\sigma_p(e_4)$, $\sigma_p(e_1)\sigma_p(e_2)=p^2e_2e_1=-p^2e_3=\sigma_p(e_3)$, $\sigma_p(e_2)\sigma_p(e_1)=p^2e_3=\sigma_p(-e_3)$ and $\sigma_p(e_2)^2=p^2e_4=\sigma_p(2e_3+e_4)$, so $\sigma_p$ is an automorphism.
By Lemma~\ref{lem:rank}, admissibility means that $(\alpha_4,\alpha_3)$ and $(\alpha_3+2\alpha_4,-\alpha_4)$ span $\C^2$, i.e.\ $-(\alpha_3+\alpha_4)^2\neq0$.
The coefficients of $u$ and $y$ in $(\alpha_1^*,\alpha_2^*)/(2x)$ form the matrix $\begin{pmatrix}\alpha_3+\alpha_4&-(\alpha_3+\alpha_4)\\-\alpha_4&\alpha_3+2\alpha_4\end{pmatrix}$ with determinant $(\alpha_3+\alpha_4)^2\neq0$, so $\alpha_1,\alpha_2$ can be killed.
The connected automorphisms preserve $[\alpha_3:\alpha_4]$, and $\sigma_p$ maps it to $[2\alpha_3+3\alpha_4:-(\alpha_3+2\alpha_4)]$.
In particular $[1:0]$ (the class $\N3$) is mapped to $[-2:1]$, so every admissible class is equivalent to some $\alpha\N3+\N4$, $\alpha\neq-1$, and $\sigma_p$ maps $\alpha\N3+\N4$ to a multiple of $f(\alpha)\N3+\N4$ with $f(\alpha)=-\frac{2\alpha+3}{\alpha+2}$ when $\alpha\neq-2$.
One checks that $f(f(\alpha))=\alpha$, so the orbits are exactly the sets $\{\alpha,f(\alpha)\}$.
\end{proof}

\begin{longtable}{@{}l p{0.90\textwidth}@{}}
$\mathcal{A}_{101}(\alpha)$ & $e_{1}e_{1}=e_{4},\ e_{1}e_{2}=e_{3},\ e_{1}e_{3}=e_{5},\ e_{1}e_{4}=\alpha e_{5},\ e_{2}e_{1}=-e_{3},\ e_{2}e_{2}=2 e_{3}+e_{4}$ \\ & $e_{2}e_{3}=(\alpha + 2) e_{5},\ e_{2}e_{4}=-e_{5},\ e_{3}e_{1}=e_{5},\ e_{3}e_{2}=(\alpha + 2) e_{5},\ e_{4}e_{1}=\alpha e_{5},\ e_{4}e_{2}=-e_{5}$ \\ & \quad {\small $\theta=\langle \alpha \nabla_{3}+\nabla_{4}\rangle$;\quad $\alpha\neq-1;\ \mathcal{A}_{101}(\alpha)\cong\mathcal{A}_{101}\big(-\tfrac{2\alpha+3}{\alpha+2}\big)$} \\[4pt]
\end{longtable}

\begin{remark}
The fixed points of the involution $f$ are $\alpha=-1$, which is excluded, and $\alpha=-3$.
\end{remark}

Central extensions of $\M{13}$. Let us use the following notations:
$$\N1=[\D{11}],\ \N2=[\D{12}],\ \N3=[\D{23}+\D{32}],\ \N4=[\D{31}-\D{42}],\ \N5=[\D{24}+\D{42}]$$ and
\[
\begin{array}{l}
\alpha_1^*=x(\alpha_1x+\alpha_4t),\quad \alpha_2^*=\alpha_1xy+\alpha_2xz+\alpha_3tz+\alpha_4ty+(\alpha_5-\alpha_4)vz,\\
\alpha_3^*=z^2(\alpha_3z+\alpha_5y),\quad \alpha_4^*=\alpha_4xz^2,\quad \alpha_5^*=\alpha_5xz^2 .
\end{array}
\]

\begin{proposition}\label{prop:M13}
A class $\theta$ is admissible if and only if $\alpha_4\neq0$.
Every admissible class lies in the orbit of exactly one of $\N4+\alpha\N5$ $(\alpha\in\C)$, $\N3+\N4$, $\N2+\N4+\N5$, and different values of $\alpha$ give different orbits.
\end{proposition}

\begin{proof}
By Lemma~\ref{lem:rank}, admissibility means that $(\alpha_3,\alpha_5)$, $(\alpha_4,0)$, $(\alpha_3,\alpha_5-\alpha_4)$ span $\C^2$; the minors are $-\alpha_4\alpha_5$, $\alpha_4(\alpha_5-\alpha_4)$ and $-\alpha_3\alpha_4$, and they all vanish if and only if $\alpha_4=0$.
Let $\alpha_4\neq0$; then $\alpha=\alpha_5/\alpha_4$ is an invariant and $\alpha_1$ is killed by $t$.
If $\alpha\neq0,1$, we kill $\alpha_3$ by $y$ and $\alpha_2$ by $v$.
If $\alpha=1$, we kill $\alpha_3$ by $y$; the stabilizer then has $t=y=0$ and $\alpha_2/\alpha_4\mapsto z^{-1}\alpha_2/\alpha_4$, giving $\N4+\N5$ and $\N2+\N4+\N5$.
If $\alpha=0$, we kill $\alpha_2$ by $v$, and $\alpha_3/\alpha_4\mapsto(z/x)\alpha_3/\alpha_4$, giving $\N4$ and $\N3+\N4$.
\end{proof}

\begin{longtable}{@{}l p{0.80\textwidth}@{}}
$\mathcal{A}_{102}(\alpha)$ & $e_{2}e_{1}=e_{4},\ e_{2}e_{2}=e_{3},\ e_{2}e_{4}=\alpha e_{5},\ e_{3}e_{1}=e_{5},\ e_{4}e_{2}=(\alpha - 1) e_{5}$ \\ & \quad {\small $\theta=\langle \nabla_{4}+\alpha \nabla_{5}\rangle$;\quad $\alpha\in\mathbb{C}$} \\[4pt]
$\mathcal{A}_{103}$ & $e_{2}e_{1}=e_{4},\ e_{2}e_{2}=e_{3},\ e_{2}e_{3}=e_{5},\ e_{3}e_{1}=e_{5},\ e_{3}e_{2}=e_{5},\ e_{4}e_{2}=-e_{5}$ \\ & \quad {\small $\theta=\langle \nabla_{3}+\nabla_{4}\rangle$} \\[4pt]
$\mathcal{A}_{104}$ & $e_{1}e_{2}=e_{5},\ e_{2}e_{1}=e_{4},\ e_{2}e_{2}=e_{3},\ e_{2}e_{4}=e_{5},\ e_{3}e_{1}=e_{5}$ \\ & \quad {\small $\theta=\langle \nabla_{2}+\nabla_{4}+\nabla_{5}\rangle$} \\[4pt]
\end{longtable}

\textbf{Central extensions of $\M{14}(\lambda)$, $\lambda\neq0,-1$.} Let us use the following notations:
$$\N1=[\D{11}],\ \N2=[\D{12}],\ \N3=[(\lambda+1)\D{31}+\D{24}],\ \N4=[\D{23}+\D{32}],\ \N5=[\D{13}-\lambda\D{31}+\D{42}],$$ and
\[
\begin{array}{l}
\alpha_1^*=x\big(\alpha_1x+c_1t\big),\quad c_1=(\lambda+1)\alpha_3+(1-\lambda)\alpha_5,\\
\lambda\alpha_2^*=(\lambda-1)\alpha_1xy+\lambda\alpha_2xz+(\lambda\alpha_5-\alpha_3)uz+\big(2\lambda\alpha_5-(\lambda+1)\alpha_3\big)vx+(\lambda-1)\alpha_4tz\\
\hphantom{\lambda\alpha_2^*=}+\big((\lambda^2+\lambda)\alpha_3-(\lambda^2+1)\alpha_5\big)ty,\\
\alpha_3^*=\alpha_3xz^2,\quad \alpha_4^*=z^2(c_2y+\alpha_4z),\quad c_2=(\lambda+1)\alpha_3+\alpha_5,\quad \alpha_5^*=\alpha_5xz^2 .
\end{array}
\]

\begin{proposition}\label{prop:M14}
Let $\lambda\neq0,-1$.
A class $\theta$ is admissible if and only if $(\alpha_3,\alpha_5)\neq(0,0)$.
Every admissible class lies in the orbit of exactly one of
\[
\N3+\alpha\N5\ (\alpha\in\C),\quad \N5,\quad (\lambda-1)(\N1+\N3)+(\lambda+1)\N5\ (\lambda\neq1),\quad \N3+\N4-(\lambda+1)\N5,
\]
and, if $\lambda=1$, of one of the two further classes $\N1+\N5$ and $\N2+\N3+\N5$.
Different values of $\alpha$ give different orbits.
\end{proposition}

\begin{proof}
By Lemma~\ref{lem:rank}, admissibility means that $(\alpha_5,0)$, $(\alpha_4,\alpha_3)$, $((\lambda+1)\alpha_3-\lambda\alpha_5,0)$ and $(\alpha_4,\alpha_5)$ span $\C^2$; if $\alpha_5\ne0$ the first and last vectors are independent, if $\alpha_5=0\neq\alpha_3$ the second and third have determinant $-(\lambda+1)\alpha_3^2\neq0$, and if $\alpha_3=\alpha_5=0$ all second coordinates vanish.
The point $[\alpha_3:\alpha_5]$ is an invariant, and so are the conditions $c_1=0$ and $c_2=0$.
We can kill $\alpha_1$ by $t$ if $c_1\neq0$, then $\alpha_4$ by $y$ if $c_2\neq0$, and finally $\alpha_2$ by $u$ or $v$, unless $\lambda\alpha_5=\alpha_3$ and $2\lambda\alpha_5=(\lambda+1)\alpha_3$, i.e.\ unless $\lambda=1$ and $\alpha_3=\alpha_5$.
Note that $c_1=c_2=0$ would imply $\lambda(\lambda+1)\alpha_3=0$, hence $\alpha_3=\alpha_5=0$.

If $c_1c_2\neq0$ and $(\lambda,\alpha_3)\neq(1,\alpha_5)$, we get $\N3+\alpha\N5$ or $\N5$.
If $c_2=0$, then $(\alpha_3,\alpha_5)\sim(1,-(\lambda+1))$, $c_1=\lambda(\lambda+1)\neq0$, and $\alpha_4/\alpha_3\mapsto(z/x)\alpha_4/\alpha_3$; we get $\N3-(\lambda+1)\N5$ (a member of the family) or $\N3+\N4-(\lambda+1)\N5$.
If $c_1=0$ and $\lambda\neq1$, then $(\alpha_3,\alpha_5)\sim(\lambda-1,\lambda+1)$, $c_2=\lambda(\lambda+1)\neq0$ and $\alpha_1/\alpha_3\mapsto(x/z^2)\alpha_1/\alpha_3$; we get the member $\alpha=\frac{\lambda+1}{\lambda-1}$ of the family, or $(\lambda-1)(\N1+\N3)+(\lambda+1)\N5$.
If $c_1=0$ and $\lambda=1$, then $\alpha_3=0$, and similarly we get $\N5$ or $\N1+\N5$.
Finally, if $\lambda=1$ and $\alpha_3=\alpha_5$, then $c_1=2\alpha_3$, $c_2=3\alpha_3$, we kill $\alpha_1,\alpha_4$, and the stabilizer acts by $\alpha_2\mapsto xz\alpha_2$, $\alpha_3\mapsto xz^2\alpha_3$ (all other contributions to $\alpha_2^*$ cancel for $\lambda=1$, $\alpha_3=\alpha_5$); we get $\N3+\N5$ or $\N2+\N3+\N5$.
\end{proof}

\begin{longtable}{@{}l p{0.90\textwidth}@{}}
$\mathcal{A}_{105}(\alpha,\lambda)$ & $e_{1}e_{2}=e_{4},\ e_{1}e_{3}=\alpha e_{5},\ e_{2}e_{1}=\lambda e_{4},\ e_{2}e_{2}=e_{3}$ \\ & $e_{2}e_{4}=e_{5},\ e_{3}e_{1}=(- \alpha \lambda + \lambda + 1) e_{5},\ e_{4}e_{2}=\alpha e_{5}$ \\ & \quad {\small $\theta=\langle \nabla_{3}+\alpha \nabla_{5}\rangle$;\quad $\lambda\neq0,-1,\ \alpha\in\mathbb{C}$} \\[4pt]
$\mathcal{A}_{106}(\lambda)$ & $e_{1}e_{2}=e_{4},\ e_{1}e_{3}=e_{5},\ e_{2}e_{1}=\lambda e_{4},\ e_{2}e_{2}=e_{3},\ e_{3}e_{1}=- \lambda e_{5},\ e_{4}e_{2}=e_{5}$ \\ & \quad {\small $\theta=\langle \nabla_{5}\rangle$;\quad $\lambda\neq0,-1$} \\[4pt]
$\mathcal{A}_{107}(\lambda)$ & $e_{1}e_{1}=(\lambda - 1) e_{5},\ e_{1}e_{2}=e_{4},\ e_{1}e_{3}=(\lambda + 1) e_{5},\ e_{2}e_{1}=\lambda e_{4}$ \\ & $e_{2}e_{2}=e_{3},\ e_{2}e_{4}=(\lambda - 1) e_{5},\ e_{3}e_{1}=(- \lambda - 1) e_{5},\ e_{4}e_{2}=(\lambda + 1) e_{5}$ \\ & \quad {\small $\theta=\langle (\lambda - 1) \nabla_{1}+(\lambda - 1) \nabla_{3}+(\lambda + 1) \nabla_{5}\rangle$;\quad $\lambda\neq0,\pm1$} \\[4pt]
$\mathcal{A}_{108}(\lambda)$ & $e_{1}e_{2}=e_{4},\ e_{1}e_{3}=(- \lambda - 1) e_{5},\ e_{2}e_{1}=\lambda e_{4},\ e_{2}e_{2}=e_{3},\ e_{2}e_{3}=e_{5}$ \\ & $e_{2}e_{4}=e_{5},\ e_{3}e_{1}=(\lambda^{2} + 2 \lambda + 1) e_{5},\ e_{3}e_{2}=e_{5},\ e_{4}e_{2}=(- \lambda - 1) e_{5}$ \\ & \quad {\small $\theta=\langle \nabla_{3}+\nabla_{4}+(- \lambda - 1) \nabla_{5}\rangle$;\quad $\lambda\neq0,-1$} \\[4pt]
$\mathcal{A}_{109}$ & $e_{1}e_{1}=e_{5},\ e_{1}e_{2}=e_{4},\ e_{1}e_{3}=e_{5},\ e_{2}e_{1}=e_{4},\ e_{2}e_{2}=e_{3},\ e_{3}e_{1}=-e_{5},\ e_{4}e_{2}=e_{5}$ \\ & \quad {\small $\theta=\langle \nabla_{1}+\nabla_{5}\rangle$;\quad $\text{extension of }\mathcal{M}_{14}(1)$} \\[4pt]
$\mathcal{A}_{110}$ & $e_{1}e_{2}=e_{4}+e_{5},\ e_{1}e_{3}=e_{5},\ e_{2}e_{1}=e_{4},\ e_{2}e_{2}=e_{3},\ e_{2}e_{4}=e_{5},\ e_{3}e_{1}=e_{5},\ e_{4}e_{2}=e_{5}$ \\ & \quad {\small $\theta=\langle \nabla_{2}+\nabla_{3}+\nabla_{5}\rangle$;\quad $\text{extension of }\mathcal{M}_{14}(1)$} \\[4pt]
\end{longtable}

\textbf{Central extensions of $\M{14}(0)$.} Here $\N1,\dots,\N6$ are the classes of $\D{11}$, $\D{21}$, $\D{24}+\D{31}$, $\D{23}+\D{32}$, $\D{14}+\D{41}$, $\D{13}+\D{42}$, and
\[
\begin{array}{l}
\alpha_1^*=x\big(\alpha_1x+(\alpha_3+\alpha_6)t+2\alpha_5u\big),\\
\alpha_2^*=\alpha_1xy+\alpha_2xz+\alpha_3(uz+vx)+\alpha_4tz+\alpha_5(uy+wx)+\alpha_6ty,\\
\alpha_3^*=xz(\alpha_3z+\alpha_5y),\quad \alpha_4^*=z\big(\alpha_4z^2+(\alpha_3+\alpha_6)yz+\alpha_5y^2\big),\\
\alpha_5^*=\alpha_5x^2z,\quad \alpha_6^*=xz(\alpha_5y+\alpha_6z).
\end{array}
\]

\begin{proposition}\label{prop:M140}
A class $\theta$ is admissible if and only if either $\alpha_5\neq0$ and $(\alpha_3-\alpha_6,\ \alpha_4\alpha_5-\alpha_3^2)\neq(0,0)$, or $\alpha_5=0$ and $(\alpha_3,\alpha_6)\neq(0,0)$.
Every admissible class lies in the orbit of exactly one of
\[
\begin{array}{lllll}
\N3+\alpha\N4+\N5, & \N4+\N5, & \N3+\alpha\N6\ (\alpha\neq-1), & \N6, & \N2+\N6,\\
\N1+\N3+\N4-\N6, & \N1+\N3-\N6, & \N3+\N4-\N6, & \N3-\N6, &
\end{array}
\]
where $\alpha\in\C$ unless stated otherwise, and different values of $\alpha$ give different orbits.
\end{proposition}

\begin{proof}
By Lemma~\ref{lem:rank}, admissibility means that $(\alpha_6,\alpha_5)$, $(\alpha_4,\alpha_3)$, $(\alpha_3,\alpha_5)$, $(\alpha_4,\alpha_6)$ span $\C^2$.
If $\alpha_5=0$ this holds if and only if $(\alpha_3,\alpha_6)\neq0$.
If $\alpha_5\neq0$ it fails if and only if all four vectors are proportional to $(\alpha_6,\alpha_5)$, i.e.\ $\alpha_3=\alpha_6$ and $\alpha_4\alpha_5=\alpha_3\alpha_6=\alpha_3^2$; after the normalization $\alpha_6=0$ below, this means $\alpha_3=\alpha_4=0$.
The condition $\alpha_5=0$ is invariant.

\emph{Case $\alpha_5\neq0$.}
We kill $\alpha_1$ by $u$, $\alpha_6$ by $y$ and $\alpha_2$ by $w$.
Since $\alpha_3^*-\alpha_6^*=xz^2(\alpha_3-\alpha_6)$, the remaining coefficients transform (with $y=0$) by $\alpha_3\mapsto xz^2\alpha_3$, $\alpha_4\mapsto z^3\alpha_4$, $\alpha_5\mapsto x^2z\alpha_5$.
If $\alpha_3\neq0$, the normalization $\alpha_3=\alpha_5$ forces $z=x$, and then $\alpha_4/\alpha_5$ is an invariant: $\N3+\alpha\N4+\N5$.
If $\alpha_3=0$, admissibility forces $\alpha_4\neq0$ and we get $\N4+\N5$.

\emph{Case $\alpha_5=0$.}
Then $[\alpha_3:\alpha_6]$ is an invariant.
If $\alpha_3+\alpha_6\neq0$, we kill $\alpha_1$ by $t$ and $\alpha_4$ by $y$; if moreover $\alpha_3\neq0$, we kill $\alpha_2$ by $v$ and obtain $\N3+\alpha\N6$, $\alpha\neq-1$; if $\alpha_3=0$, the stabilizer has $t=y=0$ and $\alpha_2/\alpha_6\mapsto z^{-1}\alpha_2/\alpha_6$, giving $\N6$ and $\N2+\N6$.
If $\alpha_3+\alpha_6=0$, we normalize $\alpha_3=1$, $\alpha_6=-1$ and kill $\alpha_2$ by $v$; now $\alpha_1\mapsto x^2\alpha_1$ and $\alpha_4\mapsto z^3\alpha_4$ while $\alpha_3\mapsto xz^2\alpha_3$, so that $\alpha_1/\alpha_3$ and $\alpha_4/\alpha_3$ are multiplied by the independent factors $x/z^2$ and $z/x$.
We obtain the four classes $\N3-\N6$, $\N1+\N3-\N6$, $\N3+\N4-\N6$, $\N1+\N3+\N4-\N6$.
\end{proof}

\begin{longtable}{@{}l p{0.90\textwidth}@{}}
$\mathcal{A}_{111}(\alpha)$ & $e_{1}e_{2}=e_{4},\ e_{1}e_{4}=e_{5},\ e_{2}e_{2}=e_{3},\ e_{2}e_{3}=\alpha e_{5},\ e_{2}e_{4}=e_{5},\ e_{3}e_{1}=e_{5},\ e_{3}e_{2}=\alpha e_{5},\ e_{4}e_{1}=e_{5}$ \\ & \quad {\small $\theta=\langle \nabla_{3}+\alpha \nabla_{4}+\nabla_{5}\rangle$;\quad $\alpha\in\mathbb{C}$} \\[4pt]
$\mathcal{A}_{112}$ & $e_{1}e_{2}=e_{4},\ e_{1}e_{4}=e_{5},\ e_{2}e_{2}=e_{3},\ e_{2}e_{3}=e_{5},\ e_{3}e_{2}=e_{5},\ e_{4}e_{1}=e_{5}$ \\ & \quad {\small $\theta=\langle \nabla_{4}+\nabla_{5}\rangle$} \\[4pt]
$\mathcal{A}_{113}(\alpha)$ & $e_{1}e_{2}=e_{4},\ e_{1}e_{3}=\alpha e_{5},\ e_{2}e_{2}=e_{3},\ e_{2}e_{4}=e_{5},\ e_{3}e_{1}=e_{5},\ e_{4}e_{2}=\alpha e_{5}$ \\ & \quad {\small $\theta=\langle \nabla_{3}+\alpha \nabla_{6}\rangle$;\quad $\alpha\neq-1$} \\[4pt]
$\mathcal{A}_{114}$ & $e_{1}e_{2}=e_{4},\ e_{1}e_{3}=e_{5},\ e_{2}e_{2}=e_{3},\ e_{4}e_{2}=e_{5}$ \\ & \quad {\small $\theta=\langle \nabla_{6}\rangle$} \\[4pt]
$\mathcal{A}_{115}$ & $e_{1}e_{2}=e_{4},\ e_{1}e_{3}=e_{5},\ e_{2}e_{1}=e_{5},\ e_{2}e_{2}=e_{3},\ e_{4}e_{2}=e_{5}$ \\ & \quad {\small $\theta=\langle \nabla_{2}+\nabla_{6}\rangle$} \\[4pt]
$\mathcal{A}_{116}$ & $e_{1}e_{1}=e_{5},\ e_{1}e_{2}=e_{4},\ e_{1}e_{3}=-e_{5},\ e_{2}e_{2}=e_{3},\ e_{2}e_{3}=e_{5}$ \\ & $e_{2}e_{4}=e_{5},\ e_{3}e_{1}=e_{5},\ e_{3}e_{2}=e_{5},\ e_{4}e_{2}=-e_{5}$ \\ & \quad {\small $\theta=\langle \nabla_{1}+\nabla_{3}+\nabla_{4}-\nabla_{6}\rangle$} \\[4pt]
$\mathcal{A}_{117}$ & $e_{1}e_{1}=e_{5},\ e_{1}e_{2}=e_{4},\ e_{1}e_{3}=-e_{5},\ e_{2}e_{2}=e_{3},\ e_{2}e_{4}=e_{5},\ e_{3}e_{1}=e_{5},\ e_{4}e_{2}=-e_{5}$ \\ & \quad {\small $\theta=\langle \nabla_{1}+\nabla_{3}-\nabla_{6}\rangle$} \\[4pt]
$\mathcal{A}_{118}$ & $e_{1}e_{2}=e_{4},\ e_{1}e_{3}=-e_{5},\ e_{2}e_{2}=e_{3},\ e_{2}e_{3}=e_{5},\ e_{2}e_{4}=e_{5},\ e_{3}e_{1}=e_{5},\ e_{3}e_{2}=e_{5},\ e_{4}e_{2}=-e_{5}$ \\ & \quad {\small $\theta=\langle \nabla_{3}+\nabla_{4}-\nabla_{6}\rangle$} \\[4pt]
$\mathcal{A}_{119}$ & $e_{1}e_{2}=e_{4},\ e_{1}e_{3}=-e_{5},\ e_{2}e_{2}=e_{3},\ e_{2}e_{4}=e_{5},\ e_{3}e_{1}=e_{5},\ e_{4}e_{2}=-e_{5}$ \\ & \quad {\small $\theta=\langle \nabla_{3}-\nabla_{6}\rangle$} \\[4pt]
\end{longtable}
 
\textbf{Central extensions of $\M{14}(-1)$.} Here $\N1,\dots,\N6$ are the classes of $\D{11}$, $\D{12}$, $\D{14}$, $\D{24}$, $\D{23}+\D{32}$, $\D{13}+\D{31}+\D{42}$, and
\[
\begin{array}{l}
\alpha_1^*=x(\alpha_1x+\alpha_3u+2\alpha_6t),\\
\alpha_2^*=2\alpha_1xy+\alpha_2xz+\alpha_3(uy+wx)+\alpha_4uz+2\alpha_5tz+2\alpha_6(ty+vx)+\alpha_6uz,\\
\alpha_3^*=\alpha_3x^2z,\quad\alpha_4^*=xz(\alpha_3y+\alpha_4z),\quad\alpha_5^*=z^2(\alpha_5z+\alpha_6y),\quad\alpha_6^*=\alpha_6xz^2 .
\end{array}
\]

\begin{proposition}\label{prop:M14m1}
A class $\theta$ is admissible if and only if $\alpha_6\neq0$, or $\alpha_6=0$, $\alpha_5\neq0$ and $(\alpha_3,\alpha_4)\neq(0,0)$.
Every admissible class lies in the orbit of exactly one of
\[
\N3+\alpha\N4+\N6,\quad \alpha\N4+\N6,\quad \N3+\N5,\quad \N4+\N5,\quad \N1+\N4+\N5\quad(\alpha\in\C),
\]
and different values of $\alpha$ give different orbits.
\end{proposition}

\begin{proof}
By Lemma~\ref{lem:rank}, admissibility means that $(\alpha_6,\alpha_3)$, $(\alpha_5,\alpha_4)$, $(\alpha_6,0)$, $(\alpha_5,\alpha_6)$ span $\C^2$, which gives the first claim.
If $\alpha_6\neq0$, we kill $\alpha_1$ by $t$, $\alpha_2$ by $v$ and $\alpha_5$ by $y$; the stabilizer has $y=0$ and acts by $\alpha_3\mapsto x^2z\alpha_3$, $\alpha_4\mapsto xz^2\alpha_4$, $\alpha_6\mapsto xz^2\alpha_6$.
Hence $\alpha=\alpha_4/\alpha_6$ is an invariant and $\alpha_3/\alpha_6\mapsto(x/z)\alpha_3/\alpha_6$, giving $\N3+\alpha\N4+\N6$ and $\alpha\N4+\N6$.
If $\alpha_6=0$, then $\alpha_5\neq0$ and we kill $\alpha_2$ by $t$.
If $\alpha_3\neq0$, we kill $\alpha_4$ by $y$ and $\alpha_1$ by $u$, and normalize $\alpha_3/\alpha_5\mapsto(x/z)^2\alpha_3/\alpha_5$: this gives $\N3+\N5$.
If $\alpha_3=0$, then $\alpha_4\neq0$, and $\alpha_4/\alpha_5$, $\alpha_1/\alpha_5$ are multiplied by the independent factors $x/z$ and $x^2/z^3$; we get $\N4+\N5$ and $\N1+\N4+\N5$.
\end{proof}

\begin{longtable}{@{}l p{0.90\textwidth}@{}}
$\mathcal{A}_{120}(\alpha)$ & $e_{1}e_{2}=e_{4},\ e_{1}e_{3}=e_{5},\ e_{1}e_{4}=e_{5},\ e_{2}e_{1}=-e_{4},\ e_{2}e_{2}=e_{3},\ e_{2}e_{4}=\alpha e_{5},\ e_{3}e_{1}=e_{5},\ e_{4}e_{2}=e_{5}$ \\ & \quad {\small $\theta=\langle \nabla_{3}+\alpha \nabla_{4}+\nabla_{6}\rangle$;\quad $\alpha\in\mathbb{C}$} \\[4pt]
$\mathcal{A}_{121}(\alpha)$ & $e_{1}e_{2}=e_{4},\ e_{1}e_{3}=e_{5},\ e_{2}e_{1}=-e_{4},\ e_{2}e_{2}=e_{3},\ e_{2}e_{4}=\alpha e_{5},\ e_{3}e_{1}=e_{5},\ e_{4}e_{2}=e_{5}$ \\ & \quad {\small $\theta=\langle \alpha \nabla_{4}+\nabla_{6}\rangle$;\quad $\alpha\in\mathbb{C}$} \\[4pt]
$\mathcal{A}_{122}$ & $e_{1}e_{2}=e_{4},\ e_{1}e_{4}=e_{5},\ e_{2}e_{1}=-e_{4},\ e_{2}e_{2}=e_{3},\ e_{2}e_{3}=e_{5},\ e_{3}e_{2}=e_{5}$ \\ & \quad {\small $\theta=\langle \nabla_{3}+\nabla_{5}\rangle$} \\[4pt]
$\mathcal{A}_{123}$ & $e_{1}e_{2}=e_{4},\ e_{2}e_{1}=-e_{4},\ e_{2}e_{2}=e_{3},\ e_{2}e_{3}=e_{5},\ e_{2}e_{4}=e_{5},\ e_{3}e_{2}=e_{5}$ \\ & \quad {\small $\theta=\langle \nabla_{4}+\nabla_{5}\rangle$} \\[4pt]
$\mathcal{A}_{124}$ & $e_{1}e_{1}=e_{5},\ e_{1}e_{2}=e_{4},\ e_{2}e_{1}=-e_{4},\ e_{2}e_{2}=e_{3},\ e_{2}e_{3}=e_{5},\ e_{2}e_{4}=e_{5},\ e_{3}e_{2}=e_{5}$ \\ & \quad {\small $\theta=\langle \nabla_{1}+\nabla_{4}+\nabla_{5}\rangle$} \\[4pt]
\end{longtable}
\section{The classification}\label{sec:main}

We now collect the results of Sections~\ref{sec:ann2}--\ref{sec:twostep}.

\begin{theorem}\label{thm:main}
Let $\mathcal{B}$ be a five-dimensional complex nilpotent right alternative algebra without annihilator component which is not $2$-step nilpotent.
Then $\mathcal{B}$ is isomorphic to exactly one of the algebras $\A1,\dots,\A{124}$ listed below (after the statement), with the following exceptions:
\begin{gather*}
\A{92}(\alpha,\gamma)\cong\A{92}(\gamma^3\alpha^{-1},\gamma) (\gamma\neq0),\
\A{96}(\alpha,\beta)\cong\A{96}(\beta,\alpha),\ \A{101}(\alpha)\cong\A{101}\Big(-\frac{2\alpha+3}{\alpha+2}\Big) (\alpha\neq-2).
\end{gather*}
Here all parameters are arbitrary complex numbers subject to the indicated restrictions.
The algebras $\A1,\dots,\A{29}$ have a two-dimensional annihilator; the remaining ones have a one-dimensional annihilator.
The quotient $\mathcal{B}/\Ann(\mathcal{B})$ is $\RR31$ for $\A1$--$\A{10}$, $\RR33$ for $\A{11}$--$\A{20}$, $\RR34(0)$ for $\A{21}$--$\A{29}$, $\RR41$ for $\A{30}$--$\A{32}$, $\RR44$ for $\A{33}$, $\RR45$ for $\A{34}$--$\A{41}$, $\RR46$ for $\A{42}$--$\A{46}$, $\RR49$ for $\A{47}$, $\M1$ for $\A{48}$--$\A{53}$, $\M2$ for $\A{54}$, $\A{55}$, $\M3$ for $\A{56}$--$\A{72}$, $\M4(0)$ for $\A{73}$--$\A{87}$, $\M6$ for $\A{88}$--$\A{91}$, $\M7(\gamma)$ for $\A{92}(\alpha,\gamma)$, $\M7(1)$ for $\A{93}$--$\A{95}$, $\M{11}$ for $\A{96}$--$\A{100}$, $\M{12}$ for $\A{101}$, $\M{13}$ for $\A{102}$--$\A{104}$, $\M{14}(\lambda)$ for $\A{105}$--$\A{108}$, $\M{14}(1)$ for $\A{109}$, $\A{110}$, $\M{14}(0)$ for $\A{111}$--$\A{119}$ and $\M{14}(-1)$ for $\A{120}$--$\A{124}$.
\end{theorem}

\noindent The algebras of Theorem~\ref{thm:main} (only nonzero products are written; the parameter restrictions are indicated below each algebra):
\begin{longtable}{@{}l p{0.90\textwidth}@{}}
$\mathcal{A}_{1}$ & $e_{1}e_{1}=e_{2},\ e_{1}e_{2}=e_{4},\ e_{2}e_{1}=e_{4},\ e_{3}e_{3}=e_{5}$ \\[4pt]
$\mathcal{A}_{2}$ & $e_{1}e_{1}=e_{2},\ e_{1}e_{2}=e_{4},\ e_{1}e_{3}=e_{4},\ e_{2}e_{1}=e_{4}, \ e_{3}e_{3}=e_{5}$ \\[4pt]
$\mathcal{A}_{3}$ & $e_{1}e_{1}=e_{2},\ e_{1}e_{2}=e_{4},\ e_{1}e_{3}=e_{5},\ e_{2}e_{1}=e_{4}, \ e_{3}e_{3}=e_{5}$ \\[4pt]
$\mathcal{A}_{4}$ & $e_{1}e_{1}=e_{2},\ e_{1}e_{2}=e_{4},\ e_{1}e_{3}=e_{4}+e_{5},\ e_{2}e_{1}=e_{4}, \ e_{3}e_{3}=e_{5}$ \\[4pt]
$\mathcal{A}_{5}(\alpha)$ & $e_{1}e_{1}=e_{2},\ e_{1}e_{2}=e_{4},\ e_{1}e_{3}=e_{5},\ e_{2}e_{1}=e_{4}, \ e_{3}e_{1}=\alpha e_{5},\ e_{3}e_{3}=e_{4}, \ \alpha\in\mathbb{C}$ \\[4pt]
$\mathcal{A}_{6}$ & $e_{1}e_{1}=e_{2},\ e_{1}e_{2}=e_{4},\ e_{2}e_{1}=e_{4},\ e_{3}e_{1}=e_{5}, \ e_{3}e_{3}=e_{4}$ \\[4pt]
$\mathcal{A}_{7}$ & $e_{1}e_{1}=e_{2},\ e_{1}e_{2}=e_{4},\ e_{1}e_{3}=e_{4}+e_{5},\ e_{2}e_{1}=e_{4}, \ e_{3}e_{1}=e_{5},\ e_{3}e_{3}=e_{4}$ \\[4pt]
$\mathcal{A}_{8}(\alpha)$ & $e_{1}e_{1}=e_{2},\ e_{1}e_{2}=e_{4},\ e_{1}e_{3}=e_{5},\ e_{2}e_{1}=e_{4}, \ e_{3}e_{1}=\alpha e_{5}, \ \alpha\in\mathbb{C}$ \\[4pt]
$\mathcal{A}_{9}$ & $e_{1}e_{1}=e_{2},\ e_{1}e_{2}=e_{4},\ e_{2}e_{1}=e_{4},\ e_{3}e_{1}=e_{5}$ \\[4pt]
$\mathcal{A}_{10}$ & $e_{1}e_{1}=e_{2},\ e_{1}e_{2}=e_{4},\ e_{1}e_{3}=e_{4}+e_{5},\ e_{2}e_{1}=e_{4}, \ e_{3}e_{1}=e_{5}$ \\[4pt]
$\mathcal{A}_{11}$ & $e_{1}e_{2}=e_{3},\ e_{1}e_{3}=e_{4},\ e_{2}e_{1}=-e_{3},\ e_{2}e_{3}=e_{5}$ \\[4pt]
$\mathcal{A}_{12}$ & $e_{1}e_{1}=e_{5},\ e_{1}e_{2}=e_{3},\ e_{1}e_{3}=e_{4},\ e_{2}e_{1}=-e_{3}, \ e_{2}e_{3}=e_{5}$ \\[4pt]
$\mathcal{A}_{13}$ & $e_{1}e_{2}=e_{3}+e_{5},\ e_{1}e_{3}=e_{4},\ e_{2}e_{1}=-e_{3},\ e_{2}e_{3}=e_{5}$ \\[4pt]
$\mathcal{A}_{14}$ & $e_{1}e_{1}=e_{5},\ e_{1}e_{2}=e_{3},\ e_{1}e_{3}=e_{4},\ e_{2}e_{1}=-e_{3}, \ e_{2}e_{2}=e_{4},\ e_{2}e_{3}=e_{5}$ \\[4pt]
$\mathcal{A}_{15}$ & $e_{1}e_{1}=e_{5},\ e_{1}e_{2}=e_{3},\ e_{1}e_{3}=e_{4},\ e_{2}e_{1}=-e_{3}$ \\[4pt]
$\mathcal{A}_{16}$ & $e_{1}e_{1}=e_{5},\ e_{1}e_{2}=e_{3},\ e_{1}e_{3}=e_{4},\ e_{2}e_{1}=-e_{3}, \ e_{2}e_{2}=e_{4}$ \\[4pt]
$\mathcal{A}_{17}$ & $e_{1}e_{2}=e_{3}+e_{5},\ e_{1}e_{3}=e_{4},\ e_{2}e_{1}=-e_{3}$ \\[4pt]
$\mathcal{A}_{18}$ & $e_{1}e_{2}=e_{3}+e_{5},\ e_{1}e_{3}=e_{4},\ e_{2}e_{1}=-e_{3},\ e_{2}e_{2}=e_{4}$ \\[4pt]
$\mathcal{A}_{19}$ & $e_{1}e_{2}=e_{3},\ e_{1}e_{3}=e_{4},\ e_{2}e_{1}=-e_{3},\ e_{2}e_{2}=e_{5}$ \\[4pt]
$\mathcal{A}_{20}$ & $e_{1}e_{1}=e_{5},\ e_{1}e_{2}=e_{3},\ e_{1}e_{3}=e_{4},\ e_{2}e_{1}=-e_{3}, \ e_{2}e_{2}=e_{5}$ \\[4pt]
$\mathcal{A}_{21}(\alpha)$ & $e_{1}e_{1}=e_{5},\ e_{2}e_{1}=e_{3},\ e_{2}e_{2}=e_{3}+\alpha e_{5},\ e_{2}e_{3}=e_{4}, \ e_{3}e_{2}=e_{4}, \ \alpha\in\mathbb{C}$ \\[4pt]
$\mathcal{A}_{22}$ & $e_{1}e_{1}=e_{5},\ e_{1}e_{2}=e_{5},\ e_{2}e_{1}=e_{3},\ e_{2}e_{2}=e_{3}, \ e_{2}e_{3}=e_{4},\ e_{3}e_{2}=e_{4}$ \\[4pt]
$\mathcal{A}_{23}$ & $e_{1}e_{1}=e_{5},\ e_{1}e_{2}=e_{5},\ e_{2}e_{1}=e_{3},\ e_{2}e_{2}=e_{3}+e_{5}, \ e_{2}e_{3}=e_{4},\ e_{3}e_{2}=e_{4}$ \\[4pt]
$\mathcal{A}_{24}$ & $e_{1}e_{2}=e_{5},\ e_{2}e_{1}=e_{3},\ e_{2}e_{2}=e_{3},\ e_{2}e_{3}=e_{4}, \ e_{3}e_{2}=e_{4}$ \\[4pt]
$\mathcal{A}_{25}$ & $e_{1}e_{1}=e_{4},\ e_{1}e_{2}=e_{5},\ e_{2}e_{1}=e_{3},\ e_{2}e_{2}=e_{3}, \ e_{2}e_{3}=e_{4},\ e_{3}e_{2}=e_{4}$ \\[4pt]
$\mathcal{A}_{26}$ & $e_{1}e_{2}=e_{5},\ e_{2}e_{1}=e_{3},\ e_{2}e_{2}=e_{3}+e_{5},\ e_{2}e_{3}=e_{4}, \ e_{3}e_{2}=e_{4}$ \\[4pt]
$\mathcal{A}_{27}$ & $e_{1}e_{1}=e_{4},\ e_{1}e_{2}=e_{5},\ e_{2}e_{1}=e_{3},\ e_{2}e_{2}=e_{3}+e_{5}, \ e_{2}e_{3}=e_{4},\ e_{3}e_{2}=e_{4}$ \\[4pt]
$\mathcal{A}_{28}$ & $e_{2}e_{1}=e_{3},\ e_{2}e_{2}=e_{3}+e_{5},\ e_{2}e_{3}=e_{4},\ e_{3}e_{2}=e_{4}$ \\[4pt]
$\mathcal{A}_{29}$ & $e_{1}e_{1}=e_{4},\ e_{2}e_{1}=e_{3},\ e_{2}e_{2}=e_{3}+e_{5},\ e_{2}e_{3}=e_{4}, \ e_{3}e_{2}=e_{4}$ \\[4pt]
$\mathcal{A}_{30}$ & $e_{1}e_{1}=e_{2},\ e_{1}e_{2}=e_{3},\ e_{1}e_{3}=e_{5},\ e_{1}e_{4}=e_{5}, \ e_{2}e_{1}=e_{3},\ e_{2}e_{2}=e_{5},\ e_{3}e_{1}=e_{5},\ e_{4}e_{4}=e_{5}$ \\[4pt]
$\mathcal{A}_{31}$ & $e_{1}e_{1}=e_{2},\ e_{1}e_{2}=e_{3},\ e_{1}e_{3}=e_{5},\ e_{1}e_{4}=e_{5}, \ e_{2}e_{1}=e_{3},\ e_{2}e_{2}=e_{5},\ e_{3}e_{1}=e_{5}$ \\[4pt]
$\mathcal{A}_{32}$ & $e_{1}e_{1}=e_{2},\ e_{1}e_{2}=e_{3},\ e_{1}e_{3}=e_{5},\ e_{2}e_{1}=e_{3}, \ e_{2}e_{2}=e_{5},\ e_{3}e_{1}=e_{5},\ e_{4}e_{4}=e_{5}$ \\[4pt]
$\mathcal{A}_{33}(\alpha)$ & $e_{1}e_{1}=e_{2},\ e_{1}e_{2}=e_{4},\ e_{1}e_{4}=e_{5},\ e_{2}e_{1}=e_{4}, \ e_{2}e_{2}=e_{5},\ e_{2}e_{3}=e_{5},\ e_{3}e_{1}=e_{4},\ e_{3}e_{2}=e_{5}$ \\ & $e_{3}e_{3}=\alpha e_{5},\ e_{4}e_{1}=e_{5}, \ \alpha\in\mathbb{C}$ \\[4pt]
$\mathcal{A}_{34}$ & $e_{1}e_{2}=e_{3},\ e_{1}e_{3}=e_{4},\ e_{2}e_{1}=-e_{3},\ e_{2}e_{4}=e_{5}, \ e_{3}e_{3}=-e_{5},\ e_{4}e_{2}=e_{5}$ \\[4pt]
$\mathcal{A}_{35}$ & $e_{1}e_{1}=e_{5},\ e_{1}e_{2}=e_{3},\ e_{1}e_{3}=e_{4},\ e_{2}e_{1}=-e_{3}, \ e_{2}e_{4}=e_{5},\ e_{3}e_{3}=-e_{5},\ e_{4}e_{2}=e_{5}$ \\[4pt]
$\mathcal{A}_{36}$ & $e_{1}e_{2}=e_{3},\ e_{1}e_{3}=e_{4},\ e_{2}e_{1}=-e_{3},\ e_{2}e_{3}=e_{5}, \ e_{2}e_{4}=e_{5},\ e_{3}e_{3}=-e_{5},\ e_{4}e_{2}=e_{5}$ \\[4pt]
$\mathcal{A}_{37}$ & $e_{1}e_{1}=e_{5},\ e_{1}e_{2}=e_{3},\ e_{1}e_{3}=e_{4},\ e_{2}e_{1}=-e_{3}, \ e_{2}e_{3}=e_{5},\ e_{2}e_{4}=e_{5},\ e_{3}e_{3}=-e_{5},\ e_{4}e_{2}=e_{5}$ \\[4pt]
$\mathcal{A}_{38}$ & $e_{1}e_{2}=e_{3},\ e_{1}e_{3}=e_{4},\ e_{1}e_{4}=e_{5},\ e_{2}e_{1}=-e_{3}, \ e_{4}e_{1}=e_{5}$ \\[4pt]
$\mathcal{A}_{39}$ & $e_{1}e_{2}=e_{3},\ e_{1}e_{3}=e_{4},\ e_{1}e_{4}=e_{5},\ e_{2}e_{1}=-e_{3}, \ e_{2}e_{2}=e_{5},\ e_{4}e_{1}=e_{5}$ \\[4pt]
$\mathcal{A}_{40}$ & $e_{1}e_{2}=e_{3},\ e_{1}e_{3}=e_{4},\ e_{1}e_{4}=e_{5},\ e_{2}e_{1}=-e_{3}, \ e_{2}e_{3}=e_{5},\ e_{4}e_{1}=e_{5}$ \\[4pt]
$\mathcal{A}_{41}$ & $e_{1}e_{2}=e_{3},\ e_{1}e_{3}=e_{4},\ e_{1}e_{4}=e_{5},\ e_{2}e_{1}=-e_{3}, \ e_{2}e_{2}=e_{5},\ e_{2}e_{3}=e_{5},\ e_{4}e_{1}=e_{5}$ \\[4pt]
$\mathcal{A}_{42}$ & $e_{1}e_{2}=e_{3},\ e_{1}e_{3}=e_{4},\ e_{2}e_{1}=-e_{3},\ e_{2}e_{2}=e_{4}, \ e_{2}e_{4}=e_{5},\ e_{3}e_{3}=-e_{5},\ e_{4}e_{2}=e_{5}$ \\[4pt]
$\mathcal{A}_{43}$ & $e_{1}e_{1}=e_{5},\ e_{1}e_{2}=e_{3},\ e_{1}e_{3}=e_{4},\ e_{2}e_{1}=-e_{3}, \ e_{2}e_{2}=e_{4},\ e_{2}e_{4}=e_{5},\ e_{3}e_{3}=-e_{5},\ e_{4}e_{2}=e_{5}$ \\[4pt]
$\mathcal{A}_{44}(\alpha)$ & $e_{1}e_{1}=\alpha e_{5},\ e_{1}e_{2}=e_{3},\ e_{1}e_{3}=e_{4},\ e_{2}e_{1}=-e_{3}, \ e_{2}e_{2}=e_{4},\ e_{2}e_{3}=e_{5}$\\&  $e_{2}e_{4}=e_{5},\ e_{3}e_{3}=-e_{5}, \ e_{4}e_{2}=e_{5}, \ \alpha\in\mathbb{C}$ \\[4pt]
$\mathcal{A}_{45}(\alpha)$ & $e_{1}e_{2}=e_{3},\ e_{1}e_{3}=e_{4},\ e_{1}e_{4}=e_{5},\ e_{2}e_{1}=-e_{3}, \ e_{2}e_{2}=e_{4}$ \\ & $ e_{2}e_{3}=\alpha e_{5},\ e_{3}e_{2}=e_{5},\ e_{4}e_{1}=e_{5}, \ \alpha\in\mathbb{C}$ \\[4pt]
$\mathcal{A}_{46}$ & $e_{1}e_{2}=e_{3},\ e_{1}e_{3}=e_{4}+3 e_{5},\ e_{1}e_{4}=3 e_{5},\ e_{2}e_{1}=-e_{3}$ \\ & $e_{2}e_{2}=e_{4},\ e_{2}e_{3}=-e_{5},\ e_{3}e_{2}=3 e_{5},\ e_{4}e_{1}=3 e_{5}$ \\[4pt]
$\mathcal{A}_{47}$ & $e_{1}e_{1}=e_{2},\ e_{1}e_{2}=e_{3},\ e_{1}e_{3}=e_{4},\ e_{1}e_{4}=e_{5}, \ e_{2}e_{1}=e_{3}$ \\ & $e_{2}e_{2}=e_{4},\ e_{2}e_{3}=e_{5},\ e_{3}e_{1}=e_{4}, \ e_{3}e_{2}=e_{5},\ e_{4}e_{1}=e_{5}$ \\[4pt]
$\mathcal{A}_{48}(\alpha)$ & $e_{1}e_{1}=e_{2},\ e_{1}e_{2}=e_{5},\ e_{2}e_{1}=e_{5},\ e_{3}e_{3}=e_{5}, \ e_{3}e_{4}=e_{5},\ e_{4}e_{3}=-e_{5},\ e_{4}e_{4}=\alpha e_{5}, \ \alpha\in\mathbb{C}$ \\[4pt]
$\mathcal{A}_{49}$ & $e_{1}e_{1}=e_{2},\ e_{1}e_{2}=e_{5},\ e_{2}e_{1}=e_{5},\ e_{3}e_{4}=e_{5}, \ e_{4}e_{3}=-e_{5}$ \\[4pt]
$\mathcal{A}_{50}$ & $e_{1}e_{1}=e_{2},\ e_{1}e_{2}=e_{5},\ e_{2}e_{1}=e_{5},\ e_{3}e_{4}=e_{5}, \ e_{4}e_{3}=e_{5}$ \\[4pt]
$\mathcal{A}_{51}$ & $e_{1}e_{1}=e_{2},\ e_{1}e_{2}=e_{5},\ e_{2}e_{1}=e_{5},\ e_{3}e_{1}=e_{5}, \ e_{3}e_{4}=e_{5},\ e_{4}e_{3}=e_{5}$ \\[4pt]
$\mathcal{A}_{52}$ & $e_{1}e_{1}=e_{2},\ e_{1}e_{2}=e_{5},\ e_{2}e_{1}=e_{5},\ e_{3}e_{1}=e_{5}, \ e_{3}e_{4}=e_{5},\ e_{4}e_{1}=e_{5},\ e_{4}e_{3}=e_{5}$ \\[4pt]
$\mathcal{A}_{53}$ & $e_{1}e_{1}=e_{2},\ e_{1}e_{2}=e_{5},\ e_{2}e_{1}=e_{5},\ e_{3}e_{3}=e_{5}, \ e_{4}e_{1}=e_{5}$ \\[4pt]
$\mathcal{A}_{54}$ & $e_{1}e_{1}=e_{2},\ e_{1}e_{2}=e_{5},\ e_{2}e_{1}=e_{5},\ e_{3}e_{3}=e_{4}, \ e_{3}e_{4}=e_{5},\ e_{4}e_{3}=e_{5}$ \\[4pt]
$\mathcal{A}_{55}$ & $e_{1}e_{1}=e_{2},\ e_{1}e_{2}=e_{5},\ e_{2}e_{1}=e_{5},\ e_{3}e_{1}=e_{5}, \ e_{3}e_{3}=e_{4},\ e_{3}e_{4}=e_{5},\ e_{4}e_{3}=e_{5}$ \\[4pt]
$\mathcal{A}_{56}$ & $e_{1}e_{2}=e_{3},\ e_{2}e_{1}=-e_{3},\ e_{4}e_{3}=e_{5}$ \\[4pt]
$\mathcal{A}_{57}$ & $e_{1}e_{1}=e_{5},\ e_{1}e_{2}=e_{3},\ e_{2}e_{1}=-e_{3},\ e_{4}e_{3}=e_{5}$ \\[4pt]
$\mathcal{A}_{58}$ & $e_{1}e_{2}=e_{3}+e_{5},\ e_{2}e_{1}=-e_{3},\ e_{4}e_{3}=e_{5}$ \\[4pt]
$\mathcal{A}_{59}$ & $e_{1}e_{2}=e_{3},\ e_{1}e_{4}=e_{5},\ e_{2}e_{1}=-e_{3},\ e_{4}e_{3}=e_{5}$ \\[4pt]
$\mathcal{A}_{60}$ & $e_{1}e_{2}=e_{3},\ e_{1}e_{4}=e_{5},\ e_{2}e_{1}=-e_{3},\ e_{2}e_{2}=e_{5}, \ e_{4}e_{3}=e_{5}$ \\[4pt]
$\mathcal{A}_{61}$ & $e_{1}e_{1}=e_{5},\ e_{1}e_{2}=e_{3},\ e_{1}e_{4}=e_{5},\ e_{2}e_{1}=-e_{3}, \ e_{2}e_{2}=e_{5},\ e_{4}e_{3}=e_{5}$ \\[4pt]
$\mathcal{A}_{62}$ & $e_{1}e_{2}=e_{3}+e_{5},\ e_{1}e_{4}=e_{5},\ e_{2}e_{1}=-e_{3},\ e_{4}e_{3}=e_{5}$ \\[4pt]
$\mathcal{A}_{63}$ & $e_{1}e_{1}=e_{5},\ e_{1}e_{2}=e_{3},\ e_{1}e_{4}=e_{5},\ e_{2}e_{1}=-e_{3}, \ e_{4}e_{3}=e_{5}$ \\[4pt]
$\mathcal{A}_{64}(\alpha)$ & $e_{1}e_{2}=e_{3},\ e_{1}e_{3}=e_{5},\ e_{2}e_{1}=-e_{3},\ e_{2}e_{2}=\alpha e_{5}, \ e_{2}e_{4}=e_{5},\ e_{4}e_{4}=e_{5}, \ \alpha\in\mathbb{C}$ \\[4pt]
$\mathcal{A}_{65}$ & $e_{1}e_{2}=e_{3},\ e_{1}e_{3}=e_{5},\ e_{2}e_{1}=-e_{3},\ e_{2}e_{2}=e_{5}, \ e_{4}e_{4}=e_{5}$ \\[4pt]
$\mathcal{A}_{66}$ & $e_{1}e_{2}=e_{3},\ e_{1}e_{3}=e_{5},\ e_{2}e_{1}=-e_{3},\ e_{4}e_{4}=e_{5}$ \\[4pt]
$\mathcal{A}_{67}(\alpha)$ & $e_{1}e_{2}=e_{3},\ e_{1}e_{3}=e_{5},\ e_{2}e_{1}=-e_{3},\ e_{2}e_{4}=\alpha e_{5}, \ e_{4}e_{2}=e_{5}, \ \alpha\in\mathbb{C}$ \\[4pt]
$\mathcal{A}_{68}$ & $e_{1}e_{2}=e_{3},\ e_{1}e_{3}=e_{5},\ e_{2}e_{1}=-e_{3},\ e_{2}e_{2}=e_{5}, \ e_{2}e_{4}=-e_{5},\ e_{4}e_{2}=e_{5}$ \\[4pt]
$\mathcal{A}_{69}$ & $e_{1}e_{2}=e_{3},\ e_{1}e_{3}=e_{5},\ e_{2}e_{1}=-e_{3},\ e_{2}e_{4}=e_{5}$ \\[4pt]
$\mathcal{A}_{70}$ & $e_{1}e_{2}=e_{3},\ e_{1}e_{3}=e_{5},\ e_{2}e_{1}=-e_{3},\ e_{2}e_{4}=e_{5}, \ e_{4}e_{1}=e_{5}$ \\[4pt]
$\mathcal{A}_{71}$ & $e_{1}e_{2}=e_{3},\ e_{1}e_{3}=e_{5},\ e_{2}e_{1}=-e_{3},\ e_{4}e_{1}=e_{5}$ \\[4pt]
$\mathcal{A}_{72}$ & $e_{1}e_{2}=e_{3},\ e_{1}e_{3}=e_{5},\ e_{2}e_{1}=-e_{3},\ e_{2}e_{2}=e_{5}, \ e_{4}e_{1}=e_{5}$ \\[4pt]
$\mathcal{A}_{73}(\alpha)$ & $e_{1}e_{1}=e_{3},\ e_{1}e_{2}=e_{3},\ e_{1}e_{3}=e_{5},\ e_{2}e_{2}=\alpha e_{5}, \ e_{2}e_{4}=e_{5},\ e_{3}e_{1}=e_{5},\ e_{4}e_{4}=e_{5}, \ \alpha\in\mathbb{C}$ \\[4pt]
$\mathcal{A}_{74}(\alpha)$ & $e_{1}e_{1}=e_{3},\ e_{1}e_{2}=e_{3},\ e_{1}e_{3}=e_{5},\ e_{2}e_{2}=\alpha e_{5}$ \\ & $e_{2}e_{4}=e_{5},\ e_{3}e_{1}=e_{5},\ e_{4}e_{1}=-e_{5},\ e_{4}e_{4}=e_{5}, \ \alpha\in\mathbb{C}$ \\[4pt]
$\mathcal{A}_{75}$ & $e_{1}e_{1}=e_{3},\ e_{1}e_{2}=e_{3},\ e_{1}e_{3}=e_{5},\ e_{3}e_{1}=e_{5}, \ e_{4}e_{4}=e_{5}$ \\[4pt]
$\mathcal{A}_{76}$ & $e_{1}e_{1}=e_{3},\ e_{1}e_{2}=e_{3},\ e_{1}e_{3}=e_{5},\ e_{3}e_{1}=e_{5}, \ e_{4}e_{1}=e_{5},\ e_{4}e_{4}=e_{5}$ \\[4pt]
$\mathcal{A}_{77}$ & $e_{1}e_{1}=e_{3},\ e_{1}e_{2}=e_{3},\ e_{1}e_{3}=e_{5},\ e_{2}e_{2}=e_{5}, \ e_{3}e_{1}=e_{5},\ e_{4}e_{4}=e_{5}$ \\[4pt]
$\mathcal{A}_{78}$ & $e_{1}e_{1}=e_{3},\ e_{1}e_{2}=e_{3},\ e_{1}e_{3}=e_{5},\ e_{2}e_{2}=e_{5}, \ e_{3}e_{1}=e_{5},\ e_{4}e_{1}=e_{5},\ e_{4}e_{4}=e_{5}$ \\[4pt]
$\mathcal{A}_{79}$ & $e_{1}e_{1}=e_{3},\ e_{1}e_{2}=e_{3},\ e_{1}e_{3}=e_{5},\ e_{3}e_{1}=e_{5}, \ e_{4}e_{1}=e_{5}$ \\[4pt]
$\mathcal{A}_{80}$ & $e_{1}e_{1}=e_{3},\ e_{1}e_{2}=e_{3},\ e_{1}e_{3}=e_{5},\ e_{2}e_{2}=e_{5}, \ e_{3}e_{1}=e_{5},\ e_{4}e_{1}=e_{5}$ \\[4pt]
$\mathcal{A}_{81}(\alpha)$ & $e_{1}e_{1}=e_{3},\ e_{1}e_{2}=e_{3},\ e_{1}e_{3}=e_{5},\ e_{2}e_{4}=e_{5}, \ e_{3}e_{1}=e_{5},\ e_{4}e_{2}=\alpha e_{5}, \ \alpha\in\mathbb{C}$ \\[4pt]
$\mathcal{A}_{82}(\alpha)$ & $e_{1}e_{1}=e_{3},\ e_{1}e_{2}=e_{3},\ e_{1}e_{3}=e_{5},\ e_{2}e_{4}=e_{5}, \ e_{3}e_{1}=e_{5},\ e_{4}e_{1}=(\alpha - 1) e_{5},\ e_{4}e_{2}=\alpha e_{5}, \ \alpha\neq 1$ \\[4pt]
$\mathcal{A}_{83}$ & $e_{1}e_{1}=e_{3},\ e_{1}e_{2}=e_{3},\ e_{1}e_{3}=e_{5},\ e_{2}e_{4}=e_{5}, \ e_{3}e_{1}=e_{5},\ e_{4}e_{1}=e_{5},\ e_{4}e_{2}=e_{5}$ \\[4pt]
$\mathcal{A}_{84}$ & $e_{1}e_{1}=e_{3},\ e_{1}e_{2}=e_{3},\ e_{1}e_{3}=e_{5},\ e_{2}e_{2}=e_{5}, \ e_{2}e_{4}=e_{5},\ e_{3}e_{1}=e_{5},\ e_{4}e_{2}=-e_{5}$ \\[4pt]
$\mathcal{A}_{85}$ & $e_{1}e_{1}=e_{3},\ e_{1}e_{2}=e_{3},\ e_{1}e_{3}=e_{5},\ e_{2}e_{2}=e_{5}, \ e_{2}e_{4}=e_{5},\ e_{3}e_{1}=e_{5},\ e_{4}e_{1}=-2 e_{5},\ e_{4}e_{2}=-e_{5}$ \\[4pt]
$\mathcal{A}_{86}$ & $e_{1}e_{1}=e_{3},\ e_{1}e_{2}=e_{3},\ e_{1}e_{3}=e_{5},\ e_{3}e_{1}=e_{5}, \ e_{4}e_{2}=e_{5}$ \\[4pt]
$\mathcal{A}_{87}$ & $e_{1}e_{1}=e_{3},\ e_{1}e_{2}=e_{3},\ e_{1}e_{3}=e_{5},\ e_{3}e_{1}=e_{5}, \ e_{4}e_{1}=e_{5},\ e_{4}e_{2}=e_{5}$ \\[4pt]
$\mathcal{A}_{88}(\alpha,\beta)$ & $e_{1}e_{2}=e_{3},\ e_{1}e_{3}=\alpha e_{5},\ e_{1}e_{4}=e_{5},\ e_{2}e_{1}=e_{4}, \ e_{2}e_{2}=-e_{3},\ e_{2}e_{3}=- \alpha e_{5},\ e_{2}e_{4}=\beta e_{5}$ \\ & $e_{3}e_{1}=(\alpha + 1) e_{5}, \ e_{3}e_{2}=- \alpha e_{5},\ e_{4}e_{2}=(\beta + 1) e_{5}, \ (\alpha,\beta)\neq(0,-1)$ \\[4pt]
$\mathcal{A}_{89}(\alpha)$ & $e_{1}e_{2}=e_{3},\ e_{1}e_{3}=e_{5},\ e_{2}e_{1}=e_{4},\ e_{2}e_{2}=-e_{3}, \ e_{2}e_{3}=-e_{5}$ \\ & $e_{2}e_{4}=\alpha e_{5},\ e_{3}e_{1}=e_{5},\ e_{3}e_{2}=-e_{5}, \ e_{4}e_{2}=\alpha e_{5}, \ \alpha\neq 0$ \\[4pt]
$\mathcal{A}_{90}$ & $e_{1}e_{2}=e_{3},\ e_{1}e_{4}=e_{5},\ e_{2}e_{1}=e_{4},\ e_{2}e_{2}=-e_{3}, \ e_{2}e_{4}=-e_{5},\ e_{3}e_{1}=e_{5}$ \\[4pt]
$\mathcal{A}_{91}$ & $e_{1}e_{2}=e_{3}+e_{5},\ e_{1}e_{4}=e_{5},\ e_{2}e_{1}=e_{4},\ e_{2}e_{2}=-e_{3}, \ e_{2}e_{4}=-e_{5},\ e_{3}e_{1}=e_{5}$ \\[4pt]
$\mathcal{A}_{92}(\alpha,\gamma)$ & $e_{1}e_{1}=e_{3},\ e_{1}e_{2}=e_{4},\ e_{1}e_{3}=e_{5},\ e_{1}e_{4}=\alpha e_{5}, \ e_{2}e_{1}=- \gamma e_{3},\ e_{2}e_{2}=-e_{4},$ \\ & $e_{2}e_{3}=- \gamma e_{5},\ e_{2}e_{4}=- \alpha e_{5}, \ e_{3}e_{1}=e_{5},\ e_{3}e_{2}=- \gamma e_{5},\ e_{4}e_{1}=\alpha e_{5},\ e_{4}e_{2}=- \alpha e_{5}$ \\ & \quad {\small $\alpha\neq0,\ \gamma\neq1;\ \mathcal{A}_{92}(\alpha,\gamma)\cong\mathcal{A}_{92}(\gamma^{3}\alpha^{-1},\gamma)$} \\[4pt]
$\mathcal{A}_{93}(\alpha)$ & $e_{1}e_{1}=e_{3},\ e_{1}e_{2}=e_{4},\ e_{1}e_{4}=\alpha e_{5},\ e_{2}e_{1}=-e_{3}, \ e_{2}e_{2}=-e_{4}$ \\ & $e_{2}e_{4}=- \alpha e_{5},\ e_{3}e_{2}=e_{5},\ e_{4}e_{1}=(\alpha - 1) e_{5}, \ e_{4}e_{2}=- \alpha e_{5}, \ \alpha\in\mathbb{C}$ \\[4pt]
$\mathcal{A}_{94}$ & $e_{1}e_{1}=e_{3}+e_{5},\ e_{1}e_{2}=e_{4},\ e_{1}e_{4}=e_{5},\ e_{2}e_{1}=-e_{3}$ \\ & $e_{2}e_{2}=-e_{4},\ e_{2}e_{4}=-e_{5},\ e_{3}e_{2}=e_{5},\ e_{4}e_{2}=-e_{5}$ \\[4pt]
$\mathcal{A}_{95}$ & $e_{1}e_{1}=e_{3},\ e_{1}e_{2}=e_{4},\ e_{1}e_{3}=e_{5},\ e_{1}e_{4}=-e_{5}, \ e_{2}e_{1}=-e_{3},\ e_{2}e_{2}=-e_{4}$ \\ & $e_{2}e_{3}=-e_{5},\ e_{2}e_{4}=e_{5}, \ e_{3}e_{1}=e_{5},\ e_{4}e_{1}=-2 e_{5},\ e_{4}e_{2}=e_{5}$ \\[4pt]
$\mathcal{A}_{96}(\alpha,\beta)$ & $e_{1}e_{2}=e_{3},\ e_{1}e_{3}=\alpha e_{5},\ e_{1}e_{4}=e_{5},\ e_{2}e_{1}=e_{4}$ \\ & $e_{2}e_{3}=e_{5},\ e_{2}e_{4}=\beta e_{5},\ e_{3}e_{1}=(\alpha + 1) e_{5},\ e_{4}e_{2}=(\beta + 1) e_{5}$ \\ & \quad {\small $(\alpha,\beta)\neq(-1,-1);\ \mathcal{A}_{96}(\alpha,\beta)\cong\mathcal{A}_{96}(\beta,\alpha)$} \\[4pt]
$\mathcal{A}_{97}(\alpha)$ & $e_{1}e_{2}=e_{3},\ e_{1}e_{3}=\alpha e_{5},\ e_{1}e_{4}=e_{5},\ e_{2}e_{1}=e_{4}, \ e_{2}e_{4}=e_{5},\ e_{3}e_{1}=(\alpha + 1) e_{5},\ e_{4}e_{2}=e_{5}, \ \alpha\in\mathbb{C}$ \\[4pt]
$\mathcal{A}_{98}(\alpha)$ & $e_{1}e_{2}=e_{3},\ e_{1}e_{3}=\alpha e_{5},\ e_{1}e_{4}=e_{5},\ e_{2}e_{1}=e_{4}, \ e_{3}e_{1}=(\alpha + 1) e_{5}, \ \alpha\neq-1$ \\[4pt]
$\mathcal{A}_{99}(\alpha)$ & $e_{1}e_{2}=e_{3},\ e_{1}e_{3}=\alpha e_{5},\ e_{1}e_{4}=e_{5},\ e_{2}e_{1}=e_{4}, \ e_{2}e_{2}=e_{5},\ e_{3}e_{1}=(\alpha + 1) e_{5}, \ \alpha\neq-1$ \\[4pt]
$\mathcal{A}_{100}$ & $e_{1}e_{2}=e_{3},\ e_{1}e_{3}=e_{5},\ e_{2}e_{1}=e_{4},\ e_{2}e_{4}=e_{5}, \ e_{3}e_{1}=e_{5},\ e_{4}e_{2}=e_{5}$ \\[4pt]
$\mathcal{A}_{101}(\alpha)$ & $e_{1}e_{1}=e_{4},\ e_{1}e_{2}=e_{3},\ e_{1}e_{3}=e_{5},\ e_{1}e_{4}=\alpha e_{5}, \ e_{2}e_{1}=-e_{3},\ e_{2}e_{2}=2 e_{3}+e_{4}$ \\ & $e_{2}e_{3}=(\alpha + 2) e_{5},\ e_{2}e_{4}=-e_{5}, \ e_{3}e_{1}=e_{5},\ e_{3}e_{2}=(\alpha + 2) e_{5},\ e_{4}e_{1}=\alpha e_{5},\ e_{4}e_{2}=-e_{5}$ \\ & \quad {\small $\alpha\neq-1;\ \mathcal{A}_{101}(\alpha)\cong\mathcal{A}_{101}\big(-\tfrac{2\alpha+3}{\alpha+2}\big)$} \\[4pt]
$\mathcal{A}_{102}(\alpha)$ & $e_{2}e_{1}=e_{4},\ e_{2}e_{2}=e_{3},\ e_{2}e_{4}=\alpha e_{5},\ e_{3}e_{1}=e_{5}, \ e_{4}e_{2}=(\alpha - 1) e_{5}, \ \alpha\in\mathbb{C}$ \\[4pt]
$\mathcal{A}_{103}$ & $e_{2}e_{1}=e_{4},\ e_{2}e_{2}=e_{3},\ e_{2}e_{3}=e_{5},\ e_{3}e_{1}=e_{5}, \ e_{3}e_{2}=e_{5},\ e_{4}e_{2}=-e_{5}$ \\[4pt]
$\mathcal{A}_{104}$ & $e_{1}e_{2}=e_{5},\ e_{2}e_{1}=e_{4},\ e_{2}e_{2}=e_{3},\ e_{2}e_{4}=e_{5}, \ e_{3}e_{1}=e_{5}$ \\[4pt]
$\mathcal{A}_{105}(\alpha,\lambda)$ & $e_{1}e_{2}=e_{4},\ e_{1}e_{3}=\alpha e_{5},\ e_{2}e_{1}=\lambda e_{4},\ e_{2}e_{2}=e_{3}, \ e_{2}e_{4}=e_{5}$ \\ & $e_{3}e_{1}=(- \alpha \lambda + \lambda + 1) e_{5},\ e_{4}e_{2}=\alpha e_{5}, \ \lambda\neq0,-1,\ \alpha\in\mathbb{C}$ \\[4pt]
$\mathcal{A}_{106}(\lambda)$ & $e_{1}e_{2}=e_{4},\ e_{1}e_{3}=e_{5},\ e_{2}e_{1}=\lambda e_{4},\ e_{2}e_{2}=e_{3}, \ e_{3}e_{1}=- \lambda e_{5},\ e_{4}e_{2}=e_{5}, \ \lambda\neq0,-1$ \\[4pt]
$\mathcal{A}_{107}(\lambda)$ & $e_{1}e_{1}=(\lambda - 1) e_{5},\ e_{1}e_{2}=e_{4},\ e_{1}e_{3}=(\lambda + 1) e_{5},\ e_{2}e_{1}=\lambda e_{4}, \ e_{2}e_{2}=e_{3}$ \\ & $e_{2}e_{4}=(\lambda - 1) e_{5},\ e_{3}e_{1}=(- \lambda - 1) e_{5},\ e_{4}e_{2}=(\lambda + 1) e_{5}, \ \lambda\neq0,\pm1$ \\[4pt]
$\mathcal{A}_{108}(\lambda)$ & $e_{1}e_{2}=e_{4},\ e_{1}e_{3}=(- \lambda - 1) e_{5},\ e_{2}e_{1}=\lambda e_{4},\ e_{2}e_{2}=e_{3}, \ e_{2}e_{3}=e_{5},\ e_{2}e_{4}=e_{5}$ \\ & $e_{3}e_{1}=(\lambda^{2} + 2 \lambda + 1) e_{5},\ e_{3}e_{2}=e_{5}, \ e_{4}e_{2}=(- \lambda - 1) e_{5}, \ \lambda\neq0,-1$ \\[4pt]
$\mathcal{A}_{109}$ & $e_{1}e_{1}=e_{5},\ e_{1}e_{2}=e_{4},\ e_{1}e_{3}=e_{5},\ e_{2}e_{1}=e_{4}, \ e_{2}e_{2}=e_{3},\ e_{3}e_{1}=-e_{5},\ e_{4}e_{2}=e_{5}$ \\ & \quad {\small $\text{extension of }\mathcal{M}_{14}(1)$} \\[4pt]
$\mathcal{A}_{110}$ & $e_{1}e_{2}=e_{4}+e_{5},\ e_{1}e_{3}=e_{5},\ e_{2}e_{1}=e_{4},\ e_{2}e_{2}=e_{3}, \ e_{2}e_{4}=e_{5},\ e_{3}e_{1}=e_{5},\ e_{4}e_{2}=e_{5}$ \\ & \quad {\small $\text{extension of }\mathcal{M}_{14}(1)$} \\[4pt]
$\mathcal{A}_{111}(\alpha)$ & $e_{1}e_{2}=e_{4},\ e_{1}e_{4}=e_{5},\ e_{2}e_{2}=e_{3},\ e_{2}e_{3}=\alpha e_{5}$ \\ & $e_{2}e_{4}=e_{5},\ e_{3}e_{1}=e_{5},\ e_{3}e_{2}=\alpha e_{5},\ e_{4}e_{1}=e_{5}, \ \alpha\in\mathbb{C}$ \\[4pt]
$\mathcal{A}_{112}$ & $e_{1}e_{2}=e_{4},\ e_{1}e_{4}=e_{5},\ e_{2}e_{2}=e_{3},\ e_{2}e_{3}=e_{5}, \ e_{3}e_{2}=e_{5},\ e_{4}e_{1}=e_{5}$ \\[4pt]
$\mathcal{A}_{113}(\alpha)$ & $e_{1}e_{2}=e_{4},\ e_{1}e_{3}=\alpha e_{5},\ e_{2}e_{2}=e_{3},\ e_{2}e_{4}=e_{5}, \ e_{3}e_{1}=e_{5},\ e_{4}e_{2}=\alpha e_{5}, \ \alpha\neq-1$ \\[4pt]
$\mathcal{A}_{114}$ & $e_{1}e_{2}=e_{4},\ e_{1}e_{3}=e_{5},\ e_{2}e_{2}=e_{3},\ e_{4}e_{2}=e_{5}$ \\[4pt]
$\mathcal{A}_{115}$ & $e_{1}e_{2}=e_{4},\ e_{1}e_{3}=e_{5},\ e_{2}e_{1}=e_{5},\ e_{2}e_{2}=e_{3}, \ e_{4}e_{2}=e_{5}$ \\[4pt]
$\mathcal{A}_{116}$ & $e_{1}e_{1}=e_{5},\ e_{1}e_{2}=e_{4},\ e_{1}e_{3}=-e_{5},\ e_{2}e_{2}=e_{3}, \ e_{2}e_{3}=e_{5}$ \\ & $e_{2}e_{4}=e_{5},\ e_{3}e_{1}=e_{5},\ e_{3}e_{2}=e_{5}, \ e_{4}e_{2}=-e_{5}$ \\[4pt]
$\mathcal{A}_{117}$ & $e_{1}e_{1}=e_{5},\ e_{1}e_{2}=e_{4},\ e_{1}e_{3}=-e_{5},\ e_{2}e_{2}=e_{3}, \ e_{2}e_{4}=e_{5},\ e_{3}e_{1}=e_{5},\ e_{4}e_{2}=-e_{5}$ \\[4pt]
$\mathcal{A}_{118}$ & $e_{1}e_{2}=e_{4},\ e_{1}e_{3}=-e_{5},\ e_{2}e_{2}=e_{3},\ e_{2}e_{3}=e_{5}, \ e_{2}e_{4}=e_{5},\ e_{3}e_{1}=e_{5},\ e_{3}e_{2}=e_{5},\ e_{4}e_{2}=-e_{5}$ \\[4pt]
$\mathcal{A}_{119}$ & $e_{1}e_{2}=e_{4},\ e_{1}e_{3}=-e_{5},\ e_{2}e_{2}=e_{3},\ e_{2}e_{4}=e_{5}, \ e_{3}e_{1}=e_{5},\ e_{4}e_{2}=-e_{5}$ \\[4pt]
$\mathcal{A}_{120}(\alpha)$ & $e_{1}e_{2}=e_{4},\ e_{1}e_{3}=e_{5},\ e_{1}e_{4}=e_{5},\ e_{2}e_{1}=-e_{4}$ \\ & $e_{2}e_{2}=e_{3},\ e_{2}e_{4}=\alpha e_{5},\ e_{3}e_{1}=e_{5},\ e_{4}e_{2}=e_{5}, \ \alpha\in\mathbb{C}$ \\[4pt]
$\mathcal{A}_{121}(\alpha)$ & $e_{1}e_{2}=e_{4},\ e_{1}e_{3}=e_{5},\ e_{2}e_{1}=-e_{4},\ e_{2}e_{2}=e_{3}, \ e_{2}e_{4}=\alpha e_{5},\ e_{3}e_{1}=e_{5},\ e_{4}e_{2}=e_{5}, \ \alpha\in\mathbb{C}$ \\[4pt]
$\mathcal{A}_{122}$ & $e_{1}e_{2}=e_{4},\ e_{1}e_{4}=e_{5},\ e_{2}e_{1}=-e_{4},\ e_{2}e_{2}=e_{3}, \ e_{2}e_{3}=e_{5},\ e_{3}e_{2}=e_{5}$ \\[4pt]
$\mathcal{A}_{123}$ & $e_{1}e_{2}=e_{4},\ e_{2}e_{1}=-e_{4},\ e_{2}e_{2}=e_{3},\ e_{2}e_{3}=e_{5}, \ e_{2}e_{4}=e_{5},\ e_{3}e_{2}=e_{5}$ \\[4pt]
$\mathcal{A}_{124}$ & $e_{1}e_{1}=e_{5},\ e_{1}e_{2}=e_{4},\ e_{2}e_{1}=-e_{4},\ e_{2}e_{2}=e_{3}, \ e_{2}e_{3}=e_{5},\ e_{2}e_{4}=e_{5},\ e_{3}e_{2}=e_{5}$ \\[4pt]
\end{longtable}

\begin{proof}
By Lemma~\ref{lem:quotient}, $\mathcal{B}\cong\mathcal{A}_W$ with $\mathcal{A}=\mathcal{B}/\Ann(\mathcal{B})$ a nontrivial nilpotent right alternative algebra of dimension $5-s$, $s=\dim\Ann(\mathcal{B})\in\{1,2\}$, and $W\in\T_s(\mathcal{A})$; conversely, every such $\mathcal{A}_W$ satisfies the hypotheses of the theorem.
By Corollaries~\ref{cor:empty3} and~\ref{cor:empty4}, $\mathcal{A}$ is one of the algebras $\RR31$, $\RR33$, $\RR34(0)$ (if $s=2$) or $\RR41$, $\RR44$, $\RR45$, $\RR46$, $\RR49$, $\M1$, $\M2$, $\M3$, $\M4(0)$, $\M6$, $\M7(\gamma)$, $\M{11}$, $\M{12}$, $\M{13}$, $\M{14}(\lambda)$ (if $s=1$).
By Lemma~\ref{lem:orbits}, for fixed $\mathcal{A}$ the isomorphism classes of the algebras $\mathcal{A}_W$ correspond to the $\Aut(\mathcal{A})$-orbits on $\T_s(\mathcal{A})$.
These orbits were determined in Propositions~\ref{prop:R31}, \ref{prop:R33}, \ref{prop:R340}, \ref{prop:R41}, \ref{prop:R44}, \ref{prop:R45}, \ref{prop:R46}, \ref{prop:M1}, \ref{prop:M3}, \ref{prop:M40}, \ref{prop:M6}, \ref{prop:M71}, \ref{prop:M13}, \ref{prop:M14}, \ref{prop:M140}, \ref{prop:M14m1}, Lemmas~\ref{lem:M2}, \ref{lem:M7}, \ref{lem:M11}, \ref{lem:M12}, and in the subsection on $\RR49$; the algebras of the list are the corresponding extensions (the class $\theta$ used for each algebra is indicated in Sections~\ref{sec:ann2}--\ref{sec:twostep}).
Finally, algebras with non-isomorphic quotients $\mathcal{B}/\Ann(\mathcal{B})$, or with annihilators of different dimensions, are not isomorphic; the algebras $\RR{3}{i}$, $\RR{4}{i}$ and $\M{i}$ are pairwise non-isomorphic for different values of the indices and parameters by Theorems~\ref{thm:dim3}--\ref{thm:dim4nonpure}.
\end{proof}

\begin{corollary}\label{cor:count}
The list of Theorem~\ref{thm:main} consists of $93$ individual algebras, $27$ one-parameter families and $4$ two-parameter families, namely $\A{88}(\alpha,\beta)$, $\A{92}(\alpha,\gamma)$, $\A{96}(\alpha,\beta)$ and $\A{105}(\alpha,\lambda)$.
In particular, there are infinitely many isomorphism classes of five-dimensional nilpotent right alternative algebras without annihilator component that are not $2$-step nilpotent.
\end{corollary}

\begin{proof}
This is a direct count in the list of Theorem~\ref{thm:main}.
\end{proof}

\begin{remark}\label{rem:complete}
To obtain the complete list of five-dimensional nilpotent right alternative algebras one has to add to Theorem~\ref{thm:main}:
(a) the $2$-step nilpotent five-dimensional algebras, whose classification is a problem about subspaces of bilinear forms and does not involve the right alternative identity (Lemma~\ref{lem:2step});
(b) the algebras with an annihilator component, i.e.\ the algebras $\mathcal{B}'\oplus\C e_5$ with $\mathcal{B}'$ a four-dimensional nilpotent right alternative algebra.
\end{remark}

\end{document}